\documentclass[12pt, a4paper]{amsart}
\usepackage[utf8]{inputenc}
\usepackage[T1]{fontenc}
\usepackage[english]{babel}

\usepackage{enumerate, longtable}
\usepackage{amsmath, amscd, amsfonts, amsthm, amssymb, latexsym, comment, stmaryrd, graphicx}
\usepackage{xcolor}
\usepackage[all]{xy}
\usepackage[margin=2cm]{geometry}
\usepackage[
        colorlinks,
]{hyperref}

\usepackage{microtype}

\theoremstyle{definition}
\newtheorem{defi}{Definition}[section]
\theoremstyle{plain}
\newtheorem{prop}[defi]{Proposition}
\newtheorem*{prop*}{Proposition}
\newtheorem{lem}[defi]{Lemma}
\newtheorem{stel}[defi]{Theorem}
\newtheorem{gev}[defi]{Corollary}

\newtheorem{ques}[defi]{Question}
\newtheorem{fact}[defi]{Fact}
\newtheorem*{stel*}{Theorem}
\theoremstyle{remark}
\newtheorem{opm}[defi]{Remark}
\newtheorem{vb}[defi]{Example}

\numberwithin{equation}{section}

\newcommand{\ff}{\mathbb F}
\newcommand{\qq}{\mathbb Q}

\newcommand{\car}{{\rm char}}
\newcommand{\zz}{\mathbb Z}

\newcommand{\La}{\mathcal L}
\newcommand{\Lar}{\mathcal{L}_{\operatorname{ring}}}

\newcommand{\Erk}{{\rm rk}^{\scriptscriptstyle\exists}}
\newcommand{\Erkn}{{\rm rk}^{\scriptscriptstyle\exists,n}}
\newcommand{\Erkone}{{\rm rk}^{\scriptscriptstyle\exists,1}}
\newcommand{\qr}{\Erk}
\newcommand{\qrp}{{\rm rk}^{\scriptscriptstyle\exists^{\raisebox{-1pt}{\scalebox{0.5}{$+$}}}}}

\newcommand{\erk}{\Erk}
\newcommand{\perk}{\qrp}

\newcommand{\Tfields}{T_{\rm fields}}

\newcommand{\TKleq}{T_K^<}
\newcommand{\TKprec}{T_K^\prec}
\newcommand{\Tec}[1]{T_{#1}^{\scalebox{0.6}{$\prec_{\exists}$}}}

\newcommand{\Tprec}[1]{T_{#1}^\prec}
\newcommand{\TKec}{\Tec{K}}
\newcommand{\TKequiv}{T_K^{\equiv}}
\newcommand{\subTKleq}{{T_K^{\scalebox{0.5}{$<$}}}}
\newcommand{\subTKprec}{{T_K^{\scalebox{0.5}{$\prec$}}}}
\newcommand{\subTec}[1]{{T_{#1}^{\scalebox{0.4}{$\prec_{\exists}$}}}}

\newcommand{\subTKequiv}{{T_K^{\scalebox{0.5}{$\equiv$}}}}
\newcommand{\Ext}{T^<}

\newcommand{\pow}[1]{^{(#1)}}
\newcommand{\powx}[2]{#1^{\times(#2)}}

\newcommand{\pii}[2]{\wp_{#1}^{#2}}

\newcommand{\dotneq}{\not\doteq}

\DeclareMathOperator{\Spec}{Spec}

\DeclareMathOperator{\Frac}{Frac}

\DeclareMathOperator{\trdeg}{trdeg}
\DeclareMathOperator{\cdim}{cdim}
\DeclareMathOperator{\Gal}{Gal}

\DeclareMathOperator{\ed}{ed}
\DeclareMathOperator{\efd}{efd}

\DeclareMathOperator{\Natwithzero}{\mathbb{Z}_{\geq0}}
\DeclareMathOperator{\Natwithoutzero}{\mathbb{Z}_{>0}}
\DeclareMathOperator{\nat}{\Natwithzero}

\title[Existential rank and essential dimension of diophantine sets]
  {Existential rank and essential dimension\\of diophantine sets}
\author{Nicolas Daans}
\date{\today}
\address{Universiteit Antwerpen, Departement Wiskunde, Middelheimlaan 1, 2020 Antwerpen, Belgium}
\curraddr{Université de Mons, Département de Mathematique, Place du Parc 20, 7000 Mons, Belgium}
\email{nicolas.daans@umons.ac.be}
\author{Philip Dittmann}
\address{Technische Universit\"at Dresden, Fakult\"at Mathematik, Institut f\"ur Algebra, 01062 Dresden, Germany}
\curraddr{Department of Mathematics, University of Manchester, Manchester M13 9PL, United Kingdom}
\email{philip.dittmann@manchester.ac.uk}
\author{Arno Fehm}
\address{Technische Universit\"at Dresden, Fakult\"at Mathematik, Institut f\"ur Algebra, 01062 Dresden, Germany}
\email{arno.fehm@tu-dresden.de}

\begin{document}

\maketitle

\begin{center}
\em Dedicated to Alexander Prestel on the occasion of his 80th birthday
\end{center}

\begin{abstract}
We study the minimal number of existential quantifiers needed to define a diophantine set over a field and relate this number 
to the essential dimension of the functor of points associated to such a definition.
\end{abstract}

\section{Introduction}

\noindent
The notion of diophantine sets originates in the study of Hilbert's tenth problem for the integers,
and has subsequently been applied also to other rings and in particular to fields, see e.g.~\cite{Poonen_survey,Shlapentokh,Koenigsmann_survey}.
For a field $K$, a subset $D\subseteq K^n$ is {\em diophantine} if it has a diophantine definition, i.e. if
any of the following three equivalent conditions hold:
\begin{enumerate}
\item\label{diophantine1} $D=\{\underline{x}\in K^n: f_1(\underline{x},\underline{Y}),\dots,f_r(\underline{x},\underline{Y})\mbox{ have a common zero in $K^m$}\}$, for finitely many polynomials
$f_1,\dots,f_r\in K[X_1,\dots,X_n,Y_1,\dots,Y_m]$.
\item\label{diophantine2} $D=\pi(W(K))$ for a $K$-variety $W$ and a morphism $\pi\colon W\rightarrow\mathbb{A}_K^n$, where we identify $K^n=\mathbb{A}_K^n(K)$.
\item\label{diophantine3} $D=\varphi(K)$, the set defined in $K$ by an existential formula $\varphi$ in $n$ free variables in the first-order language of rings with parameters from $K$.
\end{enumerate}
One of the central questions regarding diophantine sets in fields is the following:
\begin{ques}\label{q:Q}
Is $\mathbb{Z}$ diophantine in $\mathbb{Q}$?
\end{ques}
A positive answer to this question would imply that 
the analogue of Hilbert's tenth problem for $\mathbb{Q}$ has a negative answer: There is no algorithm that decides whether a given variety over $\mathbb{Q}$ has a rational point.
In this direction, Koenigsmann \cite{Koenigsmann} showed that $\mathbb{Q}\setminus\mathbb{Z}$ is diophantine in $\mathbb{Q}$,
but there are conjectures in arithmetic geometry suggesting that $\mathbb{Z}$ should not be diophantine in $\mathbb{Q}$, see e.g.~\cite{Mazur}.
Also the analogous question 
in positive characteristic is of great interest:
\begin{ques}\label{q:FqT}
Is $\mathbb{F}_q[t]$ diophantine in $\mathbb{F}_q(t)$?
\end{ques}
The analogue of Koenigsmann's result for $\mathbb{F}_p(t)$ was proven in \cite{EisMor,Daans}.
For further recent results on diophantine sets in fields see for example
\cite{Kollar,AnscombeFehm,Dittmann}.

In this work we introduce and study several complexity measures of diophantine sets, related to the three descriptions of diophantine sets given above:

\begin{defi}\label{def:intro_erk}
Let $D \subseteq K^n$ be a diophantine set.
We define
 the
{\em positive-existential rank} $\qrp_K(D)$,
the {\em essential fibre dimension} $\efd_K(D)$,
and the 
{\em existential rank} $\Erk_K(D)$
 of $D$ as the following nonnegative integers:\footnote{For fully precise (and also more general) definitions see Definitions \ref{def:erk}, \ref{def:erk_K} and \ref{def:efd}, 
 and Remark~\ref{rem:perk_analogues}.
 Remark \ref{rem:perk} and Corollary \ref{cor:efdCharacterisationDiag} explain 
 why the following definitions of $\perk_K$ and $\efd_K$ are special cases of Definition \ref{def:erk_K} respectively \ref{def:efd}.}
\begin{eqnarray*}
\qrp_K(D)&=&\min\{m:\mbox{there exist $f_1,\dots,f_r$ as in (\ref{diophantine1}) in variables $X_1,\dots,X_n,Y_1,\dots,Y_m$}\}\\
\efd_K(D)&=&\min\{m:\mbox{there exists $\pi$ as in (\ref{diophantine2})
with ${\rm dim}(\pi^{-1}(x))\leq m$ for all $x\in\mathbb{A}_K^n$}\}\\
\Erk_K(D)&=&\min\{m:\mbox{there exists $\varphi$ as in (\ref{diophantine3}) with at most $m$ existential quantifiers}\}
\end{eqnarray*}
\end{defi}

The relation between these measures will be studied in Sections \ref{sect:efd} and \ref{sect:pthpowers}.
We find that $\Erk_K(D) = \qrp_K(D)$ unless the former is equal to 0 and the latter is equal to 1 (Corollary~\ref{cor:qrvsqrp}). 
The relation between $\efd_K(D)$ and $\erk_K(D)$ is
more complicated (see e.g.~Theorem~\ref{ed=qr}),
and although in this introduction we state results only for the existential rank, the interplay between $\efd_K(D)$ and $\erk_K(D)$ underlies many of the proofs.

One of the general results on existential rank that we obtain in this way
by making use of properties of essential fibre dimension is the following:

\begin{stel}[{Corollary \ref{cor:finGenOverPerf}}]\label{intro:finGenOverPerf}
Suppose that $K$ is finitely generated over a perfect field.
If $D_1,D_2\subseteq K^n$ are diophantine sets with
$\Erk_K(D_1)>0$, $\Erk_K(D_2)>0$, then
$$
 \Erk_K(D_1\cap D_2) \leq \Erk_K(D_1) + \Erk_K(D_2) -1.
$$
\end{stel}
This constitutes a saving of one quantifier over the trivial upper bound. 
The proof is particularly involved for fields of positive characteristic,
where we make use for example of arithmetic results from \cite{Jeong}.
In Section~\ref{sect:pthpowers} we also discuss examples where the trivial upper bound cannot be improved.

Additionally, as the name suggests, essential fibre dimension is connected 
to {\em essential dimension},
more precisely to the notion of essential dimension of a functor, as suggested by Merkurjev and introduced in \cite{BerhuyFavi_EssDim}, see Section \ref{sect:geometric}. 
We refer to \cite{Reichstein,Merkurjev_EDSurvey,Merkurjev2017} for surveys on essential dimension.
This relation between essential dimension and essential fibre dimension, and the interaction with existential rank, allows us to exploit known results on essential and canonical dimension of varieties.
For example, building on work by Karpenko, Merkurjev and Totaro on the essential dimension of quadrics \cite{KarpenkoMerkurjev_EssDimQuadrics,Totaro_BiratGeomQuadricsCharTwo} and Severi-Brauer varieties \cite{Karpenko_IncompressibilitySeveriBrauer}, we prove Theorem \ref{stel:erk_quad_cyclic_norm} on the existential rank of the set of elements represented by certain forms, a special case of which is the following:

\begin{stel}[{Corollary \ref{gev:sums_of_squares}}]
There exists a field $K$ of characteristic zero such that for every $m\in\Natwithoutzero$,
the set $\sum_{i=1}^mK\pow 2$ of sums of $m$ squares in $K$ has 
$\Erk_K(\sum_{i=1}^mK\pow 2)=m$.
\end{stel}

One hope behind our investigation is that apart from improving our understanding of diophantine sets, studying these complexity measures will eventually also lead to new insights into which sets are diophantine at all. For example, we 
also assign a complexity measure to the field itself and observe the following:

\begin{defi}
The {\em existential rank} of the field $K$ is 
$\Erk(K)=\sup_D\Erk_K(D)$
where $D$ runs over all diophantine subsets $D\subseteq K^n$ for all $n$.
\end{defi}

\begin{fact}[Corollary \ref{cor:erk_Z_in_Q}]\label{fact:er_Z_in_Q}
If $\Erk(\mathbb{Q})=\infty$
and there exists $f\in\mathbb{Z}[X,Y]$ such that $f\colon\mathbb{Q}\times\mathbb{Q}\rightarrow\mathbb{Q}$ is injective, 
then $\mathbb{Z}$ is not diophantine in $\mathbb{Q}$.
If $\Erk(\mathbb{F}_p(t))=\infty$, then $\mathbb{F}_p[t]$ is not diophantine in $\mathbb{F}_p(t)$.
\end{fact}

We have a complete picture for the existential rank of local fields and of pseudo-algebraically closed fields:

\begin{fact}[Corollary \ref{cor:RCFpCF}]
The local fields of characteristic zero have existential rank at most $1$,
in particular $\Erk(\mathbb{C})=0$, $\Erk(\mathbb{R})=1$, $\Erk(\mathbb{Q}_p)=1$.
\end{fact}

\begin{stel}[Proposition \ref{large}]
The local fields of positive characteristic have existential rank $\Erk(\mathbb{F}_q((t)))=\infty$.
\end{stel}

%\pagebreak

\begin{stel}[Corollary \ref{cor:erkPAC}]
Let $K$ be a pseudo-algebraically closed field. Then
\begin{enumerate}
\item $\erk(K)=0$ if $K$ is algebraically closed,
\item $\erk(K)=1$ if $K$ is perfect but not algebraically closed,
\item $\erk(K)=\infty$ if $K$ is not perfect.
\end{enumerate}
\end{stel}

On the other hand we get only very modest results for global fields,
where Hilbert's irreducibility theorem gives us the following:
\begin{stel}[Corollary \ref{cor:global}]
Every global field $K$ has $\Erk(K)\geq 2$.
\end{stel}
We think it would be of great interest to determine specifically $\Erk(\mathbb{Q})$ and $\Erk(\mathbb{F}_q(t))$,
in particular due to the connection with Questions \ref{q:Q} and \ref{q:FqT} via Fact \ref{fact:er_Z_in_Q}.
We refer to Section~\ref{sect:global} for some further discussion in this direction.

After this incomplete summary of concrete results
we should add that
in most of this manuscript we work in a much more general setting than presented in this introduction,
namely not in a single field $K$ but in the class of models of some first-order theory of fields.
In this set-up we can study a {\em uniform} version 
of existential rank and positive-existential rank (Sections \ref{sect:erk} and \ref{sect:erkfields}) and essential fibre dimension (Section \ref{sect:efd}) of a diophantine definition.
It is this general setting in which essential fibre dimension relates to essential dimension.

Our proofs are mostly algebraic and model theoretic, 
but use tools and results also from
algebraic and arithmetic geometry (Sections \ref{sect:pthpowers}, \ref{sect:geometric} and \ref{sect:global})
and number theory (Section~\ref{sect:global}).
The use of model theoretic methods makes some of our results inexplicit. For example, in proving Theorem~\ref{intro:finGenOverPerf} we do not write down an existential formula defining $D_1\cap D_2$ with the claimed number of quantifiers
but prove its mere existence.

\begin{opm}
In a separate publication of Karim Becher and N.~D., an approach to establishing upper bounds for the existential rank of diophantine sets via explicit methods will be discussed, providing a constructive proof of Theorem \ref{intro:finGenOverPerf} and related results in the case where the fields are perfect. Some of these results (Theorem \ref{intro:finGenOverPerf} in characteristic 0 and parts of Proposition \ref{prop:pinmperfect}) preceded the work in the current article.
See also \cite[Chapter 3]{Nicolas_thesis}.
\end{opm}

\begin{opm}
Independently from us, Hector Pasten in \cite{Pasten}
also introduces and studies a version of positive-existential rank.
In particular, he obtains a variant of Fact \ref{fact:er_Z_in_Q}
(more precisely, our Proposition \ref{prop:ZinQ1}),
and he deduces from \cite{Kollar} that,
in our definition,
$\erk(\mathbb{C}(t))=\infty$.
Also the recent \cite{EMSW} introduces a notion of rank of an existential formula in the language of rings,
which however is a purely syntactic notion and largely unrelated to our existential rank.
\end{opm}

\subsection*{Acknowledgements}

The authors would like to thank Hector Pasten
for the friendly exchange regarding \cite{Pasten}.

This work was part of N.~D.'s PhD thesis,
prepared under supervision of Karim Becher and P.~D. at the University of Antwerp. Several parts of the present manuscript may therefore also be found in \cite[Chapter 4]{Nicolas_thesis}.

Part of this work was done while
all three authors were participating in the Fall 2020 programme
{\em Decidability, definability and computability in number theory: Part 1 - Virtual Semester} at the Mathematical Sciences Research Institute in Berkeley, California,
and they would like to thank its organizers.
In particular,  P.~D.\ was a postdoctoral fellow of MSRI for the duration of the programme, and as such supported by the US National Science Foundation under Grant No.\ DMS-1928930.
A.~F.~was funded by the Deutsche Forschungsgemeinschaft (DFG) - 404427454.
N.~D. acknowledges funding by the FWO PhD Fellowship fundamental research 51581 and 83494 and by the FWO Odysseus programme (project \emph{Explicit Methods in Quadratic Form Theory}).

\section{Existential rank}\label{sect:preservation}\label{sect:erk}

\noindent
We start by introducing several notions of existential rank
in arbitrary languages.

We work with first-order languages $\La$ and structures, as considered in \cite[Section 2.1]{Hodges_Longer}.
As in \cite[Section 1.1]{Hodges_Longer}, $\La$-structures may be empty, although this can only happen if $\La$ does not contain any constant symbols, and this situation will not occur after the present section.
We allow the disjunction and conjunction over an empty set of formulas, abbreviated $\bot$ and $\top$, respectively.
These are positive quantifier-free sentences which are false, respectively true, in any structure.
See Remark \ref{rem:VerumFalsum} for a discussion on situations in which they can or cannot be dropped from the language.

We will consider in particular {\em existential $\La$-formulas},
i.e.~formulas
built up from quan\-ti\-fier-free formulas using $\wedge$, $\vee$ and $\exists$,
and {\em positive existential $\La$-formulas},
i.e.~formulas 
built up from atomic formulas using $\wedge$, $\vee$ and $\exists$.
Every existential formula is equivalent to a formula
of the form $\exists Y_1,\dots,Y_m\psi$ with $\psi$ quantifier-free,
and every positive existential formula is equivalent to a formula of that form with $\psi$
quantifier-free and not involving negation.
We shorten `existential $\La$-formula' to $\exists$-$\La$-formula 
and `existential $\La$-formula with $m$ quantifiers' to $\exists_m$-$\La$-formula. 
Similarly, an $\exists^+_m$-$\La$-formula is shorthand notation for `positive existential $\La$-formula with $m$ quantifiers'.

For a set $C$, we write $\La(C)$ for the language $\La$ extended with constant symbols for the elements of $C$. In particular, if $K$ is an $\La$-structure, we can consider $K$ to be an $\La(K)$-structure in a natural way.
For an $\La$-theory $\Sigma$
we call two $\La$-formulas $\varphi$, $\psi$ 
in free variables $X_1,\dots,X_n$ {\em equivalent modulo $\Sigma$} if $\Sigma\models\forall X_1,\dots,X_n(\varphi\leftrightarrow\psi)$.
For an $\La$-formula $\varphi$ in free variables $X_1,\dots,X_n$
and an $\La$-structure $K$ we write
$\varphi(K)=\{\underline{x}\in K^n:K\models\varphi(\underline{x})\}$ for the set defined by $\varphi$ in $K$.

So let $\La$ be a language and $\Sigma$ an $\La$-theory.

\begin{defi}\label{def:erk}
Let $\varphi$ be an $\La$-formula in free variables $X_1, \ldots, X_n$. 
We define the \emph{existential rank of $\varphi$ with respect to $\Sigma$ in $\La$} as 
\begin{displaymath}
\qr_{\La, \Sigma}(\varphi) = \inf \lbrace m \in \Natwithzero : \mbox{$\varphi$ is equivalent modulo $\Sigma$ to an $\exists_m$-$\La$-formula} \rbrace \in\Natwithzero\cup\{\infty\},\footnote{We write $\Natwithzero$ for the set of nonnegative integers and $\Natwithoutzero$ for the set of positive integers.}
\end{displaymath}
and the \emph{positive-existential rank of $\varphi$ with respect to $\Sigma$ in $\La$} as
\begin{displaymath}
\qrp_{\La, \Sigma}(\varphi) = \inf \lbrace m \in \Natwithzero : \mbox{$\varphi$ is equivalent modulo $\Sigma$ to an $\exists_m^+$-$\La$-formula} \rbrace \in\Natwithzero\cup\{\infty\}.
\end{displaymath}
When the language $\mathcal{L}$ is clear from the context, we will instead also write $\erk_\Sigma(\varphi)$ and $\perk_\Sigma(\varphi)$,
and similarly for the other definitions later in this section.
\end{defi}

\begin{opm}\label{qrtrivial}
Trivially, $\qr_{\mathcal{L},\Sigma}(\varphi)\leq\qrp_{\mathcal{L},\Sigma}(\varphi)$.
It is also easy to see that 
if $\varphi_1$ and $\varphi_2$ are $\La$-formulas, then
$$
 \qr_{\mathcal{L},\Sigma}(\varphi_1\vee\varphi_2)\leq
\max\{\qr_{\mathcal{L},\Sigma}(\varphi_1),\qr_{\mathcal{L},\Sigma}(\varphi_2)\}
$$
and 
$$
 \qr_{\mathcal{L},\Sigma}(\varphi_1\wedge\varphi_2)\leq
\qr_{\mathcal{L},\Sigma}(\varphi_1)+\qr_{\mathcal{L},\Sigma}(\varphi_2),
$$
and analogously for $\qrp_{\mathcal{L},\Sigma}$.
\end{opm}

\begin{opm}\label{remstatements}
Note that $\erk_{\La, \Sigma}(\varphi)$ and $\perk_{\La, \Sigma}(\varphi)$ depend
only on the equivalence class of $\varphi$ modulo $\Sigma$.
Moreover, we can always reinterpret an existential formula as an existential sentence in an extended language, namely by adding constant symbols for the free variables. This does not affect the existential or positive-existential rank. As such, we may now restrict our attention to existential sentences, as long as we keep the language general.
\end{opm}

\begin{vb}\label{ex:ACF}
The theory
 $\Sigma$ is model complete if and only if $\erk_{\La,\Sigma}(\varphi)<\infty$ for every $\mathcal{L}$-formula $\varphi$,
cf.~\cite[Exercise 3.4.12]{Marker}.
Moreover,
$\Sigma$ has quantifier elimination
if and only if
$\erk_{\La,\Sigma}(\varphi)=0$ for every $\mathcal{L}$-formula $\varphi$.
The latter applies for example to the theory $\Sigma={\rm ACF}$
of algebraically closed fields in the language 
$\mathcal{L}=\mathcal{L}_{\rm ring}$ of rings, see \cite[Theorem 3.2.2]{Marker}.
\end{vb}

Let $K$ be an $\La$-structure and $S \subseteq K$ a subset.
The \emph{$\La$-substructure generated by $S$ in $K$} is the smallest $\La$-substructure of $K$ containing $S$.
If $A$ and $B$ are substructures of $K$,
we say that $A$ is {\em generated by $S$ over $B$}
if $A$ is the substructure of $K$ generated by $B\cup S$.

If $\varphi$ is an existential $\La$-formula and $\underline{x}\in K^n$ for some $\La$-structure $K$,
a {\em witness} for $\varphi(\underline{x})$ is
any tuple $\underline{y}$ from $K$ such that 
the substructure $A$ generated by $\underline{x},\underline{y}$ in $K$ satisfies
$A\models\varphi(\underline{x})$.
If $\varphi$ is equivalent to $\exists Y_1,\dots,Y_m\psi$ with $\psi$ quantifier-free,
then every $\underline{y}\in K^m$ with $K\models\psi(\underline{x},\underline{y})$ is a witness for $\varphi(\underline{x})$.

We start by proving a model-theoretic result which may be considered a quantitative version of the Łoś–Tarski theorem \cite[Theorem 6.5.4]{Hodges_Longer}.

\begin{prop}\label{prop:generalcriterionhom}
For an $\La$-sentence $\varphi$ and $m \in \nat$, the following are equivalent:
\begin{enumerate}[(i)]
\item Every $K \models \Sigma \cup \lbrace \varphi \rbrace$ has an $\La$-substructure $A$ generated by $m$ elements such that every $L \models \Sigma$ for which there exists a homomorphism $A \to L$ satisfies $L \models \varphi$.
\item $\varphi$ is equivalent modulo $\Sigma$ to an
$\exists_m^+$-$\mathcal{L}$-sentence, i.e.~$\perk_{\La, \Sigma}(\varphi) \leq m$.
\end{enumerate}
\end{prop}

\begin{proof}
First assume $(i)$. 
Let $\mathcal{C}$ be the set of $\exists^+_m$-$\La$-sentences $\psi$ with $\Sigma\models\psi\rightarrow\varphi$.
%Note that $\F \in \mathcal{C}$, so $\mathcal{C} \neq \emptyset$.
We claim that for every $K \models \Sigma \cup \lbrace \varphi \rbrace$ there exists some $\psi \in \mathcal{C}$ with $K \models \psi$.
Indeed, by $(i)$, there exists a substructure $A$ of $K$ 
generated by $m$ elements $y_1, \ldots, y_m$
such that $L \models \varphi$ for any $L \models \Sigma$ for which there is a homomorphism $A \to L$.
Thus, if $\Delta$ denotes the positive quantifier-free $\mathcal{L}(Y_1,\dots,Y_m)$-theory of $A$ (where we interpret $Y_i$ as $y_i$),
then $\Sigma\cup\Delta\models\varphi$. %Note that $\T \in \Delta$, whence $\Delta \neq \emptyset$.
By the compactness theorem, there is a finite $\Delta_0 \subseteq \Delta$ %finite and non-empty
with $\Sigma\cup\Delta_0\models\varphi$.
Let $\psi$ denote the $\exists_m^+$-$\mathcal{L}$-sentence
$\exists Y_1,\dots,Y_m\delta$ where $\delta$ is the conjunction of all elements of $\Delta_0$.
Then $\psi\in\mathcal{C}$,
and $K\models\psi$ since $A\models\psi$,
which proves the claim.
Now applying the compactness theorem again, there is a finite %non-empty
subset $\mathcal{C}_0$ of $\mathcal{C}$ such that, for any $K \models \Sigma$, $K \models \varphi$ is equivalent to the existence of some $\psi \in \mathcal{C}_0$ with $K \models \psi$. As a finite disjunction of $\exists^+_m$-$\La$-sentences is logically equivalent to an $\exists^+_m$-$\La$-sentence (Remark \ref{qrtrivial}), this shows that $\varphi$ is equivalent modulo $\Sigma$ to an $\exists_m^+$-$\La$-sentence, 
i.e.~$(ii)$ holds.

Now assume $(ii)$. 
Then $\Sigma\models\varphi\leftrightarrow\psi$
for an $\exists^+_m$-$\La$-sentence $\psi$. 
Take $K \models \Sigma \cup \lbrace \varphi \rbrace$ and  a witness $(y_1, \ldots, y_m)$ for $\psi$. Letting $A$ be the substructure of $K$ generated by $y_1, \ldots, y_m$, we have $A \models \psi$. In particular, $L \models \psi$ 
and hence $L\models\varphi$ for every $L \models \Sigma$ for which there is a homomorphism $A \to L$, so we have verified $(i)$.
\end{proof}

\begin{gev}\label{cor:generalcriterionhom}
For an $\La$-formula $\varphi$,
$\perk_{\La, \Sigma}(\varphi)$ is the infimum over $m\in\Natwithzero$
such that for every $K\models\Sigma$ and every $\underline{x}\in\varphi(K)$
there is a substructure $A$ of $K$ generated by $\underline{x}$
and $m$ further elements
with $L\models\varphi(\rho(\underline{x}))$ for every homomorphism $\rho\colon A\rightarrow L$ into some $L\models\Sigma$.
\end{gev}

\begin{proof}
Remark \ref{remstatements} reduces the claim to Proposition \ref{prop:generalcriterionhom}.
\end{proof}

\begin{opm}\label{rem:VerumFalsum}
  It is in general necessary that we allow $\bot$ and $\top$ as positive quantifier-free sentences.
  Namely, in the theory of an infinite set with no structure but equality, according to Proposition \ref{prop:generalcriterionhom} (with $m=0$) the sentences $\exists X (X \dot{\neq} X)$ and $\exists X (X \doteq X)$ must be equivalent to positive quantifier-free sentences.
  
  If $\Sigma$ contains the theory of fields
  in the language of rings, which is the main setting we pursue, then the sentences $0 \doteq 1$ and $0 \doteq 0$ can take the places of $\bot$ and $\top$, so it does not actually matter whether we allow $\bot$ and $\top$ as positive quantifier-free sentences in their own right.
\end{opm}

\begin{prop}\label{prop:generalcriterion}
For an $\La$-sentence $\varphi$ and $m \in \nat$, the following are equivalent:
\begin{enumerate}[(i)]
\item Every $K \models \Sigma \cup \lbrace \varphi \rbrace$ has a substructure $A$ generated by $m$ elements such that every $L \models \Sigma$ into which $A$ embeds satisfies $L \models \varphi$.
\item $\varphi$ is equivalent modulo $\Sigma$ to an
$\exists_m$-$\mathcal{L}$-sentence, i.e.~$\erk_{\La, \Sigma}(\varphi) \leq m$.
\end{enumerate}
\end{prop}

\begin{proof}
This can be proven in exactly the same way as Proposition \ref{prop:generalcriterionhom},
where in the proof one needs to replace homomorphisms by embeddings, $\exists^+_m$-$\La$-sentences by $\exists_m$-$\La$-sentences, and positive quantifier-free theory by quantifier-free theory. 
Alternatively, Proposition~\ref{prop:generalcriterion} can be derived by applying Proposition \ref{prop:generalcriterionhom} in an extended language $\La'$ with added symbols for $\neq$ and for the negation of each relation symbol of $\La$. After adding the right (universal) axioms to the theory to describe the required behaviour of these new symbols, $\La$-embeddings correspond to $\La'$-homomorphisms, and quantifier-free $\La$-formulas correspond to positive quantifier-free $\La'$-formulas.
\end{proof}

\begin{gev}\label{cor:generalcriterion}
For an $\La$-formula $\varphi$,
$\erk_{\La, \Sigma}(\varphi)$ is the infimum over $m\in\Natwithzero$
such that for every $K\models\Sigma$ and every $\underline{x}\in\varphi(K)$
there is a substructure $A$ of $K$ generated by $\underline{x}$
and $m$ further elements
with $L\models\varphi(\rho(\underline{x}))$ for every embedding $\rho\colon A\rightarrow L$ into some $L\models\Sigma$.
\end{gev}

A simpler version of Proposition \ref{prop:generalcriterion} can be given when $\Sigma$ is a universal theory, i.e. consists only of universal sentences. This is the case, for example, when $\La$ is the language of rings and $\Sigma$ the theory of commutative rings.

\begin{gev}\label{generalcriterion-universal}
Assume that $\Sigma$ is a universal $\La$-theory.
For an existential $\La$-sentence $\varphi$ and $m \in \nat$, the following are equivalent:
\begin{enumerate}[(i)]
\item Every $K \models \Sigma \cup \lbrace \varphi \rbrace$  has a substructure $A$ generated by $m$ elements such that $A \models \varphi$.
\item $\qr_{\La, \Sigma}(\varphi) \leq m$.
\end{enumerate}
\end{gev}

\begin{proof}
By Proposition \ref{prop:generalcriterion} the implication $(i) \Rightarrow (ii)$ is immediate, without assumptions on $\Sigma$. For the other implication, note that 
since $\Sigma$ is a universal theory, substructures of models of $\Sigma$ are again models of $\Sigma$.
\end{proof}

\begin{opm}\label{rem:effective}
If $\Sigma$ is an $\La$-theory
such that the set of universal-existential consequences of $\Sigma$ (i.e.~$\La$-sentences $\varphi$ which are built from $\exists$-$\La$-formulas using $\wedge$, $\vee$, $\forall$ and such that $\Sigma \models \varphi$) is recursively enumerable,
then there is an {\em effective} method to find,
for every existential formula $\varphi$ with
$\qr_{\La,\Sigma}(\varphi)\leq m$,
an $\exists_m$-$\La$-formula that is equivalent to $\varphi$ modulo $\Sigma$.
Indeed, one can enumerate
universal-existential consequences of $\Sigma$
until finding a sentence of the form
$$
 \forall X_1,\dots,X_n(\varphi(X_1,\dots,X_n)\leftrightarrow\psi(X_1,\dots,X_n))
$$
with $\psi$ an $\exists_m$-formula (or rather, a universal-existential sentence logically equivalent to it); this procedure terminates by G{\"o}del's completeness theorem.
This applies in particular if $\Sigma$ itself is recursively enumerable,
for example when $\Sigma$ is the theory of fields.
\end{opm}

\begin{defi}
The {\em existential rank of $\Sigma$ in dimension $n$} is
$$
 \Erkn_{\mathcal{L}}(\Sigma) \;=\; \sup\{\Erk_{\mathcal{L},\Sigma}(\varphi) : \varphi\mbox{ an $\exists$-$\mathcal{L}$-formula in at most $n$ free variables}\},
$$
and the {\em existential rank of $\Sigma$} is
$$
 \Erk_\mathcal{L}(\Sigma) \;=\; \sup_{n\in\Natwithzero} \Erkn_{\mathcal{L}}(\Sigma).
$$
\end{defi}

\begin{lem}\label{lem:QE}
$\Erk_\mathcal{L}(\Sigma)=0$ if and only if $\Sigma$ has quantifier elimination.
\end{lem}

\begin{proof}
It is a standard criterion for quantifier elimination that
every formula of the form $\exists Y\varphi(\underline{X},Y)$ with $\varphi$ quantifier-free
is equivalent to a quantifier-free formula $\psi(\underline{X})$, see \cite[Lemma 3.1.5]{Marker},
which is provided by $\Erk_\mathcal{L}(\Sigma)=0$.
\end{proof}

\begin{opm}\label{rem:perk_analogues}
Just like the existential rank of $\Sigma$ in dimension $n$
we can similarly define the 
{\em positive-existential rank of $\Sigma$ in dimension $n$},
replacing $\exists$ by $\exists^+$ everywhere.
The same applies to Definition \ref{def:erk_K} below,
and with these definitions,
the Lemmas \ref{lem:language}, \ref{lem:constants},  \ref{lem:erk_elem_equiv}
and Corollary \ref{cor:erk_model_theory}
have their obvious positive-existential version.
\end{opm}

\begin{lem}\label{lem:language}
If $\mathcal{L}\subseteq\mathcal{L}'$ are languages,
and $\Sigma\subseteq \Sigma'$ are $\mathcal{L}$- respectively $\mathcal{L}'$-theories,
then $\erk_{\mathcal{L},\Sigma}(\varphi)\geq\erk_{\mathcal{L}',\Sigma'}(\varphi)$
for every $\mathcal{L}$-formula $\varphi$.
In particular, if $\Sigma'$ is in fact an $\mathcal{L}$-theory, then $\Erkn_\mathcal{L}(\Sigma)\geq\Erkn_\mathcal{L}(\Sigma')$ for every $n \in \Natwithzero$.
\end{lem}

\begin{proof}
This follows immediately from the definition.
\end{proof}

The previous lemma implies that
$\Erk_\mathcal{L}(\Sigma)\geq\Erk_\mathcal{L}(\Sigma')$.
On the other hand, it is not clear that $\Erk_\mathcal{L}(\Sigma)\geq\Erk_{\mathcal{L}'}(\Sigma)$.
However, the following does hold,
which is one of the reasons why we work with $\erk$ instead of $\Erkone$.

\begin{lem}\label{lem:constants}
For a set $C$ of constants,
$\Erk_\mathcal{L}(\Sigma)\geq\Erk_{\mathcal{L}(C)}(\Sigma)$.
\end{lem}

\begin{proof}
An existential $\mathcal{L}(C)$-formula is of the form $\varphi(\underline{X},\underline{c})$ with an existential $\mathcal{L}$-formula $\varphi(\underline{X},\underline{Y})$
and $c_1,\dots,c_m\in C$.
Since $\varphi(\underline{X},\underline{Y})$ is equivalent modulo $\Sigma$ to an existential 
$\mathcal{L}$-formula $\psi(\underline{X},\underline{Y})$ with at most $\Erk_\mathcal{L}(\Sigma)$ quantifiers,
$\varphi(\underline{X},\underline{c})$ is equivalent modulo $\Sigma$ to 
the $\mathcal{L}(C)$-formula
$\psi(\underline{X},\underline{c})$ with at most $\Erk_\mathcal{L}(\Sigma)$ quantifiers.
\end{proof}

\begin{defi}\label{def:erk_K}
Let $K$ be an $\mathcal{L}$-structure.
For an $\La(K)$-formula $\varphi$, we let
$$
 \erk_K(\varphi) \;=\; \erk_{\mathcal{L}(K),{\rm Th}_{\mathcal{L}(K)}(K)}(\varphi).
$$
If $D \subseteq K^n$ is a definable subset (where we always allow definitions involving parameters from $K$), we define $\erk_K(D)$ to be $\erk_K(\varphi)$, where $\varphi$ is any $\La(K)$-formula defining $D$ in $K$.
The {\em existential rank} of $K$ is
$$
 \erk(K) \;=\; \erk_{\mathcal{L}(K)}({\rm Th}_{\mathcal{L}(K)}(K)).
$$
We also write $\Erkn(K)=\Erkn_{\mathcal{L}(K)}({\rm Th}_{\mathcal{L}(K)}(K))$.

\end{defi}

\begin{opm}
In other words, $\erk(K)\leq m$ if and only if
for every existential $\mathcal{L}(K)$-formula $\varphi$ in any number of variables 
there is an existential $\mathcal{L}(K)$-formula $\psi$ with at most $m$ quantifiers such that
$K\models\forall\underline{X}(\varphi\leftrightarrow\psi)$.
As $\erk_K(\varphi)$ depends only on the set $D=\varphi(K)$ defined by $\varphi$ in $K$, the notation $\erk_K(D)$ is justified.
\end{opm}

\begin{gev}\label{cor:erk_model_theory}
For every $K\models \Sigma$, we have $\erk(K)\leq\erk_\mathcal{L}(\Sigma)$.
\end{gev}

\begin{proof}
This follows by combining Lemmas \ref{lem:constants} and \ref{lem:language}.
\end{proof}

\begin{lem}\label{lem:erk_elem_equiv}
If $K$ and $L$ are $\mathcal{L}$-elementarily equivalent, then $\erk(K)=\erk(L)$.
\end{lem}

\begin{proof}
Let $\varphi(\underline{X},\underline{Y})$ be an existential $\mathcal{L}$-formula and $y_1,\dots,y_m\in L$.
There is an existential $\mathcal{L}$-formula $\psi(\underline{X},\underline{Y},\underline{Z})$ with at most $\erk(K)$ quantifiers and $c_1,\dots,c_r\in K$
such that 

$$
 K\models\forall\underline{X}\forall\underline{Y}(\varphi(\underline{X},\underline{Y})\leftrightarrow\psi(\underline{X},\underline{Y},\underline{c}))
$$ 
and thus
$$
 K\models\exists\underline{Z}\forall\underline{X}\forall\underline{Y}(\varphi(\underline{X},\underline{Y})\leftrightarrow\psi(\underline{X},\underline{Y},\underline{Z})),
$$ 
so if $L\equiv K$ then there exist $z_1,\dots,z_r\in L$ 
with 
$$
 L\models\forall\underline{X}(\varphi(\underline{X},\underline{y})\leftrightarrow\psi(\underline{X},\underline{y},\underline{z})).
$$
As $\varphi(\underline{X},\underline{y})$ was an arbitrary existential $\mathcal{L}(L)$-formula,
this shows $\erk(L)\leq\erk(K)$, and the claim follows by symmetry.
\end{proof}

Finally, for later use, we mention the following general lemmas.
Recall here the notion of a direct limit of $\La$-structures, see for instance \cite[Section 2.4]{Hodges_Longer}.
\begin{lem}\label{lem:directLimit}
%  Let $\Sigma$ be an $\La$-theory.
  The class of $\La$-structures $K$ for which there exists an $\La$-homomorphism $K \to L$ to a model $L$ of $\Sigma$ is closed under direct limits.
\end{lem}
\begin{proof}
  By \cite[Theorem 2.4.6]{Hodges_Longer}, it suffices to show that this class $\mathcal K$ is axiomatised by a set of universal $\La$-sentences.
  This axiomatisability in turn follows from the fact that $\mathcal K$ is closed under passage to ultraproducts and substructures using \cite[Corollary 9.5.10]{Hodges_Longer} (so $\mathcal K$ is axiomatisable) and the Łoś–Tarski theorem \cite[Theorem 6.5.4]{Hodges_Longer}.
\end{proof}

\begin{lem}\label{lem:finiteness}
Let %$\Sigma$ be an $\La$-theory 
 $\varphi$ be an $\La$-sentence. 
Suppose that $K$ is an $\La$-structure such that $L \models \varphi$ for any $L \models \Sigma$ for which there exists a homomorphism $K \to L$. 
Then there exists a finitely generated $\La$-substructure $A$ of $K$ such that $L \models \varphi$ for any $L \models \Sigma$ for which there exists a homomorphism $A \to L$.
\end{lem}
\begin{proof}
  We apply Lemma \ref{lem:directLimit} to the theory $\Sigma \cup \{ \neg \varphi \}$.
  Since $K$ is not in the class mentioned there, and $K$ is the direct limit of its finitely generated substructures, there exists a finitely generated substructure as desired.
\end{proof}

\section{Theories of fields}\label{sect:erkfields}

\noindent
For the rest of this paper we will study fields and work in 
languages $\mathcal{L}_{\rm ring}(C)$ where $\mathcal{L}_{\rm ring}=\{+,-,\cdot,0,1\}$ is the language of rings
and $C$ a set of constant symbols, 
and with an $\Lar(C)$-theory $\Sigma$ containing the $\mathcal{L}_{\rm ring}$-theory $\Tfields$ of fields.
We recall that we drop the language from the notation
if there is no danger of confusion
and therefore write e.g.~$\qr_\Sigma$ instead of $\qr_{\mathcal{L}_{\rm ring}(C),\Sigma}$.
Let $K$ always be a field.

\begin{opm}
Instead of the language of rings
one might also consider the language of fields $\La_{\rm field}$ extending the language of rings by a unary function symbol $\cdot^{-1}$ for inversion, with the convention $0^{-1}:=0$. 
While substructures of fields in $\mathcal{L}_{\rm ring}$ are subrings,
substructures in $\mathcal{L}_{\rm field}$ are then subfields.
However, we will restrict our attention to $\mathcal{L}_{\rm ring}$, since anyway
$$
 \erk_{\mathcal{L}_{\rm ring}(C),\Sigma}(\varphi) \;=\; 
 \erk_{\mathcal{L}_{\rm field}(C),\Sigma}(\varphi) \;=\;
 \perk_{\mathcal{L}_{\rm field}(C),\Sigma}(\varphi)
$$
for every $\mathcal{L}_{\rm ring}(C)$-formula $\varphi$.
While this can be seen by direct manipulation of the formula,
it also follows immediately from Corollaries \ref{cor:generalcriterionhom} and \ref{cor:generalcriterion}
using the basic facts that homomorphisms between fields are embeddings and every embedding of a ring into a field can be extended to its fraction field.
The precise relation between $\erk_{\mathcal{L}_{\rm ring}(C),\Sigma}$ and $\perk_{\mathcal{L}_{\rm ring}(C),\Sigma}$ will be determined in Proposition \ref{prop:qrvsqrp}.
\end{opm}

We will often work with the following theories of fields:

\begin{defi}\label{def:theories}
We denote 
\begin{enumerate}
\item by $\TKequiv$ the complete $\mathcal{L}_{\rm ring}$-theory ${\rm Th}_{\mathcal{L}_{\rm ring}}(K)$ of $K$,
\item by $\TKleq$ the union of $\Tfields$ with the quantifier-free diagram of $K$, i.e.~the set of atomic and negated atomic $\Lar(K)$-sentences satisfied by $K$,
\item by $\TKec$ the union of $\Tfields$ with the universal $\Lar(K)$-theory of $K$, and 
\item by $\TKprec$ the complete $\Lar(K)$-theory ${\rm Th}_{\mathcal{L}_{\rm ring}(K)}(K)$ of $K$ (also known as the elementary diagram of $K$).
\end{enumerate}
\end{defi}

\begin{opm}\label{rem:theories}
Note that we view $\TKleq$, $\TKec$ and $\TKprec$ as $\Lar(K)$-theories but $\TKequiv$ as an $\Lar$-theory.
We have $\TKleq \subseteq \TKec \subseteq \TKprec$ and $\TKequiv\subseteq\TKprec$.
The models of 
$\TKequiv$ are the fields elementarily equivalent to $K$,
the models of $\TKleq$ are the field extensions of $K$,
the models of $\TKec$ are the field extensions of $K$ in which $K$ is existentially closed
(we write $K\prec_\exists L$ for $K$ existentially closed in $L$), and
the models of $\TKprec$ are the elementary extensions of~$K$.
In particular, $\erk_K$ (see Definition \ref{def:erk_K}) can be written as $\erk_{\TKprec}$.
\end{opm}

\begin{opm}\label{rem:perk}
Since models of $\Tfields$ are integral domains,
a finite disjunction of polynomial equalities $f_i=0$ is equivalent modulo $\Tfields$ to the single equality $\prod_i f_i=0$.
Therefore, every $\exists^+_m$-$\Lar(K)$-formula $\varphi$ is equivalent modulo $\TKleq$
to a positive primitive $\Lar(K)$-formula with $m$ quantifiers,
i.e.~a formula of the form 
\begin{equation}\label{eqn:ppformula}
 \exists Y_1,\dots,Y_m \bigwedge_{i=1}^r f_i(\underline X, \underline Y) \doteq 0
\end{equation}
for suitable $f_i \in K[\underline X, \underline Y]$.
If moreover $K$ is not algebraically closed, 
then $\varphi$ is equivalent modulo $\TKec$
to a formula as in (\ref{eqn:ppformula}) with $r=1$,
since if $g\in K[Z]\setminus K$ has no zero in $K$, and $g^*=Z_0^{{\rm deg}(g)}g(\frac{Z_1}{Z_0})\in K[Z_0,Z_1]$
is the homogenization of $g$, then $z_1,z_2\in K$ satisfy $g^*(z_1,z_2)=0$ if and only if $z_1=z_2=0$.
%form $\exists Y_1,\dots,Y_m f(\underline X, \underline Y) \doteq 0$ for some $f \in K[\underline X, \underline Y]$,
In particular, $\perk_K(\varphi(K))$ agrees with the definition of positive-existential rank in the introduction (Definition~\ref{def:intro_erk}),
and unless $K$ is algebraically closed one could also take $r=1$ in that definition.
\end{opm}

\begin{prop}\label{prop:acfErkZero}
$\erk(K) = 0$ if and only if $K$ is  either finite or algebraically closed.
\end{prop}

\begin{proof}
If $K$ is algebraically closed,
then since {\rm ACF} has quantifier elimination (cf.~Example~\ref{ex:ACF}), so does $\TKequiv$, 
hence $0=\erk_{\mathcal{L}_{\rm ring}}(\TKequiv)\geq\erk(K)$
by Lemma \ref{lem:QE} and Corollary \ref{cor:erk_model_theory}.
If $K$ is finite, then trivially every subset of $K$ is quantifier-freely definable in $\Lar(K)$, hence $\erk(K) = 0$.
Conversely, if $\erk(K)=0$, 
then every $L\equiv_{\mathcal{L}_{\rm ring}} K$ satisfies
$\erk(L)=0$ by Lemma \ref{lem:erk_elem_equiv} and therefore
$T_L^\prec$ has quantifier elimination by Lemma \ref{lem:QE}.
In particular, every $L$-definable subset of $L$ is finite or cofinite.
We have thus shown that $\TKequiv$ is strongly minimal and hence $\omega$-stable,
which if $K$ is infinite implies by Macintyre's result \cite[Theorem 1]{Macintyre_CategoricalFields} that $K$ is algebraically closed.\footnote{In \cite[Theorem 3]{Macintyre_CategoricalFields}, Macintyre himself deduces from his result on stability that the only infinite fields whose $\Lar$-theory has quantifier elimination are algebraically closed, whereas we need to consider the $\Lar(K)$-theory of $K$.}
\end{proof}

We now apply the general results in the last section 
as an example to one specific $\exists$-$\Lar$-formula that will play an important role in Sections \ref{sect:efd} and \ref{sect:pthpowers}.
\begin{defi}\label{def:pi}
For $n,m\in\Natwithoutzero$ let
\begin{displaymath}
\pii{n}{m} = \exists Y_1, \ldots, Y_n\bigwedge_{i=1}^n X_i \doteq Y_i^m.
\end{displaymath}
\end{defi}
So in a field $K$, $\pii{n}{m}$ defines the set 
$(K\pow{m})^n$
of $n$-tuples of $m$-th powers.
Clearly $\perk_{\Lar, \emptyset}(\pii{n}{1}) = 0$ for every $n\in\Natwithoutzero$
and $\perk_{\Lar,\emptyset}(\pii{n}{m})\leq n$ for every $n,m\in\Natwithoutzero$.
\begin{lem}\label{lem:powersQFree}
  Let $K$ be an infinite field and $n, m \in \Natwithoutzero$. Then $\qr_{K}(\pii{n}{m}) = 0$ if and only if every element of $K$ is an $m$-th power,
  i.e.~$K=K\pow{m}$.
\end{lem}
\begin{proof}
  If every element of $K$ is an $m$-th power, then $\pii{n}{m}$ is equivalent modulo $\TKprec$ to the tautological formula $0 \doteq 0$.
  If not every element of $K$ is an $m$-th power,
  then the set $\powx{K}{m}$ of $m$-th powers in the multiplicative group $K^\times$ is an infinite proper subgroup of $K^\times$,
  in particular also co-infinite. 
  Thus $\pii{1}{m}$ is not equivalent  modulo $\TKprec$ to a quantifier-free formula over $K$, and therefore neither is $\pii{n}{m}$.
\end{proof}

\begin{prop}\label{prop:pinmperfect}
Let $m\geq2$ and let $\Sigma$ be the $\Lar$-theory of fields which are either perfect or of characteristic not dividing $m$. Then $\qr_{\Sigma}(\pii{n}{m}) = 1$
for every $n\in\Natwithoutzero$.
\end{prop}
\begin{proof}
Let $n\in\Natwithoutzero$.
Lemma \ref{lem:powersQFree} and Lemma \ref{lem:language} give that $\erk_{\Sigma}(\pii{n}{m}) \geq \erk_\mathbb{Q}(\pii{n}{m})\geq 1$.

Take $K \models \Sigma$ and $x_1, \ldots, x_n \in K$ such that $K \models \pii{n}{m}(\underline{x})$, 
i.e.~$x_1, \ldots, x_n$ are $m$-th powers in $K$. 
Let $A_0$ be the subring of $K$ generated by $\underline{x}$. 
To apply Corollary \ref{cor:generalcriterion} to get $\erk_{\Sigma}(\pii{n}{m})\leq1$, we need to show that there is a subring $A$ of $K$ generated by one element over $A_0$ such that $L \models \pii{n}{m}(\underline{x})$ for any model $L$ of $\Sigma$ containing $A$ as a substructure.
To this end, let $K_0$ be the fraction field of $A_0$ in $K$ if $\car(K) \nmid m$, or the perfect hull of the fraction field of $A_0$ in $K$ if $\car(K) \mid m$. Fix $y_1, \ldots, y_n \in K$ such that $y_i^m = x_i$ for each $i \in \lbrace 1, \ldots, n \rbrace$ and set $K' = K_0(\underline{y})$. As $K_0$ is perfect or of characteristic not dividing $m$, $K'/K_0$ is a separable finite extension. By the primitive element theorem, there exists a $y' \in K'$ such that $K' = K_0(y')$. Now set $A = A_0[y']$. As any model of $\Sigma$ containing $A$ as a subring will also contain $K'$, and $K' \models \pii{n}{m}(\underline{x})$, this $A$ is as desired.
\end{proof}

\begin{vb}\label{ex:RCF}
Let $\Sigma={\rm RCF}$ be the theory of real closed fields in $\mathcal{L}_{\rm ring}$.
By Proposition~\ref{prop:pinmperfect} and Lemma \ref{lem:powersQFree}, $\erk_{\rm RCF}(\pii{n}{2})=1$ for every $n$.
Since the theory ${\rm RCF}$ has quantifier elimination in the language
$\mathcal{L}_{\rm ring}\cup\{<\}$ of ordered rings,
and modulo ${\rm RCF}$ the formulas $X<Y$ and $\neg X<Y$ are equivalent  to 
$\exists Z(Y-X \doteq Z^2\wedge\neg X\doteq Y)$ respectively $\exists Z(X-Y \doteq Z^2)$,
one can eliminate any number $n$ of occurrences of $<$ at the expense of only
$\erk_{\rm RCF}(\pii{n}{2})=1$ existential quantifier.
Therefore, $\erk_{\rm RCF}(\varphi)\leq1$
for every $\mathcal{L}_{\rm ring}$-formula $\varphi$.
In particular,
%together with Proposition~\ref{prop:acfErkZero} 
%and Corollary \ref{cor:erk_model_theory}
we obtain $\erk({\rm RCF})=1$.
A similar argument applies to the theory
$p{\rm CF}_d$ of $p$-adically closed fields
of $p$-rank $d$, cf.~\cite{PR}.
For the special case of existential formulas
we will give a more conceptual proof in Corollary~\ref{cor:RCFpCF}.
\end{vb}

\begin{opm}\label{opm:zhangsun}
In \cite[Theorem~1.2]{ZhangSun}
an explicit $\exists_1$-formula, covering a special case of Proposition~\ref{prop:pinmperfect}
(concerning the field $\mathbb{Q}$, or more generally ordered fields, and $m=2$),
is constructed using elementary methods.
The first arXiv version of the present article precedes the first arXiv version of \cite{ZhangSun}.
\end{opm}

In fact it will follow from Proposition \ref{prop:qrvsqrp} that in the situation of Proposition \ref{prop:pinmperfect} one even has $\qrp_{\Sigma}(\pii{n}{m}) = 1$. In contrast to this, we have the following,
which will be strengthened considerably in Proposition \ref{large}.

\begin{prop}\label{ppowerscharp}
Fix a prime number $p$ and let $\Sigma$ be the $\Lar$-theory of fields of characteristic $p$. 
Then $\qr_{\Sigma}(\pii{n}{p}) = n$ for every $n\in\Natwithoutzero$.
\end{prop}

\begin{proof}
We already have $\qr_{\Sigma}(\pii{n}{p}) \leq n$. 
Let $K = \ff_p(t_1^{1/p}, \ldots, t_n^{1/p})$ for some algebraically independent $t_1, \ldots, t_n$. 
Then $K \models \pii{n}{p}(\underline{t})$.
On the other hand, 
if $A$ is any substructure of $K$ generated by elements $y_1,\dots,y_{n-1}$ over $\mathbb{F}_p[\underline{t}]$,
then $K'={\rm Frac}(A)$ is generated by $y_1,\dots,y_{n-1}$ over $\mathbb{F}_p(\underline{t})$,
but since $K/\mathbb{F}_p(\underline{t})$ is purely inseparably of degree $p^n$
and each $y_i$ has degree at most $p$ over $\mathbb{F}_p(\underline{t})$, 
we get $K'\neq K$ and therefore $K' \not\models \pii{n}{p}(\underline{t})$.
%cannot contain all $n$ elements $t_1,\dots,t_n$,
%and so $K' \not\models \pii{n}{p}(\underline{t})$.
%but no subfield $K'$ of $K$ generated by $n-1$ elements over $\ff_p(t_1, \ldots, t_n)$ satisfies $K' \models \pii{n}{p}(\underline{t})$. 
By Corollary \ref{cor:generalcriterion}, $\qr_{\Sigma}(\pii{n}{p}) > n-1$.
\end{proof}

\section{Essential fibre dimension}\label{sect:efd}

\noindent
In this section we introduce a general notion of essential fibre dimension
and investigate its relation to existential and positive-existential rank.

We fix a set $C$ of constants and an $\Lar(C)$-theory $\Sigma$
containing the $\Lar$-theory $\Tfields$ of fields.
If $K$ is a model of the $\mathcal{L}_{\rm ring}(C)$-theory of fields,
we denote by $K_C$ the fraction field of the smallest substructure of $K$, i.e.~the subfield generated by the interpretations of the constant symbols $c\in C$.
Note that if $C=\emptyset$, then $K_C$ is simply the prime field of $K$.

\begin{defi}\label{def:efd}
Let $\varphi$ be an existential $\Lar(C)$-formula in free variables $X_1, \ldots, X_n$. 
The \emph{essential fibre dimension $\efd_{\Sigma}(\varphi)$ of $\varphi$ with respect to $\Sigma$} is 
the smallest $d \in \Natwithzero$ such that for every $K \models \Sigma$ and every $x_1, \ldots, x_n \in K$ with $K \models \varphi(\underline{x})$,
there exists a subextension $K'/K_C(\underline{x})$ of $K/K_C(\underline{x})$ of transcendence degree 
$\trdeg(K'/K_C(\underline{x}))\leq d$ 
such that $L \models \varphi({\rho(x_1)}, \ldots, {\rho(x_n)})$ for every $L \models \Sigma$ with an $\mathcal{L}_{\rm ring}(C)$-embedding $\rho \colon K' \to L$.
Instead of $\efd_{\subTKprec}(\varphi)$ we will also write $\efd_K(\varphi)$ and $\efd_K(\varphi(K))$.
\end{defi}

\begin{opm}\label{rem:efdinvariant}\label{rem:efdsentences}
Note that $\efd_{\Sigma}(\varphi)$ depends only on the equivalence class of $\varphi$ modulo $\Sigma$.
In particular, $\efd_K(\varphi)$ depends only on the set $D=\varphi(K)$ defined by $\varphi$ in $K$, which justifies the notation $\efd_K(D)$.
As in the case of the existential rank, since we allow extension of the language by adding additional constant symbols, we can limit our study to that of existential sentences, i.e.~without free variables.
Justification for the name {\em essential fibre dimension}
and the connection with the definition of essential fibre dimension in the introduction
will be given in Proposition \ref{prop:efdCharacterisationGeneral} and Corollary \ref{cor:efdCharacterisationDiag}.
\end{opm}

\begin{lem}\label{rem:defEfdFieldVsRing}
For an existential $\Lar(C)$-formula $\varphi$ and $d\in\Natwithzero$,
$\efd_\Sigma(\varphi)\leq d$ if and only if
 for every $\Sigma \models K$ with smallest $\Lar(C)$-substructure $R_0$ and every $\underline x \in K^n$ with $K \models \varphi(\underline x)$, there exists a finitely generated $R_0[\underline x]$-subalgebra $R$ of $K$ with $\trdeg(\operatorname{Frac}(R) / K_C(\underline x)) \leq d$ such that $L \models \varphi(\rho(\underline x))$ for every $L\models\Sigma$ with an $\Lar(C)$-homomorphism $\rho \colon R \to L$. 
\end{lem}

\begin{proof}
This follows from Lemma \ref{lem:finiteness}.
\end{proof}

\begin{lem}\label{lem:efd_theory}
If $C\subseteq C'$ are sets of constants and $\Sigma\subseteq\Sigma'$ are $\Lar(C)$- respectively $\Lar(C')$-theories,
then $\efd_\Sigma(\varphi)\geq\efd_{\Sigma'}(\varphi)$
for every existential $\Lar(C)$-formula $\varphi$.
\end{lem}

\begin{proof}
This follows directly from the definition.
\end{proof}

\begin{lem}\label{edwedge}
For existential formulas $\varphi_1$ and $\varphi_2$ one has
\begin{eqnarray*}
\efd_{\Sigma}(\varphi_1 \vee \varphi_2)& \leq& \max\{\efd_{\Sigma}(\varphi_1),\efd_{\Sigma}(\varphi_2)\},\\
\efd_{\Sigma}(\varphi_1 \wedge \varphi_2) &\leq& \efd_{\Sigma}(\varphi_1) + \efd_{\Sigma}(\varphi_2).
\end{eqnarray*}
In particular, 
the class of existential formulas of essential fibre dimension zero
is closed under finite conjunctions and disjunctions.
\end{lem}
\begin{proof}
Assume without loss of generality that $\varphi_1$ and $\varphi_2$ are sentences (Remark \ref{rem:efdsentences}).

If $K\models\Sigma \cup \lbrace \varphi_1 \wedge \varphi_2 \rbrace$, then by assumption there exists for $i\in \lbrace 1, 2 \rbrace$ a subfield $K_i$ of $K$ of transcendence degree $d_i \leq \efd_{\Sigma}(\varphi_i)$ over $K_C$ such that $L \models \varphi_i$ for each $L \models \Sigma$ into which $K_i$ embeds over $K_C$. 
Then the compositum $K_1K_2$ is a subfield of $K$ of transcendence degree at most $d_1 + d_2$ over $K_C$, and one has $L \models \varphi_1 \wedge \varphi_2$ for every $L \models \Sigma$ into which $K_1K_2$ embeds. This shows that $\efd_{\Sigma}(\varphi_1 \wedge \varphi_2) \leq d_1 + d_2$.

The statement for $\varphi_1\vee\varphi_2$ is proven similarly.
\end{proof}

\begin{vb}\label{ex:pii_efd}
For every $n, m \in \Natwithoutzero$ we have $\efd_{\Tfields}(\pii{n}{m})=0$ (recall $\pii{n}{m}$ from Definition~\ref{def:pi}).
This shows that in general the existential rank of an existential formula can be arbitrarily large compared to its essential fibre dimension, as for a prime number $p$ and $n \in \Natwithoutzero$, $\qr_{\Tfields}(\pii{n}{p}) = n$ by Proposition \ref{ppowerscharp} and Lemma \ref{lem:language}.
\end{vb}

\begin{vb}
For every $n \in \nat$ there is an existential
%algebraic existential 
$\Lar$-formula $\varphi$ with $\efd_{\Tfields}(\varphi)=0$ in free variables $X_0, \ldots, X_n$ such that for any field $K$ and $x_0, \ldots, x_n \in K$ we have $K \models \varphi(x_0, \ldots, x_n)$ if and only if the polynomial $f=\sum_{i=0}^n x_iT^i\in K[T]$ is reducible. 
Indeed, $f$ is reducible over $K$ if and only if it is reducible over the relative algebraic closure $K'$ of 
the subfield generated by $\underline{x}$, as the coefficients of the factors of $f$ are symmetric polynomials in its roots.
\end{vb}

\begin{vb}\label{ex:efd0}
If $\Sigma$ is one of the theories
${\rm ACF}$, ${\rm RCF}$ or  $p{\rm CF}_d$ (cf.~Example \ref{ex:RCF}) in $\Lar$ (so $C=\emptyset$),
then $\efd_\Sigma(\varphi)=0$ for every existential $\mathcal{L}_{\rm ring}$-formula $\varphi$: 
In each of these theories, a relatively algebraically closed subfield is an elementary substructure
(see \cite[Theorems 3.4 and 5.1]{PR} for $p{\rm CF}_d$),
so if $K\models\Sigma$ and $K\models\varphi(\underline{x})$, 
then the relative algebraic closure $K'$ of $K_C(\underline{x})$ in $K$
satisfies $K'\prec K$
and hence $K'\models\Sigma\cup\{\varphi(\underline{x})\}$.
\end{vb}

It is not hard to see that modulo $\Tfields$, every quantifier-free formula is equivalent to an $\exists_1^+$-formula. One can invoke Corollary \ref{cor:generalcriterionhom} for this, but it can also be seen more directly: put the formula in disjunctive normal form, 
use that a conjunction of finitely many inequalities is equivalent to one inequality by virtue of the fact that all models are domains
(cf.~Remark \ref{rem:perk}), 
then finally replace the (at most) one remaining inequality by an $\exists_1^+$-formula by using that an element of a model is non-zero if and only if it has an inverse. In particular, we obtain that
\begin{equation}\label{eqn:erkperk}
\qr_{\Sigma}(\varphi) \leq \qrp_{\Sigma}(\varphi) \leq \qr_{\Sigma}(\varphi) + 1.
\end{equation}
for any $\Lar(C)$-theory $\Sigma$ containing $\Tfields$.
We will now see that, in all but some well-understood cases, $\qr_{\Sigma}(\varphi) = \qrp_{\Sigma}(\varphi)$.

\begin{prop}\label{edleqqr}
Let $\Sigma$ be an $\mathcal{L}_{\rm ring}(C)$-theory containing $\Tfields$. 
For every existential formula $\varphi$ we have
\begin{displaymath}
\efd_{\Sigma}(\varphi) \leq \max\lbrace\qrp_{\Sigma}(\varphi) - 1, 0\rbrace.
\end{displaymath}
\end{prop}
\begin{proof}
Without loss of generality $\varphi$ is an $\exists^+_m$-sentence with $m = \qrp_{\Sigma}(\varphi)$
(Remarks \ref{remstatements} and \ref{rem:efdsentences}). 
Given any $K \models \Sigma \cup \lbrace \varphi \rbrace$, let $(y_1, \ldots, y_m) \in K^m$ be a witness for $\varphi$. 
Then $K_C[\underline{y}] \models \varphi$. 
If $y_1, \ldots, y_m$ are algebraically independent over $K_C$, then there is a homomorphism $K_C[\underline{y}] \to K_C$ mapping each $y_i$ to $0$, whence also $K_C \models \varphi$, and we let $K'=K_C$. 
Otherwise $\trdeg(K_C(\underline{y})/K_C) \leq m-1$, and we let $K'=K_C(\underline{y})$. 
In either case $\trdeg(K'/K_C)\leq\max \lbrace m-1, 0 \rbrace$ and $K' \models \varphi$. 
This shows that $\efd_{\Sigma}(\varphi) \leq \max \lbrace m-1, 0 \rbrace$.
\end{proof}

\begin{lem}\label{lem:overring}
Let $S/R$ be an extension of integral domains such that $S$ is finitely generated over $R$. Set $K = \Frac(R)$, $L = \Frac(S)$ and assume that $L/K$ is a simple algebraic extension. Then there exists an overring of $S$ which is generated over $R$ by one element.
\end{lem}

\begin{proof}
By assumption, $L = K[\alpha]$ for some $\alpha \in L$, and $S = R[\beta_1, \ldots, \beta_m]$ for certain $\beta_1, \ldots, \beta_m \in S$. Using that $\alpha$ is algebraic over $K$ and a primitive element of $L/K$, we can find $f_1, \ldots, f_m \in R[\alpha]$ and $g \in R \setminus \lbrace 0 \rbrace$ such that $\frac{\beta_i}{\alpha} = \frac{f_i}{g}$ for each $i \in \lbrace 1, \ldots, m \rbrace$. Then $S \subseteq R[\frac{\alpha}{g}] \subseteq L$.
\end{proof}

\begin{prop}\label{prop:qrvsqrp}
Let $\Sigma$ be an $\Lar(C)$-theory containing $\Tfields$, and $\varphi$ an existential formula. 
Then precisely one of the following holds:
\begin{enumerate}[$\rm(i)$]
\item\label{item1} $\qr_{\Sigma}(\varphi) = \qrp_{\Sigma}(\varphi)$,
\item\label{item2} $\qr_{\Sigma}(\varphi) = 0$ and $\qrp_{\Sigma}(\varphi) = 1$,
\item\label{item3} $\qr_{\Sigma}(\varphi) = 1 = \efd_{\Sigma}(\varphi)$, $\qrp_{\Sigma}(\varphi) = 2$, and there exists some finite $L \models \Sigma$ and a tuple $\underline{a}$ in $L$ such that $L \not\models \varphi(\underline{a})$, but $M \models \varphi(\underline{a})$ for every infinite $M \models \Sigma$ extending $L$. 
\end{enumerate}
\end{prop}

\begin{proof}
First of all the three cases are clearly mutually exclusive.
Let $m = \qr_{\Sigma}(\varphi)$
and assume again without loss of generality that $\varphi$ is an $\exists_m$-$\Lar(C)$-sentence (Remarks \ref{remstatements} and \ref{rem:efdinvariant}). 
Note that $m\leq\qrp_{\Sigma}(\varphi) \leq m+1$ by (\ref{eqn:erkperk}).
If $m = 0$ we are in case (\ref{item1}) or (\ref{item2}),
so assume that $m\geq 1$.
We want to argue that we are either in case {\rm(\ref{item1})} or {\rm(\ref{item3})}.

To this end we will apply Proposition~\ref{prop:generalcriterionhom}. 
So let $K$ be a model of $\Sigma \cup \lbrace \varphi \rbrace$
and let $R_0$ denote its smallest substructure. Note that $K_C = \Frac(R_0)$.
To prove that $\qrp_\Sigma(\varphi) \leq m$, we want to show that
\begin{equation}\label{conddagger}
\begin{minipage}{14cm}
{there exists a subring $A$ of $K$ generated by $m$ elements over $R_0$
 such that
$L\models\varphi$ for every $L\models\Sigma$ 
which has an $\mathcal{L}_{\rm ring}(C)$-homomorphism $A\rightarrow L$.}
\end{minipage}
\end{equation}
We will show that either this criterion is always satisfied, or we are in case {\rm(\ref{item3})}.
By Proposition \ref{prop:generalcriterion} we have that
\begin{equation}\label{conddagger2}
\begin{minipage}{14cm}
{there exists a subring $B$ of $K$ generated by $m$ elements over $R_0$ such that $L \models \varphi$ for every $L \models \Sigma$ which has an $\Lar(C)$-embedding $B \to L$.}
\end{minipage}
\end{equation}
Let $B_0$ be some subring of $K$ such that (\ref{conddagger2}) is satisfied for $B = B_0$
and such that furthermore $\trdeg(\Frac(B_0)/K_C) \leq \trdeg(\Frac(B)/K_C)$ for all subrings $B$ of $K$ for which (\ref{conddagger2}) is satisfied. Set $F = \Frac(B_0)$. As $F$ is generated by $m$ elements $a_1, \ldots, a_m$ over $K_C$, clearly $\trdeg(F/K_C) \leq m$. We consider two cases.

\textit{Case 1: $\trdeg(F/K_C) < m$.}
In this case $a_1, \ldots, a_m$ are algebraically dependent. Without loss of generality $a_m$ is algebraic over $K_C(a_1, \ldots, a_{m-1})$. By Lemma \ref{lem:finiteness} there exists a subring $S$ of $F$ finitely generated over $R_0$ such that $L \models \varphi$ for every $L \models \Sigma$ which has an $\Lar(C)$-homomorphism $S \to L$. 
Without loss of generality, $S$ contains $a_1,\dots,a_{m-1}$.
Lemma \ref{lem:overring} applied to the extension $S/R_0[a_1, \ldots, a_{m-1}]$ yields the existence of an overring $A$ of $S$ generated by one element over $R_0[a_1, \ldots, a_{m-1}]$ and therefore by $m$ elements over $R_0$, whence (\ref{conddagger}) is satisfied.

\textit{Case 2: $\trdeg(F/K_C) = m$.}
In this case $a_1, \ldots, a_m$ are algebraically independent, and $F/K_C$ is a purely transcendental extension.
If $L$ is an infinite model of $\Sigma$ with an $\Lar(C)$-embedding $K_C \to L$, then this embedding extends to an embedding $F \to L^\ast$ for some elementary extension $L^\ast$ of $L$, by mapping $a_1, \ldots, a_m$ to any $m$ elements which are algebraically independent over $L$. 
By (\ref{conddagger2}) for $B=B_0\subseteq F$ it follows that $L^\ast \models \varphi$, whence also $L \models \varphi$.

In particular, as $K_C(a_1)$ is infinite, any $L \models \Sigma$ with an $\Lar(C)$-embedding $K_C(a_1) \to L$ satisfies $L \models \varphi$, i.e. (\ref{conddagger2}) is satisfied with $B = R_0[a_1]$. 
The choice of $B_0$ implies that $m = 1$. 
Furthermore, there exists $L \models \Sigma$ into which $K_C$ embeds such that $L \not\models \varphi$, but such that $M \models \varphi$ for all infinite $M \models \Sigma$ extending $L$ by the previous paragraph. 
In particular, $L$ and hence $K_C$ are finite.

Finally, since $K_C$ is finite, $K_C = R_0$ is a finite field. If we had $\efd_\Sigma(\varphi) = 0$, then by Lemma \ref{rem:defEfdFieldVsRing} there would be a finite extension $K'/K_C$ contained in $K/K_C$ such that $L \models \varphi$ for any $L \models \Sigma$ into which $K'$ embeds.
But then $K'$ would be generated by one element over $K_C$ and hence (\ref{conddagger2}) would be satisfied with $B = K'$, 
contradicting the choice of $B_0$.
We infer via Proposition \ref{edleqqr} that $0 < \efd_\Sigma(\varphi) < \qrp_\Sigma(\varphi) \leq m + 1 = 2$, whence $\efd_\Sigma(\varphi) = 1$ and $\qrp_\Sigma(\varphi) = 2$. As such, we are in case {\rm(\ref{item3})}, and this concludes the proof.
\end{proof}

\begin{gev}\label{cor:qrvsqrp}
Let $K$ be a field and $D \subseteq K^n$ a diophantine set. Then either $\erk_K(D) = \qrp_K(D)$ or $D$ is quantifier-freely definable.
\end{gev}
\begin{proof}
If $K$ is finite, then any subset of $K^n$ is quantifier-freely definable. If $K$ is infinite, then so is any model of $\TKprec$, whence case (\ref{item3}) in Proposition \ref{prop:qrvsqrp} cannot occur when applied to $\TKprec$, and we get the desired result.
\end{proof}

The previous proof uses that
case (\ref{item3}) in Proposition \ref{prop:qrvsqrp} can only occur when $\Sigma$ admits finite models, 
but note that $\Sigma$ also needs to admit infinite models since otherwise $\efd_\Sigma(\varphi) > 0$ is impossible.
We view this case as degenerate.
See Example \ref{vb:weirdErkPerk} for a situation in which it occurs.

It is of interest to determine how existential rank can change when finite models are excluded from the models of a given theory (and therefore case (\ref{item3}) is avoided).
The following proposition shows that existential rank cannot then drop arbitrarily far.
\begin{prop}\label{prop:infmodels}
  Let $\Sigma$ be an $\Lar(C)$-theory containing $\Tfields$.
  Denote by $\Sigma_\infty$ the theory of infinite models of $\Sigma$.
  Then for every existential $\Lar(C)$-formula $\varphi$ we have 
  $$
   \erk_\Sigma(\varphi) \leq \max\{ 2, \erk_{\Sigma_\infty}(\varphi) \}.
   $$
\end{prop}
\begin{proof}
  We may again assume without loss of generality
  that $\varphi$ is a sentence (Remarks \ref{remstatements} and \ref{rem:efdinvariant}).
  Write $d = \erk_{\Sigma_\infty}(\varphi)$, so that $\varphi$ is equivalent modulo $\Sigma_\infty$ to an $\Lar(C)$-sentence $\varphi_\infty = \exists Y_1, \dotsc, Y_d \psi$ with $\psi$ quantifier-free.
  We may suppose that the set of constant symbols $C$ is finite, by restricting to the constant symbols which occur in $\varphi$ or $\varphi_\infty$.
  By the compactness theorem, there exists $N\in\Natwithoutzero$ such that $\varphi\leftrightarrow\varphi_\infty$ holds in all models of $\Sigma$ of cardinality at least $N$.
  Choose $M > 1$ such that in all fields of cardinality at most $N$ the polynomial $Y^M - Y$ is identically zero.

  For every finite model $K$ of $\Sigma$, there exists an existential $\Lar(C)$-sentence $\varphi_K$ which holds in an $\Lar(C)$-structure $L$ if and only if $K$ can be embedded into $L$.
  We can take $\varphi_K$ to be an $\exists_1$-sentence,
  which follows for example from Corollary \ref{generalcriterion-universal}
  applied to the empty theory,
  as the finite field $K$ is generated by one element as a ring.

  Let $\mathcal{K}$ be the class of (isomorphism classes of) $\Lar(C)$-structures which are fields of positive characteristic at most $M$
%  in which $\forall Y(Y^M\dot=Y)$ holds;
%   generated by elements $x$ satisfying $x^M = x$; 
  generated as rings by roots of $Y^M-Y$;
  there are only finitely many isomorphism classes in $\mathcal{K}$.
  Let $\mathcal{K}_\varphi$ respectively $\mathcal{K}_{\varphi_\infty}$ 
  be the subclass of those $K \in \mathcal{K}$ satisfying $\varphi$ respectively $\varphi_\infty$.
  Note that for $L\models\Sigma$,
\begin{equation}\label{eqn:phiphiinfty}
   L\notin\mathcal{K} \quad\Longrightarrow\quad L\models\exists Y (Y^M \dot\neq Y) \quad\Longrightarrow\quad L\models\varphi\leftrightarrow\varphi_\infty.
\end{equation}
%   is not in $\mathcal{K}$,
%  then $L\models\exists Y (Y^M \dot\neq Y)$,
%  and $L$ is of cardinality greater than $N$
%  and therefore $L\models\varphi\leftrightarrow\varphi_\infty$.
%  In particular, we have $L\models\varphi$
%  as soon as $L\models\varphi_K$
%  for some $K\in\mathcal{K}_\varphi$, 
%  or $L\models\varphi_K$
%  for some $K\in\mathcal{K}_{\varphi_\infty}$
%  and in addition $L\models\exists Y (Y^M \dot\neq Y)$. %$L\notin\mathcal{K}$.
 We now let
 \begin{eqnarray*}
 \varphi_1 &=&  \bigvee_{K \in \mathcal{K}_\varphi} \varphi_K,\\
 \varphi_2 &=& \exists Y (Y^M \dot\neq Y) \wedge \bigvee_{K \in \mathcal{K}_{\varphi_\infty}} \varphi_K,\\
 \varphi_3 &=& \exists Y_1, \dotsc, Y_d \big(\psi \wedge (\bigwedge_{m=1}^M m \dot\neq 0 \vee\bigvee_{i=1}^d Y_i^M \dot\neq Y_i \vee \bigvee_{c \in C} c^M \dot\neq c)\big),
\end{eqnarray*}
and claim that $\varphi$ is equivalent modulo $\Sigma$ to
$\varphi'=\varphi_1\vee\varphi_2\vee\varphi_3$.
By virtue of Remark \ref{qrtrivial} we have
$\erk_\Sigma(\varphi_1)\leq 1$,
$\erk_\Sigma(\varphi_2)\leq 2$,
$\erk_\Sigma(\varphi_3)\leq d$,
and thus
$\erk_\Sigma(\varphi') \leq \max\{2, d\}$, so this will suffice to finish the proof.
So let $L\models\Sigma$.  
  
%  We now claim that $\varphi$ is equivalent modulo $\Sigma$ to the  sentence
%\begin{eqnarray*} \varphi' &=& \exists Y_1, \dotsc, Y_d \big(\psi \wedge (\bigwedge_{m=1}^M m \dot\neq 0 \vee \bigvee_{i=1}^d Y_i^M \dot\neq Y_i \vee \bigvee_{c \in C} c^M \dot\neq c)\big)\\
% &&\vee\;
%    \big(\exists Y (Y^M \dot\neq Y) \wedge \bigvee_{K \in \mathcal{K}_{\varphi_\infty}} \varphi_K \big) \vee
%    \bigvee_{K \in \mathcal{K}_\varphi} \varphi_K
%\end{eqnarray*}
%  Since $\erk_\Sigma(\varphi') \leq \max\{2, d\}$ (Remark \ref{qrtrivial}), this will suffice to finish the proof.

If $L\models\varphi_1$, then $L\models\varphi_K$ for some $K\in\mathcal{K}_\varphi$ and hence $L\models\varphi$.
If $L\models\varphi_2$, then $L\models\exists Y (Y^M \dot\neq Y)$ and $L\models\varphi_K$ for some $K\in\mathcal{K}_{\varphi_\infty}$
and hence $L\models\varphi_\infty$;
by (\ref{eqn:phiphiinfty}) this implies $L\models\varphi$.
If $L\models\varphi_3$, then
$L\models\varphi_\infty$ and $L\notin\mathcal{K}$,
which, again by (\ref{eqn:phiphiinfty}), implies $L\models\varphi$.
Putting everything together we see that $L\models\varphi'\rightarrow\varphi$.

Conversely, let $L\models\varphi$.
If $L\in\mathcal{K}$, then $L\models\varphi_1$.
If $L\notin\mathcal{K}$,
then $L\models\varphi_\infty$ by (\ref{eqn:phiphiinfty}),
so there exists a tuple $\underline{y}$ with $L\models\psi(\underline{y})$;
if the substructure $K$ generated by $\underline{y}$ is in $\mathcal{K}$ (and hence in $\mathcal{K}_{\varphi_\infty}$), then $L\models\varphi_2$, otherwise 
$K$ is either of characteristic greater than $M$ or
either one of the $y_i$ or one of the constants $c$ is not a zero of $Y^M-Y$, and hence $L\models\varphi_3$.
Thus we see that $L\models\varphi\rightarrow\varphi'$.
%
%
%
%
%  First observe that every $K \models \Sigma \cup \{ \varphi' \}$ either embeds some model of $\varphi$ (if it satisfies the third clause of the disjunction), or satisfies $\varphi_\infty$ and additionally has an element $x$ satisfying $x^M \neq x$ or has characteristic zero or greater than $M$ (if it satisfies the first or second clause of the disjunction).
%  Therefore $\Sigma \models \varphi' \rightarrow \varphi$.
%
%  Conversely, every model of $\Sigma \cup \{ \varphi \}$ either embeds a structure in $\mathcal{K}_\varphi$ (and therefore satisfies the third clause), or does not lie in $\mathcal{K}$ and satisfies $\psi(\underline y)$ for some $\underline y$, in which latter case it either satisfies the first or the second clause of the disjunction depending on whether the substructure generated by $\underline y$ lies in $\mathcal K$ or not.
%  Thus $\Sigma \models \varphi \rightarrow \varphi'$.
\end{proof}

For an $\Lar(C)$-theory $\Sigma$ containing $\Tfields$ and a prime number $p$, set $\Sigma_p = \Sigma \cup \lbrace p \doteq 0 \rbrace$. Its models are precisely the models of $\Sigma$ of characteristic $p$.
Recall the $\Lar$-formula $\pii{n}{m}$ from Definition \ref{def:pi}.

\begin{prop}\label{qrleqed}
Let $\Sigma$ be an $\Lar(C)$-theory containing $\Tfields$. 
Suppose that 
$\qr_{\Sigma_p}(\pii{n}{p^m}) \leq 1$
 for every prime number $p$ and every $n, m \in \Natwithoutzero$. Then  
 $$
  \qrp_{\Sigma}(\varphi) \leq \efd_{\Sigma}(\varphi) + 1
 $$ 
 for every existential formula $\varphi$.
\end{prop}
\begin{proof}
By virtue of Proposition \ref{prop:qrvsqrp} it suffices to show that $\qr_{\Sigma}(\varphi) \leq \efd_{\Sigma}(\varphi) + 1$. Furthermore we may restrict again to the case where $\varphi$ is an $\Lar(C)$-sentence. 

Let $K$ be any model of $\Sigma \cup \lbrace \varphi \rbrace$
and let $p$ be the characteristic exponent of $K$. 
By Lemma \ref{rem:defEfdFieldVsRing} there exists a subfield $K'$ of $K$, finitely generated and of transcendence degree $d \leq \efd_{\Sigma}(\varphi)$ over $K_C$ such that $L \models \varphi$ for any $L\models\Sigma$ into which $K'$ embeds.
Fix a transcendence basis $\lbrace t_1, \ldots, t_d \rbrace$ of $K'/K_C$ and set $K_1 = K_C(t_1, \ldots, t_d)$. 
Then $K'/K_1$ is a finite extension. Let $K_2$ be the relative separable closure of $K_1$ in $K'$. We get that $K' = K_2(u_1, \ldots, u_k)$ where $u_1, \ldots, u_k \in K'$ are purely inseparable over $K_2$. Choose $m \in \Natwithoutzero$ such that for each $i \in \lbrace 1, \ldots, k \rbrace$, $v_i = u_i^{p^m} \in K_2$.

By our hypothesis, there exists an $\exists_1$-formula $\psi$ in free variables $X_1, \ldots, X_k$ which is equivalent to $\pii{k}{p^m}$ in all models of $\Sigma$ of characteristic $p$. 
Since $K \models \pii{k}{p^m}(v_1, \ldots, v_k)$, also $K \models \psi(v_1, \ldots, v_k)$. 
Let $w \in K$ be a witness for $\psi(v_1, \ldots, v_k)$ and set $K_3 = K_2(w)$. By the primitive element theorem
in its general form \cite[Section 6.10]{vdW_Algebra}, we may write $K_3 = K_C(t_1, \ldots, t_d, y)$ for some $y \in K_3$. 
Let $R_0$ denote the smallest substructure of $K$
and let $A=R_0[t_1, \ldots, t_d, y]$.

Suppose $L$ is any model of $\Sigma$ into which $A$ embeds.
Then this embedding extends to an embedding of $K_3 = \Frac(A)$ into $L$,
and $L$ also has characteristic exponent $p$. 
After identifying $K_3$ with a subfield of $L$, we have that $L \models \psi(v_1, \ldots, v_k)$ as $w$ is a witness, whence also $L \models \pii{n}{p^m}(v_1, \ldots, v_k)$, whereby 
$$
 K' \subseteq K_3(v_1^{p^{-m}}, \ldots, v_k^{p^{-m}})\subseteq L,
$$ 
so $L \models \varphi$ by the choice of $K'$. 
We have verified condition \textit{(i)} of Proposition \ref{prop:generalcriterion} for being equivalent to an $\exists_{d+1}$-$\Lar(C)$-formula, and this concludes the proof.
\end{proof}

The following corollary should be compared with
Examples \ref{ex:ACF} and \ref{ex:RCF}
where this is stated for 
$\erk$ instead of $\perk$,
but for arbitrary instead of existential $\mathcal{L}_{\rm ring}$-formulas:

\begin{gev}\label{cor:RCFpCF}
If $\Sigma$ is one of ${\rm ACF}$, ${\rm RCF}$ or $p{\rm CF}_d$,
then $\perk(\Sigma)\leq1$.
In particular, $\erk(K)\leq1$ for every $K\models\Sigma$.
\end{gev}

\begin{proof}
As all models of $\Sigma$ are perfect,
the assumption $\erk_{\Sigma_p}(\pii{n}{p^m})\leq1$ is satisfied by
Proposition \ref{prop:pinmperfect}.
The first claim is then obtained by applying Proposition \ref{qrleqed} to Example~\ref{ex:efd0},
and the second claim follows from that via Corollary \ref{cor:erk_model_theory} and (\ref{eqn:erkperk}).
\end{proof}

\begin{stel}\label{ed=qr}
Let $\Sigma$ be an $\Lar(C)$-theory containing $\Tfields$ such that $\qr_{\Sigma_p}(\pii{n}{p^m}) \leq 1$ for every prime number $p$ and every $n, m \in \Natwithoutzero$. Then for any existential formula $\varphi$ we have
\begin{displaymath}
\efd_{\Sigma}(\varphi) = \max \lbrace \qrp_{\Sigma}(\varphi) - 1, 0 \rbrace.
\end{displaymath}
\end{stel}

\begin{proof}
Combine Proposition \ref{qrleqed} and Proposition \ref{edleqqr}.
\end{proof}

\begin{gev}\label{edqrgev}
Let $\Sigma$ be an $\Lar(C)$-theory satisfying the hypotheses of Theorem \ref{ed=qr}. Then for any existential formulas $\varphi_1$ and $\varphi_2$ with $\qrp_{\Sigma}(\varphi_1), \qrp_{\Sigma}(\varphi_2) \geq 1$ we have
\begin{displaymath}
\qrp_{\Sigma}(\varphi_1 \wedge \varphi_2) \leq \qrp_{\Sigma}(\varphi_1) + \qrp_{\Sigma}(\varphi_2) - 1.
\end{displaymath}
\end{gev}

\begin{proof}
Follows from Theorem \ref{ed=qr} via Lemma \ref{edwedge}.
\end{proof}

\begin{vb}\label{ex:perfect}
By virtue of Proposition \ref{prop:pinmperfect}, the theory of perfect fields and any extension thereof, e.g.~the theory of fields of characteristic zero, satisfies the hypotheses of Proposition \ref{qrleqed}, Theorem \ref{ed=qr} and Corollary \ref{edqrgev}.
In the next section we will prove that
the complete theory of any field finitely generated over a perfect field satisfies these hypotheses as well.
\end{vb}

\begin{opm}
If $\Sigma$ is an $\Lar(C)$-theory satisfying the  hypotheses of Theorem \ref{ed=qr} and such that 
all models are infinite or all models are finite, then one may replace $\qrp_\Sigma$ with $\erk_\Sigma$ in Corollary \ref{edqrgev} by virtue of Proposition \ref{prop:qrvsqrp}. Furthermore, if either $\erk_\Sigma(\varphi_1) > 1$ or $\erk_\Sigma(\varphi_2) > 1$, then the assumption that all models of $\Sigma$ are infinite of all models of $\Sigma$ are finite can be dropped by virtue of Proposition \ref{prop:qrvsqrp}. On the other hand, one has the following.
\end{opm}
\begin{vb}\label{vb:weirdErkPerk}
Let $\Sigma$ be the $\Lar$-theory of perfect fields and consider the $\Lar$-sentences $\varphi_1 = \exists Y_1 (Y_1^2 + Y_1 \doteq 1)$, $\varphi_2 = \exists Y_2 \neg (Y_2^4 \doteq Y_2)$, $\varphi = \varphi_1 \wedge \varphi_2$. We claim that $\erk_{K}(\varphi) = 2$. To see that $\erk_{K}(\varphi) > 1$, let $K$ be the perfect hull of $\ff_4(t)$ and note that $K \models \varphi$, but for any subfield $K'$ of $K$ generated by one element over $\ff_2$ and $L$ the perfect hull of $K'$, $L \not\models \varphi$. Now invoke Proposition \ref{prop:generalcriterion} to conclude. This example also shows by Corollary~\ref{edqrgev} that $\qrp_\Sigma(\varphi_2) = 2$, illustrating that case (\ref{item3}) in Proposition \ref{prop:qrvsqrp} can occur.
\end{vb}

\begin{gev}\label{cor:equivalences}
Let $K$ be a field of characteristic $p\geq0$. The following are equivalent:
\begin{enumerate}
\item $p=0$ or $\qr_{K}(\pii{n}{p^m}) \leq 1$ for every $n, m \in \Natwithoutzero$.
\item $\qrp_{K}(\varphi) \leq \efd_{K}(\varphi) + 1$ for every $\exists$-$\Lar(K)$-formula $\varphi$.
\item $\efd_{K}(\varphi) = \max \{ \qrp_{K}(\varphi) - 1, 0 \}$ for every $\exists$-$\Lar(K)$-formula $\varphi$.
\item $\efd_{K}(\varphi) = \max \{ \erk_{K}(\varphi) - 1, 0 \}$ for every $\exists$-$\Lar(K)$-formula $\varphi$.
\item $\qrp_{K}(\varphi_1 \wedge \varphi_2) \leq \qrp_{K}(\varphi_1) + \qrp_{K}(\varphi_2) - 1$
for every $\exists$-$\Lar(K)$-formulas $\varphi_1$ and $\varphi_2$ with $\qrp_{K}(\varphi_1), \qrp_{K}(\varphi_2) \geq 1$.
\item $\erk_{K}(\varphi_1 \wedge \varphi_2) \leq \erk_{K}(\varphi_1) + \erk_{K}(\varphi_2) - 1$
for every $\exists$-$\Lar(K)$-formulas $\varphi_1$ and $\varphi_2$ with $\erk_{K}(\varphi_1), \erk_{K}(\varphi_2) \geq 1$.
\end{enumerate}
\end{gev}

\begin{proof}
The implications from (1) to (2), (3) and (5) follow from
Theorem \ref{ed=qr} and Corollary \ref{edqrgev}
with $\Sigma=\TKprec$.
The equivalence of (3) and (4) is obtained from
Proposition~\ref{prop:qrvsqrp},
since case (\ref{item3}) there cannot occur.
If $\erk_K(\varphi_i)\geq 1$, then $\perk_K(\varphi_i)=\erk_K(\varphi_i)$, again by Proposition \ref{prop:qrvsqrp}, so (5) implies (6).
For the converse, suppose that $p>0$.
Since $\efd_K(\pii{n}{p^m})=0$ (Example \ref{ex:pii_efd}), both (2) and (3) imply 
$\perk_K(\pii{n}{p^m})\leq 1$ and hence (1).
Finally, $\erk_K(\pii{1}{p^m})\leq 1$, and $\pii{n}{p^m}$ is equivalent to
$\pii{1}{p^m}(X_1)\wedge\dots\wedge\pii{1}{p^m}(X_n)$, so (6) implies (1).
\end{proof}

The following proposition gives another example of a theory of fields in which every existential formula has essential fibre dimension zero.
For background on pseudo-algebraically closed (PAC) fields see \cite[Chapter 11]{FJ}.

\begin{prop}\label{prop:efdPerfPAC}
  Let $\Sigma$ be the $\Lar$-theory of perfect PAC fields.
  For any existential $\Lar$-formula $\varphi$, we have $\efd_\Sigma(\varphi) = 0$ and $\qrp_\Sigma(\varphi) \leq 1$.
\end{prop}
\begin{proof}
  Since $\qrp_\Sigma(\varphi) \leq \efd_\Sigma(\varphi) + 1$ by Proposition \ref{qrleqed} and Example \ref{ex:perfect}, it suffices to prove the statement about essential fibre dimension.
  Let $K$ be a perfect PAC field, $\underline x$ a tuple in $K$ with $K \models \varphi(\underline x)$, and $K' \subseteq K$ the relative algebraic closure of the subfield generated by $\underline x$. 
  Since $K$ is perfect, so is $K'$, and thus $K/K'$ is regular.
  By definition of $\efd$, it suffices to show that we have $L \models \varphi(\underline x)$ for every perfect PAC extension field $L/K'$.
  By \cite[Ch.~V §17 Proposition 8]{Bourbaki_AlgebreII}, $K \otimes_{K'} L$ is an integral domain and its fraction field $F$ is a regular extension of $L$.
  Since $F \models \varphi(\underline x)$ as $\varphi$ is existential and $K$ embeds into $F$, 
  and $L$ is existentially closed in its regular extension $F$ by the PAC property
  (cf.~\cite[Proposition 11.1.3]{FJ} and \cite[Proposition 3.1.1]{ErshovMVF}), we have $L \models \varphi(\underline x)$ as desired.
\end{proof}

\begin{opm}
  In contrast to the theories of Example \ref{ex:efd0}, it is not true in general that a relatively algebraically closed subfield of a perfect PAC field is an elementary substructure, or even itself a perfect PAC field.
  A related result to Proposition \ref{prop:efdPerfPAC} is already proved in \cite[Proposition 3.2]{vdDries_RemarkOnAx}, with a proof similar to ours.
  Phrasing the proof in terms of essential fibre dimension seems very natural to us.
\end{opm}

\section{\texorpdfstring{Tuples of $p$-th powers in characteristic $p$}{Tuples of p-th powers in characteristic p}}\label{sect:pthpowers}

\noindent
In this section we 
discuss the existential rank of 
the $\Lar$-formula $\pii{n}{m}$ from Definition~\ref{def:pi},
which in a field $K$ defines the set $(K\pow{m})^n$ of $n$-tuples of $m$-th powers, in the case where $m=p^k$ is a power of the characteristic.
In particular,
we will prove in Theorem \ref{thm:finGenOverPerf} that finitely generated extensions of a perfect field satisfy the equivalent conditions in Corollary \ref{cor:equivalences}, whereas in Proposition \ref{large} it is shown that for imperfect large fields these conditions are never satisfied.

\begin{prop}\label{prop:PToPToTheK}
 Fix a prime number $p$, let $\Sigma$ be an $\Lar$-theory containing $\Tfields\cup\{p\doteq0\}$, %the theory of fields of characteristic $p$,
 and let $n \in \Natwithoutzero$. 
 If $\qr_{\Sigma}(\pii{n}{p}) \leq 1$, then also $\qr_{\Sigma}(\pii{n}{p^k}) \leq 1$ for every $k \in \Natwithoutzero$.
\end{prop}
\begin{proof}
  It suffices to show that $\pii{n}{p^k}$ is equivalent modulo $\Sigma$ to an existential formula with one quantifier.
  Let $\varphi(X_1, \dotsc, X_n, Y)$ be a quantifier-free formula such that \[\Sigma \models \forall X_1, \ldots, X_n ( \pii{n}{p} \leftrightarrow \exists Y \varphi). \]
  We claim that for every $k > 0$, \[ \Sigma \models \forall X_1, \ldots X_n ( \pii{n}{p^k} \leftrightarrow \exists Y \varphi(X_1, \dotsc, X_n, Y^{p^{k-1}})) .\]
  To see this, let $K \models \Sigma$. If $x_1, \dotsc, x_n \in K$ are all $p^k$-th powers, then in particular the elements $x_i' = x_i^{p^{1-k}}$ are all $p$-th powers, so there exists a $y'$ with $K \models \varphi(x_1', \dotsc, x_n', y')$. As $K$ is of characteristic $p$, raising all elements to the $p^{k-1}$-th power is an $\Lar$-endomorphism of $K$, so $K \models \varphi(x_1, \dotsc, x_n, {y'}^{p^{k-1}})$. This proves one implication.

  Assume conversely that $x_1, \dotsc, x_n, y \in K$ are elements with $K \models \varphi(x_1, \dotsc, x_n, y^{p^{k-1}})$; we have to show that the $x_i$ are all $p^k$-th powers. Because of the defining property of $\varphi$, the $x_i$ are all $p$-th powers. If $k=1$, this is already the claim, so suppose $k>1$.
  Then $y^{p^{k-1}}$ is in particular a $p$-th power, and because the $p$-th power map is an injective endomorphism of $K$, we obtain
  \[ K \models \varphi(x_1^{1/p}, \dotsc, x_n^{1/p}, y^{p^{k-2}}) ,\]
  which inductively means that the $x_i^{1/p}$ are all $p^{k-1}$-th powers. Therefore the $x_i$ are all $p^k$-th powers, finishing the proof.  
\end{proof}

\begin{lem}\label{lem:pairsOfPthPowersReduction}
Let $K_0$ be a perfect field of characteristic $p>0$,
$K/K_0$ an extension,
and $\Sigma={\rm Th}_{\Lar(K_0)}(K)$.
Let $h\in K_0[X,Y]$ and $0\neq g\in K_0[X]$ such that for $x,y\in K$
with $g(x)\neq 0$,
$h(x,y)\in K\pow{p}$ implies that $x,y\in K\pow{p}$.
Then $\erk_{\Sigma}(\pii{n}{p})\leq 1$ for every $n\in\Natwithoutzero$.
\end{lem}

\begin{proof}
Since $K_0$ is perfect, the zeros of $g$ in $K$ are in $K\pow{p}$.
Furthermore, we get that $h(x, y) \in K^{(p)}$ for all $x, y \in K\pow{p}$.
Thus $\pii{n}{p}$ is equivalent modulo $\Sigma$ to 
$$ 
 (g(X_n) = 0 \wedge \pii{n-1}{p}(X_1, \dotsc, X_{n-1})) \vee (g(X_n)\neq 0 \wedge \pii{n-1}{p}(X_1, \dotsc, X_{n-2}, h(X_{n}, X_{n-1}))).
$$
Inductively, both parts of the disjunction are equivalent modulo $\Sigma$ to an $\exists_1$-$\Lar(K_0)$-formula, 
hence so is the whole formula (Remark \ref{qrtrivial}).
\end{proof}

The proof of Theorem \ref{thm:finGenOverPerf} below is based on applying the preceding lemma with a polynomial $h$ obtained as a variation of the polynomial $f$ in the following lemma.

\begin{lem}\label{lem:pairsOfPthPowersGoodCase}
  Let $K$ be a field of characteristic $p > 0$ and define the polynomial
  $$
   f(X,Y) = X^{p+1}Y + X^{p+1}Y^{p^2} + X^{2p+1} + X \in \ff_p[X,Y].
$$ 
 Given $x \in K^\times$ and $y \in K$ such that $f(x, y)\in K\pow{p}$, either 
  \begin{enumerate}
  \item $x,y\in K\pow{p}$, or 
  \item $x\notin K\pow{p}$ and for all (Krull) valuations $v$ on $K$ the value $v(x)$ is divisible by $p$ in the value group $v(K^\times)$ (written additively).
  \end{enumerate}
\end{lem}

\begin{proof}
  Since $K\pow{p}$ is a subfield of $K$, if $x \in K\pow{p}$ and $f(x,y) \in K\pow{p}$, then also $x^{p+1}y \in K\pow{p}$ and hence $y \in K\pow{p}$. 

  Assume on the other hand that there exists a valuation $v$ on $K$ such that $v(x)$ is not divisible by $p$.
  Let $z = f(x,y)/x = x^p y + x^p y^{p^2} + x^{2p} + 1$. We claim that $v(z)$ is divisible by $p$.
  To see this, first observe that the three summands $x^p y^{p^2}$, $x^{2p}$ and $1$ have distinct value with respect to $v$, using that $v(x^p)$ is not divisible by $p^2$.
The summand $x^p y$ has valuation strictly larger than one of these three other terms, depending on the signs of $v(x)$ and $v(y)$.
  Therefore $v(z)$ is equal to the value of the unique summand with the smallest value, and therefore divisible by $p$ since $v(x^py^{p^2})$, $v(x^{2p})$ and $v(1)$ are divisible by $p$.

  We deduce that $v(f(x,y))=v(x)+v(z)$ is not divisible by $p$ and therefore $f(x,y)$ is not a $p$-th power. This concludes the proof.
\end{proof}

We view case (2), $x$ not being a $p$-th power but having value divisible by $p$ with respect to all valuations on $K$, as a degenerate situation; indeed we will prove in Lemma \ref{lem:BmodKpfinite} that this only occurs for finitely many $p$-th power classes when $K$ is finitely generated over a perfect field.

\begin{opm}
In a field $K$ of characteristic $p > 0$ where the degenerate case (2) does not occur, we can apply Lemma \ref{lem:pairsOfPthPowersReduction} with $K_0 = \ff_p$, $h$ equal to the polynomial $f$ from Lemma~\ref{lem:pairsOfPthPowersGoodCase} and $g(X) = X$. Hence, in such a field, $\erk_\subTKequiv(\pii{n}{p})\leq 1$ for every $n\in\Natwithoutzero$.

An example is the fraction field $K$ of a unique factorisation domain $R$ with $p$-divisible unit group: If $x \in K^\times$ is such that $v(x)$ is divisible by $p$ for every valuation $v$ on $K$, then in the prime factorisation of $x$, all prime elements of $R$ must occur with multiplicity divisible by $p$. Since $R^\times$ is $p$-divisible, this already forces $x$ to be a $p$-th power.

 This applies for instance to purely transcendental extensions $K$ (of finite or infinite transcendence degree) of perfect fields $K_0$ of characteristic $p$. In particular, Lemma \ref{lem:pairsOfPthPowersGoodCase} yields $\qr_{\subTKequiv}(\pii{n}{p}) \leq 1$ for such $K$, which in the case of finite transcendence degree will be generalized in Theorem~\ref{thm:finGenOverPerf}.
\end{opm}

We now prove that the degenerate case in Lemma \ref{lem:pairsOfPthPowersGoodCase} can be controlled in certain situations.
Let $K$ be a field with $\car(K) = p > 0$.
We write $B(K)$ for the set of $x\in K^\times$ for which $v(x)$ is divisible by $p$ for every discrete valuation $v$ on $K$, by which we mean a Krull valuation with value group isomorphic to $\zz$.
Then $B(K)$ is a subgroup of $K^\times$ containing $\powx{K}{p}$.
We shall show in Lemma \ref{lem:BmodKpfinite} that $B(K)/\powx{K}{p}$ is finite when $K$ is a finitely generated extension of a perfect field $K_0$.

We first need a well-known result from algebraic geometry, which we prove for lack of a reference.
Here and subsequently, for a field $K$, a \emph{$K$-variety} $V$ (or {\em variety over $K$}) is a $K$-scheme of finite type, and $V(K)$ is the set of $K$-rational points, i.e.\ the set of morphisms $\Spec K \to V$ which, composed with the structure morphism $V \to \Spec K$, yield the identity on $\Spec K$.

\begin{lem}\label{lem:finitePicTorsion}
  Let $V$ be a geometrically integral geometrically normal projective variety over a field $K_0$ and $\operatorname{Pic}(V)$ the Picard group of invertible sheaves on $V$. Then for every $n\in\Natwithoutzero$ the $n$-torsion subgroup $\operatorname{Pic}(V)[n]$ is finite.
\end{lem}
\begin{proof}
  Since $\operatorname{Pic}(V)$ injects into the Picard group of the base change of $V$ to the algebraic closure of $K_0$ by \cite[Tag 0CC5]{StacksProject}, we may assume that $K_0$ is algebraically closed.

  Now we use the theory of the Picard scheme.
  The group $\operatorname{Pic}(V)$ is equal to the group of $K_0$-points of the Picard group scheme $\operatorname{Pic}_{V/K_0}$ \cite[Proposition 8.1.4 and Theorem 8.2.3]{BLR}.
  The connected component of the identity of $\operatorname{Pic}_{V/K_0}$ is denoted $\operatorname{Pic}^0_{V/K_0}$.
  The quotient $\operatorname{Pic}_{V/K_0}(K_0)/\operatorname{Pic}^0_{V/K_0}(K_0)$ is a finitely generated group (the Néron-Severi group) by \cite[Theorem 8.4.7]{BLR}. In particular it has finite $n$-torsion.
  On the other hand, $\operatorname{Pic}^0_{V/K_0}(K_0)$ agrees with the group of $K_0$-points of the reduction $(\operatorname{Pic}^0_{V/K_0})_{\rm red}$, which is an abelian variety by \cite[Corollaire 3.2]{Grothendieck_Bourbaki236} since $V$ is normal.
  Therefore $\operatorname{Pic}^0_{V/K_0}(K_0)$ has finite $n$-torsion by the general theory of abelian varieties \cite[Theorem 7.4.38]{Liu}.

  Finally, the short exact sequence \[ 0 \to \operatorname{Pic}^0_{V/K_0}(K_0) \to \operatorname{Pic}_{V/K_0}(K_0) \to \operatorname{Pic}_{V/K_0}(K_0)/\operatorname{Pic}^0_{V/K_0}(K_0) \to 0 \] remains left-exact after taking $n$-torsion, so $\operatorname{Pic}_{V/K_0}(K_0)[n]$ is finite, as desired.
\end{proof}

\begin{lem}\label{lem:BmodKpfinite}
  Let $K$ be a finitely generated extension of a perfect field $K_0$ of characteristic $p>0$.
  Then $B(K)/\powx{K}{p}$ is finite.
\end{lem}
\begin{proof}
  Select a proper integral variety $V_0/K_0$ with function field $K$. By \cite[Theorem 4.1 and Remark 4.2]{deJong}, there exists a projective geometrically integral variety $V$ smooth over a finite extension $L_0/K_0$ such that some open subvariety of $V$ is étale over $V_0$, so that (by étaleness at the generic point) the function field $L$ of $V$ is a finite separable extension of $K$.
  We will proceed to show that $B(L)/\powx{L}{p}$ is finite.
  Since the natural homomorphism $B(K)/\powx{K}{p} \to B(L)/\powx{L}{p}$ is injective, this suffices to prove the claim.

  By Lemma \ref{lem:finitePicTorsion} (observing that $V$ is geometrically normal since it is smooth over $L_0$), the $p$-torsion of $\operatorname{Pic}(V)$ is finite. Since the Picard group of $V$ is isomorphic to the group of Weil divisor classes $\operatorname{Cl}(V)$ as $V$ is regular, see e.g.\ \cite[Corollary 7.1.19 and Proposition 7.2.16]{Liu}, we deduce that $\operatorname{Cl}(V)[p]$ is finite.

  For every $f \in B(L)$, the Weil divisor of $f$ as a rational function on $V$ is the $p$-fold multiple of another divisor. By associating to $f$ this divisor $\frac{1}{p}(f)$, we obtain a homomorphism $d \colon B(L) \to \operatorname{Div}(V) \to \operatorname{Cl}(V)$, whose image is in fact contained in $\operatorname{Cl}(V)[p]$.
  If $f \in B(L)$ is in the kernel of $d$, there exists $g \in L^\times$ such that the divisor of $fg^p$ is trivial, so $fg^p$ has neither zeroes nor poles. As $V/L_0$ is proper, $fg^p$ is in fact algebraic over $L_0$ \cite[Proposition 2.2.22]{Poonen_RationalPoints}, and therefore $f$ itself is a $p$-th power since $L_0$ is perfect.
  We deduce that the kernel of $d$ is contained in $\powx{L}{p}$. Since the image of $d$ is a subset of $\operatorname{Cl}(V)[p]$ and therefore finite, we obtain that $B(L)/\powx{L}{p}$ is finite.
\end{proof}

Below we will need the following result from arithmetic geometry.
\begin{prop}\label{thm:genusChangingFinite}
  Let $K_0$ be a perfect field, $K/K_0$ a finitely generated extension of transcendence degree $1$.
  Let $C/K$ be an irreducible curve such that the genus of $C$ is strictly higher than the genus of the base change $C_{\overline K}$. (Such $C$ is also called \emph{non-conservative}.)
  Then $C(K)$ is finite.
\end{prop}
\begin{proof}
  For algebraically closed $K_0$, this is \cite[Theorem 6]{Jeong}.
  In general, the genus of $C$ is equal to the genus of the base change to the compositum $K\overline{K_0}$ since this is a separable extension \cite[Theorem 2.5.1(c)]{Poonen_RationalPoints},
   so $C(K \overline{K_0})$ is finite by the first case.
\end{proof}

\begin{lem}\label{lem:finitelyManyClassesToFinitelyManyElts}
  Let $K_0$ be a perfect field of characteristic $p>0$, $K$ a finitely generated extension of $K_0$. Let $e = 3$ if $p = 2$ and $e = 2$ otherwise. For every $a \in K^\times$ which is not a $p$-th power, there are only finitely many pairs $(x, y) \in K^2$ with $ay^p = x^e + 1$.
\end{lem}
\begin{proof}
  Extend $a$ to a separating transcendence basis $a = a_1, a_2, \dotsc, a_n$ of $K/K_0$. After taking the compositum of $K$ with the algebraic closure of $K_0(a_2, \dotsc, a_n)$ and replacing $K_0$ by that algebraic closure, we may assume that $K$ is of transcendence degree $1$ over the algebraically closed field $K_0$.

  Now consider the regular projective curve $C$ over $K$ birational to the affine curve described by the equation $aY^p = X^e + 1$. We claim that the genus of $C$ over $K$ is strictly higher than the genus of the base change $C_{\overline K}$ to the algebraic closure.
  In odd characteristic, the equation $aY^p = X^2 + 1$ becomes $X^2 = Y^p - a^{p-1}$ after a change of variables, and the latter is the standard example of a curve that changes genus after base change to the algebraic closure, see for instance 
  \cite[Chapter XV, p.~292]{Artin}. 
  In characteristic $2$, one checks that the curve described by $aY^2 = X^3 + 1$ is birational to a line over $\overline K$, but it has genus $1$ over $K$ -- e.g. apply the Riemann--Hurwitz formula to the field extension $K(C)/K(Y)$, or see
  \cite[Theorem 3]{LangTate}.

  By Proposition \ref{thm:genusChangingFinite}, $C(K)$ is finite. Therefore there are only finitely many $x, y \in K$ with $ay^p = x^e + 1$, as desired.
\end{proof}

\begin{lem}\label{lem:xr+1}
Let $K$ be a finitely generated extension of a perfect field $K_0$ of characteristic $p > 0$. 
There exists $r \in \Natwithoutzero$ not divisible by $p$ such that
\begin{displaymath}
B(K) \cap \lbrace x^r + 1 \mid x \in K \rbrace \subseteq K\pow{p}
\end{displaymath}
\end{lem}
\begin{proof}
  Since the set $B(K)/\powx{K}{p}$ is finite by Lemma \ref{lem:BmodKpfinite}, we may take representatives $z_1, \dotsc, z_m \in K^\times$ for the non-trivial classes in $B(K)/\powx{K}{p}$.
  For $e$ as in Lemma \ref{lem:finitelyManyClassesToFinitelyManyElts}, there exist only finitely many $x \in K$ such that $(x^e+1)/z_i$ is a $p$-th power for some $i$.
  Since none of these finitely many exceptional $x$ can be algebraic over the perfect subfield $K_0$, and $K/K_0$ is finitely generated, there exists an $e' > 0$ not divisible by $p$ such that none of these exceptional $x$ have an $e'$-th root in $K$. Hence there are no $x \in K$ for which $(x^{ee'}+1)/z_i$ is a $p$-th power for any $i$,
  or in other words $x^{ee'}+1\in \bigcup_{i=1}^m z_iK\pow{p}=B(K)\setminus K\pow{p}$.
As such, we may set $r = ee'$.
\end{proof}

\begin{stel}\label{thm:finGenOverPerf}
Let $K_0$ be a perfect field of characteristic $p > 0$ and $K$ a finitely generated extension of $K_0$. 
Then $\qr_{\TKequiv}(\pii{n}{p^m}) \leq 1$ for every $n, m \in \Natwithoutzero$.
\end{stel}

\begin{proof} 
Pick $r \in \Natwithoutzero$ as in Lemma \ref{lem:xr+1}, and fix $n$ and $m$. 
Lemma \ref{lem:pairsOfPthPowersGoodCase} implies that for any $x, y \in K$ such that $x^{r}+1 \neq 0$, if $f(x^{r}+1, y)$ is a $p$-th power, then so are both $x^r + 1$ and $y$, where $f$ is the polynomial from the statement of that lemma.  
As $r$ is not divisible by $p$, this implies that also $x$ is a $p$-th power.
We apply Lemma \ref{lem:pairsOfPthPowersReduction} with $K_0 = \ff_p$, $g(X) = X^r + 1$ and $h(X, Y) = f(g(X), Y)$ to
get that $\erk_{\subTKequiv}(\pii{n}{p})\leq 1$.
Proposition \ref{prop:PToPToTheK} then yields $\qr_{\TKequiv}(\pii{n}{p^m}) \leq 1$.
\end{proof}

\begin{gev}\label{cor:finGenOverPerf}
Let $K$ be finitely generated over a perfect field $K_0$. 
Then for any diophantine sets $D_1,D_2\subseteq K^n$  
with $\Erk_{K}(D_1), \Erk_{K}(D_2) \geq 1$ we have
\begin{displaymath}
\Erk_{K}(D_1\cap D_2) \leq \Erk_{K}(D_1) + \Erk_{K}(D_2) - 1.
\end{displaymath}
\end{gev}

\begin{proof}
This follows from  Theorem \ref{thm:finGenOverPerf},
Lemma \ref{lem:language}
and Corollary \ref{cor:equivalences}.
\end{proof}
  
\begin{opm}
Note that in the proof of Theorem \ref{thm:finGenOverPerf}
the existential formula with one quantifier that we obtained 
depends on the field $K$.
This is in fact necessarily so, 
as Corollary \ref{cor:all_finite_ext} will show.

However, we do obtain, for every $p$ and $n$, 
a sequence $(\psi_s)_{s\in\mathbb{N}}$ of
existential formulas with one quantifier
such that in every field $K$ as in the statement of the theorem,
$\pii{n}{p}$ is equivalent to $\psi_s$ 
for all sufficiently large $s$.
More precisely, for each $r$ not divisible by $p$ one obtains an $\exists_1$-$\Lar$-formula $\varphi_r$
which is equivalent to $\pii{n}{p}$ uniformly in all fields of characteristic $p$ satisfying $B(K) \cap \lbrace x^r + 1 \mid x \in K \rbrace \subseteq K\pow{p}$,
so if $r_p(s)$ denotes the product of the first $s$ positive integers not divisible by $p$, $\psi_s=\varphi_{r_p(s)}$ satisfies the claim.
\end{opm}
\begin{opm}
One can show that $B(K(t)) = B(K)\cdot \powx{K(t)}{p}$ for any field $K$ of characteristic $p$ and a transcendental element $t$, and from this it follows readily that if $K$ satisfies $B(K) \cap \lbrace x^r + 1 \mid x \in K \rbrace \subseteq K\pow{p}$ for a given $r$, then so does any purely transcendental extension of $K$, for the same $r$. 
It follows that $\qr_{\subTKequiv}(\pii{n}{p^k}) \leq 1$ holds also when $K$ is a purely transcendental extension of a finitely generated field extension of a perfect field of characteristic $p$.
\end{opm}

Recall that a field $K$ is {\em large} if the set of $K$-rational points of any smooth geometrically integral $K$-variety is either empty or Zariski-dense.
This is equivalent to $K\prec_\exists K((t))$, see \cite[Proposition 1.1]{Pop}.
The class of large fields is elementary and includes all
henselian nontrivially valued fields
and all PAC fields.
For more on large fields see \cite{BarysorokerFehm}.

\begin{prop}\label{large}
Let $K$ be an imperfect large field of characteristic $p$,
and let $n\in\Natwithoutzero$.
Then $\qr_{K}(\pii{n}{p})=n$.
\end{prop}

\begin{proof}
Assume there exists an $\exists_{n-1}$-$\mathcal{L}_{\rm ring}(K)$-formula $\varphi$
equivalent to $\pii{n}{p}$ modulo $\TKprec$.
Let $K\prec K^*$ be a $|K|^+$-saturated elementary extension \cite[Corollary 10.2.2]{Hodges_Longer}.
In particular, $K^*/K$ is separable (see \cite[Corollary 3.1.3]{ErshovMVF}) and $\trdeg(K^*/K)=\infty$.

Fix $x_1,\dots,x_n\in (K^*)\pow{p}$ algebraically independent over $K$.
Since $K^*\models\varphi(\underline{x})$ there exist $y_1,\dots,y_{n-1}\in K^*$ with $F:= K(\underline{x},\underline{y})\models\varphi(\underline{x})$.
Since $K^*/K$ is separable, so is $F/K$, hence we can choose a separating transcendence basis $z_0,\dots,z_r$ of $F/K$ among $\underline{x},\underline{y}$,
which by $\trdeg(F/K)\geq n$ must contain some $x_i$,
without loss of generality say $z_0=x_1$.
In particular, $x_1\notin F\pow p$.
We will construct a $K$-embedding $\iota\colon F\rightarrow K^*$ with $\iota x_1\notin (K^{*})\pow{p}$, which will give the desired contradiction,
as $\underline{x}\in\varphi(F)$ implies
$\iota \underline{x}\in\varphi(\iota F)$ and hence
$\iota\underline{x}\in\varphi(K^*)=((K^{*})\pow{p})^n$.

Fix $u\in K\setminus K\pow{p}$.
As $F/K(\underline{z})$ is finite and separable,
by the primitive element theorem there exists
$a\in F$ with $F=K(\underline{z})(a)$.
Without loss of generality,
$f:={\rm MinPol}(a/K(\underline{z}))\in K[z_0,\dots,z_r][X]$.
Let $K^\ast[t]$ be the ring of univariate polynomials over $K^\ast$ and set
$$
 g(X):=f(x_1+ut^p,z_1,\dots,z_r,X)\in K^*[t][X].
$$ 
As $g(X)\equiv f(\underline{z},X)\mbox{ mod }t$,
by Hensel's lemma there exists $a'\in K^*[[t]]$ with $g(a')=0$.
Note that $x_1+ut^p,z_1,\dots,z_r$ are algebraically independent over $K$.

Since $K^*\prec_\exists K^*((t))$
and $K^*$ is $|K|^+$-saturated,
the quantifier-free type of $(t, a')$ over $K(z_0, \ldots, z_r)$
is satisfied by a tuple $(\tau, \alpha)$ of $K^*$.
In particular, $\tau\neq 0$, 
the elements $\xi:=x_1+u\tau^p,z_1, \ldots, z_r$ are algebraically independent over $K$, 
and $f(\xi,z_1,\dots,z_r,\alpha)=0$.
Therefore there is a $K(z_1, \ldots, z_r)$-embedding $\iota\colon F\rightarrow K^*$ defined by
$\iota x_1=\xi$ and $\iota a=\alpha$.
As $\tau^p,x_1\in (K^{*})\pow p$ but $u\notin (K^{*})\pow p$ and $\tau\neq 0$, also
$\iota x_1 = \xi = x_1+u\tau^p\notin (K^{*})\pow p$, as desired.
\end{proof}

\begin{opm}
The basic idea for the embedding of $F$ into $K^*$ is taken from \cite[Lemma~4]{Fehm}.
In the case of a henselian nontrivially valued field $K$,
the contradiction
can also be obtained by applying 
\cite[Lemma 20]{Anscombe} to the separable extension
$F/K(x_1)$,
which shows that there is a whole ball around $x_1$
contained in the projection of the set defined by $\varphi$,
but any such ball contains elements that are not in $(K^{*})\pow p$.
\end{opm}

\begin{gev}\label{cor:all_finite_ext}
Let $K$ be an imperfect field of characteristic $p$,
let $m \in \Natwithoutzero$ which is a multiple of $p$, 
and let $n \in \Natwithoutzero$.
There is no $\exists_{n-1}$-$\Lar(K)$-formula equivalent to $\pii{n}{m}$ in all finite separable extensions of $K$.
\end{gev}

\begin{proof}
If the $\exists_{n-1}$-$\Lar(K)$-formula $\varphi(X_1,\dots,X_n)$ 
is equivalent to $\pii{n}{m}$ 
in every finite separable extension $L/K$,
then $\varphi(X_1^{m/p},\dots,X_n^{m/p})$  
is equivalent to $\pii{n}{p}$
in the 
separable closure of $K$,
which is imperfect and PAC \cite[Chapter 11.1]{FJ} (so in particular large),
contradicting Proposition \ref{large}.
\end{proof}

\begin{gev}\label{cor:erkPAC}
Let $K$ be a PAC field. Then
\begin{enumerate}
\item $\erk(K)=0$ if $K$ is algebraically closed,
\item $\erk(K)=1$ if $K$ is perfect but not algebraically closed,
\item $\erk(K)=\infty$ if $K$ is not perfect.
\end{enumerate}
\end{gev}
\begin{proof}
  PAC fields are large, so if $K$ is not perfect, we have $\erk(K) = \infty$ by Proposition~\ref{large}.
  If $K$ is perfect, then $\erk(K) \leq 1$ by Proposition \ref{prop:efdPerfPAC} and Corollary \ref{cor:erk_model_theory}.
  We have $\erk(K) = 0$ if and only if $K$ is algebraically closed by Proposition \ref{prop:acfErkZero}.
\end{proof}

For an imperfect PAC field $K$,
we do not know whether $\efd$ is also unbounded,
or whether the result $\erk(K)=\infty$ is only due to the failure of Theorem \ref{ed=qr} in this situation.
In particular:
\begin{ques}
Let $K$ be separably closed of finite imperfection degree $e$. For diophantine sets $D\subseteq K$, is $\efd_K(D)$ bounded in terms of $e$?
\end{ques}

\section{Essential dimension}
\label{sect:geometric}

\noindent
In this section we connect essential fibre dimension to essential dimension in the sense of Merkurjev. 
This will -- in particular in the case where the background theory is the $\Lar(K)$-theory $\Ext_K$ of fields extending a fixed field $K$ -- yield a correspondence of formulas with families of varieties.

In making this precise, we find it helpful to introduce a functorial point of view.
Let always $K$ be a field and let $\mathrm{Fields}_K$ be the category whose objects are fields extending $K$ and whose morphisms are 
$K$-embeddings, i.e.~ring homomorphisms which are the identity on $K$.

\begin{defi}
Let $\varphi$ be an existential $\Lar(K)$-formula in free variables $X_1,\dots,X_n$.
We define a functor $I_\varphi$ from $\mathrm{Fields}_K$ to the category of sets,  
by mapping $L$ to the set $\varphi(L)$ of tuples $\underline x \in L^n$ with $L \models \varphi(\underline x)$, and mapping a $K$-embedding $L \to L'$ to the induced injection $I_\varphi(L) \to I_\varphi(L')$.
\end{defi}

\begin{opm}\label{rem:same_functor_equivalent_formulas}
Since the models of $\Ext_K$ correspond exactly to the objects of $\mathrm{Fields}_K$, we see that two existential $\Lar(K)$-formulas $\varphi$ and $\varphi'$ in the same free variables are equivalent modulo $\Ext_K$ if and only if the associated functors $I_\varphi$ and $I_{\varphi'}$ are equal.
(We stress that this is equality of functors, as opposed to natural isomorphism.)
This definition is also valid if $n=0$, i.e.~$\varphi$ is an $\Lar(K)$-sentence: in that case, $I_\varphi(L)$ is either empty (if  $L \not\models \varphi$) or the one-element set $L^0$.
\end{opm}

One can also obtain related functors from $\mathrm{Fields}_K$ to sets geometrically.
Recall that by a $K$-variety we mean a $K$-scheme of finite type.
A morphism of $K$-varieties (or more generally of $R$-schemes, where $R$ is a ring) is a morphism of schemes commuting with the structure morphisms to $\Spec(K)$ (or $\Spec(R)$).
We also call this a \emph{$K$-morphism} (respectively \emph{$R$-morphism}).
When $R$ is an arbitrary ring and $V$ an $R$-scheme, and $L$ is a field with a fixed homomorphism $R \to L$, we write $V(L)$ for the set of $L$-rational points of $V$, i.e. the set of $R$-morphisms $\Spec(L) \to V$.
The most important case is when $V$ is a $K$-variety and $L/K$ a field extension.

\begin{defi}
Let $V$ be a $K$-variety with a morphism $f \colon V \to \mathbb{A}^n_K$.
We define a functor $I_f$ from $\mathrm{Fields}_K$ to the category of sets
by mapping a field $L$ to $f(V(L))$, which is a subset of $\mathbb{A}^n_K(L)=L^n$,
and mapping a $K$-embedding $L\to L'$ to the injection $f(V(L))\to f(V(L'))$ coming from 
the injection $V(L)\to V(L')$.
\end{defi}
\begin{opm}\label{rem:restrict_to_affine}
Here we can restrict without loss of generality to reduced $K$-varieties $V$, since the $L$-points of $V$ and its reduction agree.
We may even reduce to the case of affine $V$, by covering $V$ with finitely many affine open $K$-subvarieties $V_i$, $1 \leq i \leq k$,
 and considering the affine variety $\coprod_{i=1}^k V_i$:
The functor $I_f$ associated to a morphism $f \colon V \to \mathbb{A}^n_K$ agrees with the functor associated to the morphism $\coprod_{i=1}^k V_i \to \mathbb{A}^n_K$ obtained from the restriction of $f$ to the $V_i$.
\end{opm}

We observe that the functors obtained from existential formulas and from morphisms of varieties are the same -- this is well-known, but we give a proof.
\begin{lem}\label{lem:formulasEquivVarieties}
  For every existential $\Lar(K)$-formula $\varphi(X_1, \dotsc, X_n)$ there exists a $K$-variety $V$ with a morphism $f \colon V \to \mathbb{A}_K^n$ such that $I_\varphi = I_f$.
  Conversely, for every $K$-variety $V$ and morphism $f \colon V \to \mathbb{A}_K^n$ there exists an existential $\Lar(K)$-formula $\varphi$ with $I_f = I_\varphi$.
\end{lem}
\begin{proof}
  Up to replacing $\varphi$ by a formula which is equivalent modulo $\Ext_K$, we may assume that $\varphi$ is positive primitive, i.e.~of the form 
  $\exists \underline Y \bigwedge_{i=1}^r f_i(\underline X, \underline Y) \doteq 0$ for suitable $f_i \in K[\underline X, \underline Y]$
  (Remark \ref{rem:perk}).
Then the affine $K$-variety $V = \operatorname{Spec}(K[\underline X, \underline Y]/(f_1, \dotsc, f_r))$ with the natural morphism to $\mathbb{A}_K^n = \operatorname{Spec}(K[\underline X])$ is as desired.

For the converse direction, if a given $V$ is affine, one can construct a suitable $\varphi$ by inverting the construction above.
The case of an arbitrary $K$-variety can be reduced to the affine case (Remark \ref{rem:restrict_to_affine}).
\end{proof}

\begin{defi}[{\cite[Section 2a]{Merkurjev_EDSurvey}}]
  Let $I$ be a (covariant) functor from $\mathrm{Fields}_K$ to sets. 
  The \emph{essential dimension} of $I$, denoted $\ed(I)$, is defined to be the infimum over all $n\in\Natwithzero$ such that for every field extension $L/K$ and every $x \in I(L)$ there exists a subextension $L_0/K$ of $L/K$ with $\trdeg(L_0/K) \leq n$ such that $x$ is in the image of $I(L_0)$ under the map $I(L_0 \hookrightarrow L)$.\footnote{In particular, the functor assigning the empty set to every object of $\mathrm{Fields}_K$ has essential dimension zero. The definition in \cite{Merkurjev_EDSurvey} does not cover this case. While an argument can be made that the essential dimension of this functor should be $-\infty$ instead, zero is more convenient for us.}
\end{defi}
The essential dimension of functors arising from morphisms of varieties is briefly treated in \cite[Section 11.d]{Merkurjev_EDSurvey}, although there it is assumed that $V$ is irreducible.
We give some special attention to the essential dimension of the functors associated to $\Lar(K)$-sentences, or equivalently $K$-varieties $V$ with a morphism to $\mathbb{A}_K^0=\operatorname{Spec}(K)$,
which must automatically be the structure morphism of $V$.
\begin{defi}[{\cite[§4]{Merkurjev_EDSurvey}}]
Let $V$ be a $K$-variety.
The {\em canonical dimension} of $V$ is $\cdim(V)=\ed(I_V)$,
where $I_V$ is $I_f$ for the structure morphism 
$f\colon V \to \operatorname{Spec}(K)$.
An irreducible variety $V$ with $\cdim(V)=\dim(V)$ is called {\em incompressible}.
\end{defi}

The following lemma 
reduces the determination of essential dimension and essential fibre dimension to canonical dimension.
A related result is considered in \cite[Theorem 11.4]{Merkurjev_EDSurvey}.

\begin{lem}\label{lem:edEfdFormulaToSentence}
  Let $\varphi$ be an existential $\Lar(K)$-formula in $n$ free variables, 
  and $f \colon V \to \mathbb{A}_K^n$ a morphism of $K$-varieties with $I_f = I_\varphi$.
  Let $D$ be the image of $f$, as a subset of the scheme $\mathbb{A}_K^n$.
  Then
  \begin{equation}\label{eqn:efd}
   \efd_{\Ext_K}(\varphi) \;=\; \sup_{x \in D} \cdim(f^{-1}(x))
  \end{equation} 
  and
  \begin{equation}\label{eqn:ed}
    \ed(I_\varphi) \;=\; \ed(I_f) \;=\; \sup_{x \in D} \big(\dim(\overline{\{x\}}) + \cdim(f^{-1}(x))\big), 
  \end{equation}
  where $f^{-1}(x)$ is the scheme-theoretic fibre of $x$ considered as a $K(x)$-variety,
  and ${\rm dim}(\overline{\{x\}})$ is the dimension of the Zariski closure of the singleton set $\{ x\}$ in $\mathbb{A}_K^n$.
\end{lem}
\begin{proof}
  For a field extension $L/K$ with a tuple $\underline a\in L^n$ such that $\underline a \in f(V(L))$ and therefore $L \models \varphi(\underline a)$, by definition there exists a subextension $F/K(\underline a)$ of $L/K(\underline a)$ of transcendence degree at most $\efd_{\Ext_K}(\varphi)$, such that $F \models \varphi(\underline a)$ and therefore $\underline a \in f(V(F))$.
  This means that $\cdim(f^{-1}(\underline a)) \leq \efd_{\Ext_K}(\varphi)$.
  Conversely, we find that $\efd_{\Ext_K}(\varphi)$ is bounded by the supremum of $\cdim(f^{-1}(\underline a))$, where $\underline a$ varies over $n$-tuples in extensions of $K$ satisfying $\varphi$.
Since the field extensions $K(\underline a)/K$ with fixed $n$-tuples $\underline a$ that satisfy $\varphi$ in some extension $L/K(\underline a)$ are (up to isomorphism over $K$) precisely the residue fields of points in $D$, we obtain (\ref{eqn:efd}).

  For (\ref{eqn:ed}), replacing $\efd_{\Ext_K}(\varphi)$ by $\ed(I_\varphi)$, we need to replace transcendence degrees over $K(\underline a)$ by transcendence degrees over $K$ in the argument above.
  Since transcendence degrees are additive, and for $x \in \mathbb{A}_K^n$ the transcendence degree of $K(x)$ over $K$ is equal to $\dim(\overline{\{x\}})$, we obtain (\ref{eqn:ed}).
\end{proof}

The following easy lemma is also observed in \cite[Theorem 11.4]{Merkurjev_EDSurvey} in a slightly different context.
\begin{lem}\label{lem:edDim}
  For any morphism of $K$-varieties $f \colon V \to \mathbb{A}_K^n$, we have $\ed(I_f) \leq \dim(V)$.
\end{lem}
\begin{proof}
  For a field extension $L/K$, any $L$-rational point of $V$ is already defined over a subextension $F/K$ of transcendence degree at most $\dim(V)$.
\end{proof}

The following lemma about incompressible varieties will be used in Section \ref{sect:lowerBoundsCompleteTheories}.

\begin{lem}\label{lem:incompressibleFunctionField}
  Let $V$ be an incompressible integral $K$-variety.
  Then the function field $K(V)/K$ has no subextension $L/K$ of transcendence degree smaller than $\dim(V)$ with $V(L) \neq \emptyset$.
\end{lem}
\begin{proof}
  Since $\cdim(V) = \dim(V)$, there exists a field extension $F/K$ with $V(F) \neq \emptyset$ such that no subextension $F_0/K$ with $\trdeg(F_0/K) < \dim(V)$ satisfies $V(F_0) \neq \emptyset$.
  Let $x \in V(F)$. Then the residue field $K(x)$ of $x$ is a subfield of $F$ with $V(K(x)) \neq \emptyset$ and is hence of transcendence degree at least $\dim(V)$ over $K$.
  This already forces $x$ to be generic, so $K(x)$ is isomorphic to the function field $K(V)$.
  In particular, $K(V) \subseteq F$ has no subfield $L$ of transcendence degree over $K$ smaller than $\dim(V)$ with $V(L) \neq \emptyset$.
\end{proof}

We now obtain a number of examples where we can compute the essential fibre dimension of a sentence, by referring to earlier work on the canonical dimension of varieties.

\begin{vb}\label{vb:efdQuadric}
  Let $n \in \Natwithoutzero$ and let $q$ be a quadratic form over $K$ in $n+1$ variables.
  Suppose that $q$ is anisotropic (i.e.\ has no non-trivial zero over $K$) and let $V$ be the integral projective hypersurface described by $q$.
  Suppose that the first Witt index $i_1(q)$ (i.e.\ the maximal $K(V)$-dimension of a subspace of $K(V)^{n+1}$ consisting of zeros of $q$) is $1$.
  Then by \cite[Example 4.14]{Merkurjev_EDSurvey} (essentially due to Karpenko--Merkurjev and Totaro), $V$ is incompressible.\footnote{The references in \cite[Example 4.14]{Merkurjev_EDSurvey} are somewhat imprecise.
  One can derive the incompressibility from \cite[Proposition 4.3]{Merkurjev_EDSurvey} and \cite[Corollary 75.6]{ElmanKarpenkoMerkurjev}, where the latter corollary incorporates results from \cite{KarpenkoMerkurjev_EssDimQuadrics} and \cite{Totaro_BiratGeomQuadricsCharTwo}.}
  Hence an existential $\Lar(K)$-sentence $\varphi$ such that a field extension $L/K$ satisfies $\varphi$ if and only if $q$ becomes isotropic over $L$ (i.e.\ $I_\varphi = I_V$) has essential fibre dimension $\efd_\subTKleq(\varphi) = \cdim(V) = n-1$ by Lemma \ref{lem:edEfdFormulaToSentence}.
  One can take for $\varphi$ the $\exists^+_n$-$\Lar(K)$-sentence expressing that one of the $n+1$ dehomogenizations of $q$ has a zero, yielding $\qrp_\subTKleq(\varphi) \leq n$; the latter also follows from Proposition \ref{prop:smoothErk} below, at least if $V$ is smooth (e.g.\ $\car(K) \neq 2$).
  %(if $\car(K) \neq 2$, $\erk_\subTKleq(\varphi) \leq n$ also follows from Proposition \ref{prop:smoothErk}).
 If $n=1$, then since $\varphi$ cannot be quantifier-free,
 this immediately gives $\erk_\subTKleq(\varphi)=\perk_\subTKleq(\varphi)=1$.
 If $n>1$, we get $\perk_\subTKleq(\varphi) \geq n$ by Proposition \ref{edleqqr},
 and then $\erk_\subTKleq(\varphi) = \perk_\subTKleq(\varphi) = n$ by Proposition \ref{prop:qrvsqrp}, since case (\ref{item3}) there cannot occur as a quadratic form anisotropic over a field $L$ is also anisotropic over $L(t)$.
\end{vb}

\begin{vb}\label{vb:edCyclicDivisionAlgebra}
Let $A/K$ be a central division algebra of prime power degree $d = p^n > 1$
and let $V$ be the Severi--Brauer variety associated to $A$.
If $\varphi$ is an existential $\Lar(K)$-sentence with $I_\varphi=I_V$ (Lemma \ref{lem:formulasEquivVarieties}),
then a field extension $L/K$ satisfies $\varphi$ if and only if
$A$ splits over $L$, i.e.~if and only if the base change $A_L$ is isomorphic to the algebra of $d \times d$-matrices over $L$.
By \cite[Theorem 4.3]{Karpenko_IncompressibilitySeveriBrauer}, 
$V$ is incompressible, hence
$\efd_\subTKleq(\varphi) = \cdim(V) = d-1$ by Lemma \ref{lem:edEfdFormulaToSentence}.
Since $V$ is smooth of dimension $d-1$,
it will follow from Proposition \ref{prop:smoothErk} below that $\perk_\subTKleq(\varphi) = d$.
 By Proposition \ref{prop:qrvsqrp}, we also have $\erk_\subTKleq(\varphi) = d$ since the base field $K$ is necessarily infinite by Wedderburn's little theorem.
\end{vb}

The following lemma gives a geometric criterion for a functor $I_f$ to be a subfunctor of another functor $I_g$.
Here a locally closed partition of a $K$-variety is a finite collection of locally closed $K$-subvarieties which is a partition of the underlying topological space.
\begin{lem}\label{lem:varietiesSubfunctor}
  Let $f \colon V \to \mathbb{A}_K^n$ and $g \colon W \to \mathbb{A}_K^n$ be morphisms of $K$-varieties.
  Then the following are equivalent:
  \begin{enumerate}[\quad$(1)$]
  \item $f(V(L)) \subseteq g(W(L))$ for all field extensions $L/K$.
  \item There exist a locally closed partition $V = \bigcup_{i = 1}^k V_i$ and morphisms $V_i \to W$ making the triangle 
  \[ \xymatrix{
      V \ar[dr]_f & V_i \ar@{_{(}->}[l] \ar[r] & W \ar[dl]^g \\
      & \mathbb{A}_K^n &
} \]
commute.
\end{enumerate}
\end{lem}

\begin{proof}
  The implication $(2)\Rightarrow(1)$ is clear, since $f(V_i(L)) \subseteq g(W(L))$ for any $L/K$.

  For $(1)\Rightarrow(2)$, we may replace $V$ by its reduction.
  Let $x \in V$ be the generic point of an irreducible component $S \subseteq V$, and $L = K(x)$ its function field.
  Then $x \in V(L)$, and so $f(x) \in f(V(L)) \subseteq g(W(L))$, so let $y \in W(L)$ with $g(y) = f(x)$.
  Using the correspondence between $L$-points and rational maps from $S$, we obtain an open neighbourhood $U \subseteq S$ of $x$ and a morphism $U \to W$ making the triangle to $\mathbb{A}_K^n$ commute.
  By removing from $U$ the irreducible components of $V$ distinct from $S$, we may assume that $U$ is open in $V$.
  Replace $V$ by the complement of $U$ with the reduced scheme structure, and repeat until we are left with $V = \emptyset$.
  This procedure terminates because $V$ is noetherian as a topological space.
\end{proof}

In the remainder of this section, we give geometric characterisations for essential fibre dimension and essential dimension.
We begin with a characterisation of essential fibre dimension, see Proposition \ref{prop:efdCharacterisationGeneral} and its corollaries.
Since we want this characterisation to work in full generality, we temporarily depart from only working in the theory $\TKleq$ for some field $K$ as in the rest of the section.

When working in $\Lar(C)$-theories of rings, it is useful to consider the free ring $\zz_C = \mathbb{Z}[X_c \colon c \in C]$ on the constants $C$: every $\Lar(C)$-structure which is a ring is naturally a $\zz_C$-algebra, and $\Lar(C)$-homomorphisms between such rings naturally corresponds to $\zz_C$-algebra homomorphisms. Recall that if $K$ is an $\Lar(C)$-structure and a field, $K_C$ is the subfield of $K$ generated by the interpretations of the constant symbols $c \in C$. In other words, it is the fraction field of the image of $\zz_C$ under the canonical homomorphism $\zz_C \to K_C$.
\begin{lem}\label{lem:efd_char}
  Let $\Sigma$ be an $\Lar(C)$-theory containing the theory of fields, 
  let $\varphi$ be an existential $\Lar(C)$-sentence, and let $d = \efd_\Sigma(\varphi)$.
  Then there exists a finitely presented $\zz_C$-algebra $S$ such that
  \begin{enumerate}
  \item $S \otimes_{\zz_C} \Frac(\zz_C/\mathfrak p)$ has Krull dimension at most $d$ for every $\mathfrak p\in{\rm Spec}(\zz_C)$, and 
  \item a field $K \models \Sigma$ satisfies $\varphi$ if and only if there exists a $\zz_C$-homomorphism $S \to K$.
  \end{enumerate}
\end{lem}

\begin{proof}
  Let $\mathcal S$ be the class of finitely presented $\zz_C$-algebras $S$ that satisfy (1)
  and are such that $K \models \varphi$ for every $K \models \Sigma$ with a $\zz_C$-homomorphism $S \to K$.

  Let $K \models \Sigma \cup \{ \varphi \}$, consider $K$ as a $\zz_C$-algebra in the natural way, and write $\mathfrak{p}_0$ for the kernel of $\zz_C \to K$, so that $K_C \cong \Frac(\zz_C/\mathfrak{p}_0)$.
  By Lemma \ref{rem:defEfdFieldVsRing}, there exists a finitely generated $\zz_C/\mathfrak{p}_0$-subalgebra $R$ of $K$, which we may also view as a $\zz_C$-algebra, such that
  $\trdeg(\Frac(R)/K_C) \leq d$
  and $L \models \varphi$ for every $L \models \Sigma$ with a $\zz_C$-homomorphism $R \to L$.
  Note that $\trdeg(\Frac(R)/K_C)$ is precisely the Krull dimension of the finitely generated integral $K_C$-algebra $R \otimes_{\zz_C} K_C$,
  see e.g.~\cite[Theorem 5.22]{GoertzWedhorn}.

  Since the $\zz_C$-algebra $R$ is a direct limit of finitely presented $\zz_C$-algebras \cite[Exercise 10.21]{GoertzWedhorn}, by Lemma \ref{lem:directLimit} applied to the theory $\Sigma \cup \{ \neg\varphi \}$ there is a finitely presented $\zz_C$-algebra $S$ with a homomorphism $S \to R$ such that $L \models \varphi$ for every $L \models \Sigma$ with a $\zz_C$-homomorphism $S \to L$.
  Since $R \otimes_{\zz_C} K_C$ is finitely presented as a $K_C$-algebra as it is finitely generated, and direct limits commute with tensor products, we may even assume that $S \otimes_{\zz_C} K_C \to R \otimes_{\zz_C} K_C$ is an isomorphism.
  
  Write $\mathfrak q$ for the kernel of the composite homomorphism $S \to R \to K$.
  Then $\dim(S \otimes_{\zz_C} K_C) \leq d$ means that the fibre of $\mathfrak q$ of the morphism $\Spec S \to \Spec \zz_C$, i.e.\ the fibre above $\mathfrak p_0 \in \Spec \zz_C$, has dimension at most $d$.
  By semicontinuity of the fibre dimension \cite[Theorem 14.110]{GoertzWedhorn}, after localising at finitely many elements of $S$, we may suppose that all fibres of $\Spec S \to \Spec \zz_C$ have dimension at most $d$, i.e.~$S$ satisfies (1).

  We have shown that for every $K \models \Sigma \cup \{ \varphi \}$, there exists a $\zz_C$-homomorphism $S \to K$ for some $S \in \mathcal S$.
  For every $S \in \mathcal S$ we can find an $\Lar(C)$-sentence $\varphi_S$ such that for every $\Lar(C)$-structure $R$ which is a ring, we have $R \models \varphi_S$ if and only if there is a $\zz_C$-homomorphism $S \to R$: namely, if $S \cong \zz_C[X_1, \dotsc, X_m]/(f_1, \dotsc, f_k)$ for some polynomials $f_i$, then we may take
  \[ \varphi_S = \exists X_1, \dotsc, X_m \bigwedge_{i=1}^k f_i(X_1, \dotsc, X_m) \doteq 0 .\]
  Since we have shown that every $K \models \Sigma \cup \{ \varphi \}$ has a $\zz_C$-homomorphism from some $S \in \mathcal S$, every such $K$ must satisfy some $\varphi_S$.
  The compactness theorem implies that there are in fact finitely many $S_1, \dotsc, S_n \in \mathcal S$ such that every $K \models \Sigma \cup \{ \varphi \}$ has a $\zz_C$-homomorphism from some $S_i$.
  Now $S = S_1 \times \cdots \times S_n$ is a finitely presented $\zz_C$-algebra satisfying all required conditions.
\end{proof}

\begin{prop}\label{prop:efdCharacterisationGeneral}
  Let $\Sigma$ be an $\Lar(C)$-theory containing $\Tfields$.
  %and let $\zz_C = \mathbb{Z}[X_c \colon c \in C]$.
  An existential $\Lar(C)$-formula $\varphi(X_1, \dotsc, X_n)$ has $\efd_\Sigma(\varphi) \leq d$ if and only if there exists a $\zz_C$-scheme $V$ of finite presentation and a $\zz_C$-morphism $f \colon V \to \mathbb{A}_{\zz_C}^n$, all of whose fibres have dimension at most $d$, such that for all $K \models \Sigma$ we have $\varphi(K) = f(V(K))$.
\end{prop}

\begin{proof}
  By adding the free variables $X_1, \dotsc, X_n$ to the language as constants, it suffices to treat the case $n=0$, i.e.\ $\varphi$ is a sentence, and $f \colon V \to \Spec \zz_C$ is the structure morphism.

  If $\efd_\Sigma(\varphi) \leq d$, we may apply Lemma \ref{lem:efd_char} to obtain a finitely presented $\zz_C$-scheme $S$ for which $V = \Spec S$ is as desired: namely, $V$ has a $K$-point for precisely those $K \models \Sigma$ for which $K \models \varphi$, and all fibres of $V$ over $\Spec \zz_C$ are of dimension at most $d$.

  Assume conversely that there is a $\zz_C$-scheme $V$ of finite presentation satisfying the conditions,
  and let $K \models \Sigma \cup \{ \varphi \}$. 
  Write $\mathfrak p$ for the kernel of the natural homomorphism $\zz_C \to K$.
  We have $V(K) \neq \emptyset$, so writing $W = V \otimes_{\zz_C} \Frac(\zz_C/\mathfrak p)$, we have $W(K) \neq \emptyset$, and therefore $W(K') \neq \emptyset$ for some subfield $K'$ of $K$ of transcendence degree at most $\cdim(W) \leq \dim(W) \leq d$ (Lemma \ref{lem:edDim}) over $K_C$.
  Over every field $L \models \Sigma$ with an embedding of $K'$, $W$ and therefore $V$ have a rational point,
  and thus $L\models\varphi$.
  This proves $\efd_\Sigma(\varphi) \leq d$.
\end{proof}

We separately note the special cases of Proposition \ref{prop:efdCharacterisationGeneral} in the case of the theories $\TKleq$ and $\TKprec$
(see Definition \ref{def:theories}).
\begin{gev}\label{cor:formulaToVarietyEfd}
  An existential $\Lar(K)$-formula $\varphi$ has $\efd_{\subTKleq}(\varphi) \leq d$ if and only if $I_\varphi = I_f$ for a morphism $f \colon V \to \mathbb{A}_K^n$ of $K$-varieties with all fibres of dimension at most $d$.
\end{gev}
\begin{proof}
  For necessity, assume that $\efd_{\subTKleq}(\varphi) \leq d$, apply Proposition \ref{prop:efdCharacterisationGeneral} with $\Sigma = \TKleq$, and take the base change of $f$ and $V$ as obtained there along $\zz_C \to K$.
  Sufficiency follows from the first part of Lemma \ref{lem:edEfdFormulaToSentence}: if $I_\varphi = I_f$, then $\efd_{\subTKleq}(\varphi)$ is bounded by the canonical dimension of the fibres, and therefore by the dimension of the fibres (Lemma \ref{lem:edDim}).
\end{proof}

\begin{gev}\label{cor:efdCharacterisationDiag}
  A diophantine set $D \subseteq K^n$ has $\efd_K(D) \leq d$ if and only if $D = f(V(K))$ for a $K$-variety $V$ and a morphism $f \colon V \to \mathbb{A}_K^n$ with all fibres of dimension at most $d$.
\end{gev}
\begin{proof}
  For necessity, apply Proposition \ref{prop:efdCharacterisationGeneral} with $\Sigma = \TKprec$, and take the base change of $f$ and $V$ as obtained there along $\zz_C \to K$.
  For sufficiency, assume that $V$ and $f$ exist as in the statement, and let $\varphi$ be an existential $\Lar(K)$-formula with $I_\varphi = I_f$ (Lemma \ref{lem:formulasEquivVarieties}).
  Then $\efd_K(D)=\efd_{\subTKprec}(\varphi) \leq \efd_{\subTKleq}(\varphi) \leq d$ by
  Lemma \ref{lem:efd_theory} and Corollary \ref{cor:formulaToVarietyEfd}.
\end{proof}

In the same vein as Corollary \ref{cor:formulaToVarietyEfd}, in good situations we also obtain a bound for the existential rank in terms of the dimension of fibres.
Note, however, that this is in general only a sufficient but not a necessary criterion.
\begin{prop}\label{prop:smoothErk}
  Let $f \colon V \to \mathbb{A}_K^n$ be a morphism of $K$-varieties with all fibres of dimension at most $d$.
  Suppose that $V$ can be covered by finitely many locally closed subvarieties $V_i$ such that, letting $W_i$ be the schematic closure of $f(V_i)$ in $\mathbb{A}_K^n$, the induced morphisms $V_i \to W_i$ are smooth.
  Then an $\exists$-$\Lar(K)$-formula $\varphi$ with $I_\varphi = I_f$ has $\perk_\subTKleq(\varphi) \leq d+1$.
\end{prop}
\begin{proof}
  Since $\efd_\subTKleq(\varphi) \leq d$ by Corollary \ref{cor:formulaToVarietyEfd}, it suffices to show that $\erk_\subTKleq(\varphi) \leq d+1$, as then $\perk_\subTKleq(\varphi) \leq d+1$ follows from Proposition \ref{prop:qrvsqrp}.
  Since the disjunction of finitely many $\exists_{d+1}$-$\Lar(K)$-formulas is itself equivalent to an $\exists_{d+1}$-$\Lar(K)$-formula (Remark \ref{qrtrivial}), it suffices to treat the case $V_1=V$, writing $W=W_1$.

  Let $x \in V$.
  According to \cite[Proposition 2.2.11]{BLR}, smoothness implies that there is an open neighbourhood $U \subseteq V$ of $x$ such that $f \colon U \to W$ factors through the projection $W \times_K \mathbb{A}_K^e \to W$, with $U \to W \times_K \mathbb{A}_K^e$ an étale morphism and $e \leq d$ the dimension of the fibre $f^{-1}(f(x))$.

  Because of the local structure of étale morphisms given in \cite[Proposition 2.3.3]{BLR}, after possibly shrinking $U$ to a smaller open neighbourhood of $x$, $U \to W \times_K \mathbb{A}_K^e$ factors through a locally closed embedding $U \hookrightarrow W \times_K \mathbb{A}_K^{e+1}$.
  Write $H \subseteq W \times_K \mathbb{A}_K^{e+1} \subseteq \mathbb{A}_K^{n+e+1}$ for its image, which is locally closed and therefore defined by a quantifier-free $\Lar(K)$-formula $\psi_H(X_1, \dotsc, X_{n+e+1})$, in the sense that over any field $L/K$, $\psi_H$ is satisfied by precisely those tuples in $L$ which lie in $H(L)$.
  Letting $\varphi_U(X_1, \dotsc, X_n) = \exists X_{n+1}, \dotsc, X_{n+e+1} \psi_H$, we therefore have $I_{f|_U} = I_{\varphi_U}$, and $\erk_\subTKleq(\varphi_U) \leq e+1 \leq d+1$.

  Varying the point $x \in V$ we start with and using the compactness of $V$, we can find finitely many $U_1, \dotsc, U_k$ as above covering $V$.
  Then an existential formula $\varphi$ with $I_\varphi = I_f$ is equivalent modulo $\TKleq$ to the disjunction of the $\varphi_{U_i}$, and therefore equivalent to an $\exists_{d+1}$-$\Lar(K)$-formula (Remark \ref{qrtrivial}).
\end{proof}

\begin{opm}
  If the base field $K$ is of characteristic zero, then for every morphism $f \colon V \to \mathbb{A}_K^n$ we can find a stratification of $V$ into pieces $V_i$ on which $f$ is smooth as in the hypothesis of Proposition \ref{prop:smoothErk}: This follows by first passing to the reduction of the irreducible components of $V$ and then repeatedly using generic smoothness (see for instance \cite[Exercise 6.2.9]{Liu}), where the crucial assumption of characteristic zero yields that $f$ is smooth at generic points.

  In this situation, Corollary \ref{cor:formulaToVarietyEfd} and Proposition \ref{prop:smoothErk} thus yield a proof within the geometric paradigm that $\perk_\subTKleq(\varphi) \leq \efd_\subTKleq(\varphi) + 1$ for any $\exists$-$\Lar(K)$-formula $\varphi$.
  We already obtained this result much earlier in 
  Proposition \ref{qrleqed}
  % Theorem \ref{ed=qr}
  using different methods and in greater generality.
\end{opm}

We next aim to characterise essential dimension in a similar manner as essential fibre dimension (Corollary \ref{cor:edCharacterisation}).
We begin with the following proposition, which also leads to the independently interesting Corollary \ref{cor:finDisjEdEfdDim}.

\begin{prop}\label{prop:unionFlatIncompressible}
  Let $V$ be a $K$-variety and $f \colon V \to \mathbb{A}_K^n$ a morphism.
  Then there are finitely many morphisms $g_i \colon V_i \to \mathbb{A}_K^n$ of $K$-varieties, such that the functor $I_f$ is equal to the functor $I_g$ associated to the disjoint union $g \colon \coprod_i V_i \to \mathbb{A}_K^n$,
  where the $g_i$ can be chosen such that their image $U_i$ is irreducible and locally closed in $\mathbb{A}_K^n$, and $V_i \to U_i$ is flat with generic fibre $(V_i)_\eta$ satisfying $\cdim((V_i)_\eta) = \dim((V_i)_\eta)$, and further the $U_i$ are disjoint.
\end{prop}
\begin{proof}
  Let $D$ be the image of $f$, as a subset of the scheme $\mathbb{A}_K^n$.
  By Chevalley's theorem, $D$ is constructible.
  Furthermore, $D$ is a noetherian topological space  \cite[Lemma 1.25(1)]{GoertzWedhorn}.
  Let $S_0 \subseteq D$ be an irreducible component of the topological space $D$, $S$ its closure in $\mathbb{A}_K^n$ with the reduced scheme structure,  $x \in S$ the generic point, and $W=f^{-1}(x)$ the fibre of $x$.
  
  Applying Lemma \ref{lem:formulasEquivVarieties} to pass from $W$ to an $\exists$-$\Lar(K(x))$-sentence $\varphi_W$ with $I_W = I_{\varphi_W}$ and therefore $\efd_{\subTKleq}(\varphi_W) = \cdim W$ by Lemma \ref{lem:edEfdFormulaToSentence}, and then using Corollary \ref{cor:formulaToVarietyEfd}, we obtain a $K(x)$-variety $W'$ of dimension $\dim W' = \cdim W$ with $I_W = I_{W'}$ as functors on $\mathrm{Fields}_{K(x)}$.
  By Lemma \ref{lem:varietiesSubfunctor} twice, this means that after possibly replacing $W'$ by the disjoint union of a locally closed partition of $W'$, we have a $K(x)$-morphism $W' \to W$ and a locally closed partition $W = W_1 \cup \dotsb \cup W_m$ with $K(x)$-morphisms $W_i \to W'$.
  Now we spread everything out as in \cite[Theorem 3.2.1]{Poonen_RationalPoints}: 
  There exists a dense open subset $U \subseteq S$ and a $K$-variety $Z'$ with a morphism $Z' \to U$ as well as a locally closed partition $Z_1 \cup \dotsb \cup Z_m$ of $Z = f^{-1}(U)$ and $U$-morphisms $Z' \to Z$ and $Z_i \to Z'$, such that $Z'$ and the $Z_i$ have $W'$ and the $W_i$ as their fibres over $x$.
  Shrinking $U$ to a possibly smaller dense open subset of $S$, we may assume that $Z' \to U$ is flat by generic flatness (see for instance \cite[Corollary 10.85]{GoertzWedhorn}), and that $U \cap D$ is open in $D$ (as opposed to only open in $S$), by removing from $U$ the finitely many irreducible components of $D$ distinct from $S_0$.

  Write $V_1$ for $Z'$ and $g_1$ for the morphism $V_1 = Z' \to U$.
  By construction, the functor $I_{g_1}$ is equal to the functor $I_{f|_Z}$:
  by two applications of Lemma \ref{lem:varietiesSubfunctor}, we have both $I_{g_1}(L) \subseteq I_{f|_Z}(L)$ and vice versa for every $L/K$.
  Thus $I_f$ is equal to the functor associated to the morphism $V_1 \cup (V \setminus f^{-1}(U)) \to \mathbb{A}_K^n$ given by $g_1$ and the restriction of $f$, where $V \setminus f^{-1}(U)$ is the complement of the preimage of $U$ with the reduced scheme structure.

  Now pass from $f \colon V \to \mathbb{A}_K^n$ to the restriction $V \setminus f^{-1}(U) \to \mathbb{A}_K^n$, which removes $U$ from $D$, and repeat the process.
  This terminates after finitely many steps since $D$ is noetherian.
  By construction, the functor associated to the morphism $g$ constructed by gluing the morphisms $g_i$ thus obtained is equal to the functor $I_f$.
\end{proof}

\begin{gev}\label{cor:finDisjEdEfdDim}
  Every existential $\mathcal{L}_{\rm ring}(K)$-formula $\varphi$ is equivalent modulo $\Ext_K$ to a finite disjunction of existential $\mathcal{L}_{\rm ring}(K)$-formulas $\varphi_i$ which have 
  $$
   \ed(I_{\varphi_i}) \;=\; \efd_{\Ext_K}(\varphi_i) + \dim(\varphi_i(\overline{K}))
$$ 
 and are mutually exclusive in the sense that no tuple in a field extension $L/K$ can satisfy both $\varphi_i$ and $\varphi_j$ for some indices $i \neq j$.
  Here we write $\dim(\varphi_i(\overline{K}))$ for the dimension of the set of tuples (with the Zariski topology) defined by $\varphi_i$ over the algebraic closure $\overline{K}$.
\end{gev}
\begin{proof}
  Assume first that the functor $I_\varphi$ is equal to the functor $I_g$ corresponding to a morphism $g \colon V \to U \hookrightarrow \mathbb{A}_K^n$ with $U$ irreducible and locally closed in $\mathbb{A}_K^n$, and $V \to U$ flat with generic fibre $V_\eta$ satisfying $\cdim(V_\eta) = \dim(V_\eta)$.
  By lower semicontinuity of the dimension of the fibres of $V \to U$ as a function on $U$ \cite[Tag 0D4H]{StacksProject}, all fibres have dimension at most $\dim V_\eta$,
  so in particular $\dim V_\eta \geq \dim V - \dim U$ \cite[Proposition 14.107]{GoertzWedhorn}.
  On the other hand \[\dim V_\eta = \cdim V_\eta \leq \ed(I_g) - \dim U \leq \dim V - \dim U,\]
  where the first inequality follows from Lemma \ref{lem:edEfdFormulaToSentence} and the second is Lemma \ref{lem:edDim}.
  Therefore $\dim V_\eta = \dim V - \dim U$, and Lemma \ref{lem:edEfdFormulaToSentence} implies $\efd_\subTKleq(\varphi) = \dim V - \dim U$ and $\ed(I_\varphi) = \dim V$, so the result holds in this case.
  
  In general, we pass from $\varphi$ to a morphism $f \colon V \to \mathbb{A}_K^n$ of $K$-varieties with $I_f = I_\varphi$ using Lemma \ref{lem:formulasEquivVarieties}, and then apply Proposition \ref{prop:unionFlatIncompressible} to $f$.
  We obtain a morphism $g \colon \coprod_i V_i \to \mathbb{A}_K^n$ with $I_g = I_f = I_\varphi$, and the restriction $g_i$ of $g$ to $V_i$ satisfies the conditions of the first case.
  Let $\varphi_i$ be a formula with $I_{\varphi_i} = I_{g_i}$, afforded by Lemma \ref{lem:formulasEquivVarieties}.
  Then $\varphi$ is equivalent modulo $\Ext_K$ to the disjunction of the $\varphi_i$ (Remark \ref{rem:same_functor_equivalent_formulas}), the $\varphi_i$ are mutually exclusive since the $g_i$ have disjoint images in $\mathbb{A}_K^n$, and each $\varphi_i$ is as desired by the first case.
\end{proof}

\begin{vb}
  In Corollary \ref{cor:finDisjEdEfdDim} it is necessary to allow finite disjunctions.
  Consider for instance an existential $\Lar(K)$-sentence $\varphi$ with $\efd_{\Ext_K}(\varphi) > 0$.
  Then the one-variable formula $\psi(X) = (\varphi \vee X \dotneq 0)$ has $\ed(I_\psi) = \efd_{\Ext_K}(\psi) = \efd_{\Ext_K}(\varphi)$ and $\dim(\psi(\overline K)) = 1$.
\end{vb}

\begin{gev}\label{cor:edCharacterisation}
  An existential $\Lar(K)$-formula $\varphi$ has $\ed(I_\varphi) \leq d$ if and only if there exists a $K$-variety $V$ of dimension at most $d$ with a $K$-morphism $f \colon V \to \mathbb{A}_K^n$ such that $I_\varphi = I_f$.
\end{gev}
\begin{proof}
  For sufficiency, it suffices to note that by Lemma \ref{lem:edDim} the functor $I_f$ associated to a morphism from a variety of dimension at most $d$ has essential dimension at most $d$.

  For necessity, assume first that $\varphi$ satisfies $\ed(I_\varphi) = \efd_{\TKleq}(\varphi) + \dim(\varphi(\overline{K}))$ as in Corollary~\ref{cor:finDisjEdEfdDim}.
  By Corollary \ref{cor:formulaToVarietyEfd}, $I_\varphi = I_f$ for some morphism $f \colon V \to \mathbb{A}_K^n$ with fibres of dimension at most $\efd_\subTKleq(\varphi)$.
  Since the image of $f$ has dimension $\dim(\varphi(\overline{K}))$, it follows that $V$ has dimension at most $\efd_\subTKleq(\varphi) + \dim(\varphi(\overline{K})) = \ed(I_\varphi)$, as desired.

  In the general case, we write $\varphi$ as a mutually exclusive disjunction $\varphi_1 \vee \dotsb \vee \varphi_m$  according to Corollary \ref{cor:finDisjEdEfdDim}
 such that each $\varphi_i$ satisfies the previous condition, 
 obtain for each $i$ a morphism $f_i \colon V_i \to \mathbb{A}_K^n$  with $I_{\varphi_i} = I_{f_i}$ and $\dim(V_i) \leq \ed(I_{\varphi_i})$, and construct $V$ as the disjoint union of the $V_i$, with $f \colon V \to \mathbb{A}_K^n$ given by patching together the $f_i$.
  Since the $\varphi_i$ are mutually exclusive, the functor $I_\varphi$ is (pointwise) the disjoint union of the $I_{\varphi_i}$, which implies that
  \[ \ed(I_\varphi) = \max_i \ed(I_{\varphi_i}) \geq \max_i \dim(V_i) = \dim(V) \]
  using \cite[Lemma 1.10]{BerhuyFavi_EssDim}.
\end{proof}

\section{Lifting lower bounds to complete theories}
\label{sect:lowerBoundsCompleteTheories}

\noindent
In the last section, we saw in Examples \ref{vb:efdQuadric} and \ref{vb:edCyclicDivisionAlgebra} how knowledge of the canonical dimension of certain varieties over a field $K$ allows us to compute essential fibre dimension and existential rank of corresponding sentences with respect to the theory $\TKleq$.
However, this alone does not suffice to determine interesting examples of existential ranks with respect to any \emph{complete} theory of fields.
We rectify this in the present section by means of a general construction (Proposition \ref{prop:lowerBoundsLimit}), leading us to concretely determine existential ranks in a fixed field $L$ in several instances.

Let $K$ always be a field.

\begin{lem}\label{lem:free_amalgamation}
Let $(K_i)_{i\in I}$ be a family of regular extensions of $K$.
There exists an extension $L/K$ and $K$-embeddings $\iota_i\colon K_i\rightarrow L$ such that the family $(\iota_iK_i)_{i\in I}$ is linearly disjoint over $K$ with compositum $L$.
\end{lem}

\begin{proof}
Take $L=\varinjlim {\rm Frac}(K_{i_1}\otimes_K\dots\otimes_K K_{i_n})$,
where the direct limit runs over finite subsets $\{i_1,\dots,i_n\}$ of $I$.
Indeed, as each $K_{i_j}$ is regular over $K$,
$K_{i_1}\otimes_K\dots\otimes_K K_{i_n}$
is an integral domain by \cite[Ch.~V §17 Proposition 8]{Bourbaki_AlgebreII}.
The embedding $\iota_i$ is obtained from the natural embedding of $K_i$ into $K_{i_1}\otimes_K\dots\otimes_K K_{i_n}$ for $i\in\{i_1,\dots,i_n\}$.
\end{proof}

The proof of the following lemma follows \cite[Corollary 3.1.4]{ErshovMVF} %\cite[Cor. 3 to Prop. 1]{ErshovRRCF}, 
where this is stated in the case that $L/K$ is algebraic.

\begin{lem}\label{lem:ec_base_change}
Let $F$ and $L$ be linearly disjoint extensions of $K$ with $K\prec_\exists F$. Then $L\prec_\exists FL$.
\end{lem}

\begin{proof}
As $K\prec_\exists F$, there is a $K$-embedding 
$\iota\colon F\rightarrow K^*$ of $F$ into the ultrapower
$K^*=K^I/\mathcal{F}$ for some $I$ and some ultrafilter $\mathcal{F}$ on $I$ \cite[Exercise 9.5.12]{Hodges_Longer}. 
The corresponding ultrapower $L^*=L^I/\mathcal{F}$ contains both $L$ and $K^*$, hence also $\iota(F)L$, so $L\prec L^*$ implies that $L\prec_\exists\iota(F)L$.
As subfields of $L^*$, the fields $K^*$ and $L$ are linearly disjoint over $K$,
since a $K^*$-linear dependence of a subset of $L$ 
gives, in at least one of the factors of the ultrapower, a $K$-linear dependence of the same set.
In particular, $\iota(F)$ and $L$ are linearly disjoint over $K$,
so since also $F$ and $L$ are linearly disjoint over $K$,
$\iota(F)L\cong_L FL$, proving the claim.
\end{proof}

\begin{lem}\label{lem:ecc}
Let $L$ be the compositum of a linearly disjoint family $(K_i)_{i\in I}$ of extensions of $K$ with $K\prec_\exists K_i$ for each $i$.
Then $K_i\prec_\exists L$ for each $i$.
\end{lem}

\begin{proof}
Without loss of generality $I=\{1,\dots,n\}$ is finite
and we want to show that $K_1\prec_\exists K_1\cdots K_n$.
By transitivity of $\prec_\exists$
it suffices to show that $K':=K_1\cdots K_i\prec_\exists K'K_{i+1}$.
As $K\prec_\exists K_{i+1}$ and $K'$, $K_{i+1}$ are linearly disjoint over $K$, this follows from Lemma \ref{lem:ec_base_change}.
\end{proof}

Recall from Definition \ref{def:theories} that $\TKec$ is the $\Lar(K)$-theory of extension fields of $K$ in which $K$ is existentially closed.
A universal-existential $\Lar(K)$-sentence (see Remark \ref{rem:effective}) will also be called an $\forall\exists$-$\Lar(K)$-sentence.
%\begin{prop}\label{prop:lowerBoundsCompositum}
%  For every field $K$ there exists a field $K\prec_\exists L$ such that for any two $\exists$-$\Lar(K)$-formulas $\varphi(\underline{X})$, $\psi(\underline{X})$ with $\TKec \not\models \forall \underline{X}(\varphi \rightarrow \psi)$ one has $\varphi(L) \nsubseteq \psi(L)$.
%  In particular, 
%  with $\Sigma={{\rm Th}_{\Lar(K)}(L)}$,
%  $\qr_\Sigma(\varphi) = \qr_{\subTKec}(\varphi)$ for every $\exists$-$\Lar(K)$-formula $\varphi$.
%\end{prop}
%
%\begin{proof}
%  % For any $\exists$-$\Lar(K)$-formula $\varphi$, let $n_\varphi$ be the number of free variables of $\varphi$.
%
%This proves the first claim.
%In particular, any two $\exists$-$\Lar(K)$-formulas $\varphi$, $\psi$ which are equivalent modulo $\Sigma$ are already equivalent modulo $\TKec$, implying that $\qr_\Sigma(\varphi) \geq \qr_\subTKec(\varphi)$ for any $\varphi$,
%and the other inequality is given by Lemma \ref{lem:language} as $\TKec\subseteq\Sigma$.
%\end{proof}

\begin{prop}\label{prop:lowerBoundsLimit}
For every field $K$ there exists a field $K \prec_\exists L$ such that 
$\Tprec{L}$ and $\Tec{L}$ imply the same 
$\forall\exists$-$\Lar(L)$-sentences.
%for every $\forall\exists$-$\Lar(L)$-sentence $\varphi$
%one has $\Tec{L} \models \varphi$ if and only if $\Tprec{L} \models \varphi$.
  In particular, 
  $\qr_L(\varphi) = \qr_{\subTec{L}}(\varphi)$ for every $\exists$-$\Lar(L)$-formula $\varphi$.
\end{prop}

\begin{proof}
We first show that
for every field $K$ there exists a field $K\prec_\exists K'$ such that for any $\forall\exists$-$\Lar(K)$-sentence $\varphi$ with $\TKec \not\models \varphi$ one has $K' \not\models \varphi$.
 Write $P$ for the set of $\forall\exists$-$\Lar(K)$-sentences $\varphi$ such that $\TKec \not\models \varphi$. 
  For every $\varphi \in P$, choose $F_{\varphi}\models\TKec$ such that $F_{\varphi} \not\models \varphi$.
  As $K$ is existentially closed in all $F_{\varphi}$, in particular all $F_{\varphi}$ are regular extensions of $K$  
  \cite[Corollary 3.1.3]{ErshovMVF}.
By Lemma \ref{lem:free_amalgamation} there exists 
an extension $K'/K$ with $K$-embeddings $\iota_{\varphi} \colon F_{\varphi}\rightarrow K'$ for all $\varphi \in P$ such that the family $(\iota_{\varphi}(F_{\varphi}))_{\varphi \in P}$ is linearly disjoint over $K$ with compositum $K'$.
By Lemma \ref{lem:ecc} we have $\iota_{\varphi}(F_{\varphi}) \prec_\exists K'$ for all $\varphi \in P$; in particular $K \prec_\exists K'$.
If we now had $K' \models \varphi$ for some $\varphi \in P$, then $F_{\varphi} \cong \iota_{\varphi}(F_{\varphi}) \prec_\exists K'$ would imply that $F_{\varphi} \models \varphi$ in contradiction to the choice of $F_{\varphi}$.
 
  We now iterate this construction to obtain a chain of fields $K=K_0\prec_\exists K_1 \prec_\exists \ldots$ 
  such that for every $i$ 
  and every  $\forall\exists$-$\Lar(K_{i-1})$-sentence $\varphi$ for which   $\Tec{K_{i-1}} \not\models \varphi$
  one has $K_i \not\models \varphi$.
%  Such a field is afforded by Proposition \ref{prop:lowerBoundsCompositum}.
  Let $L$ be the direct limit of $(K_i)_{i \in \Natwithoutzero}$. Clearly $K \prec_\exists L$; we will show that this $L$ is as desired.
  To this end, let $\varphi$ be an $\forall\exists$-$\Lar(L)$-sentence.
  Since $\Tec{L}\subseteq\Tprec{L}$ (Remark \ref{rem:theories}), 
  it is trivial that $\Tec{L}\models\varphi$ implies $\Tprec{L}\models\varphi$, so assume that $\Tec{L} \not\models \varphi$.
%  The inequality $\erk_L(\varphi) \leq \erk_{\subTec{L}}(\varphi)$ is immediate from Lemma \ref{lem:language}.
%  Suppose for the sake of a contradiction that $\psi$ is an $\exists_m$-$\Lar(L)$-formula for some $m < \erk_{\subTec{L}}(\varphi)$ such that $\psi(L) = \varphi(L)$.
  Since $L$ is the direct limit of $(K_i)_{i \in \Natwithoutzero}$, there exists some $i \in \Natwithoutzero$ such that 
    $\varphi$ is an $\Lar(K_i)$-formula, and since $K_i \prec_\exists L$, also $\Tec{K_i} \not\models \varphi$. %, and we have $m < \erk_{\subTec{L}}(\varphi) \leq \erk_{\subTec{K_i}}(\varphi)$ by Lemma \ref{lem:language}
  But then by construction $K_{i+1} \not\models \varphi$, and since $K_{i+1} \prec_\exists L$, also $L \not\models \varphi$,
  i.e.~$\Tprec{L}\not\models\varphi$.
  This proves the first claim,
  and the `in particular' part follows by translating equivalence of existential formulas into a universal-existential sentence as in Remark \ref{rem:effective}.
 %  In particular, any two $\exists$-$\Lar(L)$-formulas $\varphi, \psi$ for which $\varphi(L) = \psi(L)$ are already equivalent modulo $\Tec{L}$, implying that $\erk_L(\varphi) \geq \erk_{\subTec{L}}(\varphi)$, and the other inequality is given by Lemma \ref{lem:language}.
\end{proof}

\begin{stel}\label{stel:erk_quad_cyclic_norm}
Let $K$ be an infinite field and $K\prec_\exists L$ as in Proposition \ref{prop:lowerBoundsLimit}. 
Then for every $m\in\Natwithoutzero$ the following hold:
\begin{enumerate}
\item $\erk_L(q(L^m)) = m$ for every anisotropic quadratic form $q\in L[Y_1,\dots,Y_m]$.
\item\label{it:cyclic} If $m=p^r > 1$ is a prime power, then $\erk_L(\operatorname{N}_{M/L}(M^\times)) = m$ for every cyclic Galois extension $M/L$ of degree $m$. %such that $\operatorname{N}_{M/L}(M^\times) \nsubseteq L^\times,$
\end{enumerate}
\end{stel}
\begin{proof}
%Let $L$ be as constructed in Proposition \ref{prop:lowerBoundsLimit}, consider $m \in \Natwithoutzero$, and set 
Since $L$ is infinite we have that $L \prec_\exists F:= L(t_1, \ldots, t_m)$, cf.~\cite[Example 3.1.2]{ErshovMVF}.
For $f\in L[Y_1, \ldots, Y_m]$ of total degree at least $2$,
let $t = f(t_1, \ldots, t_m)\in F$ and
$$
 \sigma_f(X) = \exists Y_1, \ldots, Y_m (X \stackrel{.}{=} f(Y_1, \ldots, Y_m)).
$$ 
Then $F \models \sigma_f(t)$, and
if we can show that $L' \not\models \sigma_f(t)$ for every subextension $L'$ of $F/L(t)$ of transcendence degree $m-2$, 
then $\efd_{\subTec{L}}(\sigma_f) \geq m-1$. 
Propositions \ref{edleqqr} and \ref{prop:qrvsqrp} then yield that $m \leq \erk_{\subTec{L}}(\sigma_f) \leq \erk_L(\sigma_f) = \erk_L(f(L^m))\leq m$, unless possibly $m = 1$ and $\qr_{\subTec{L}}(\sigma_q) = 0$, but this latter case is impossible, since for $m=1$ the formula $\sigma_f$ defines an infinite co-infinite set in $L(t) \models \Tec{L}$ and is therefore not equivalent to a quantifier-free formula.
We will show that if $f$ is either a quadratic form in $m$ variables or the norm form of a cyclic Galois extension of prime power degree $m$, then this condition is satisfied, concluding the proof.

In the case where $f=q$ is an anisotropic quadratic form, consider the quadratic form
$$
 q'(Y_1, \ldots, Y_m, Z) = q(Y_1, \ldots, Y_m) - tZ^2 \in L(t)[Y_1,\dots,Y_m,Z],
$$ 
and note that $F$ is isomorphic to the function field of 
the integral projective hypersurface described by $q'$ over $L(t)$; in particular $q'$ is isotropic over $F$.
We claim that the first Witt index $i_1(q') = 1$.\footnote{If $\car(K) \neq 2$, this is \cite[Lemma 11.9]{Merkurjev_EDSurvey}.} 
  Suppose instead that $i_1(q') > 1$. 
  Then there exists a $2$-dimensional $F$-subspace of $F^{m+1}$ consisting of zeros of $q'$. 
  This subspace must non-trivially intersect $F^m \times \lbrace 0 \rbrace$, yielding a non-trivial zero of $q$ over $F$, 
  thereby contradicting the fact that $q$ is anisotropic over $L$ and $L \prec_\exists F$. 
  Hence $i_1(q') = 1$. 
  By Example \ref{vb:efdQuadric} and Lemma \ref{lem:incompressibleFunctionField}
  $F/L(t)$ has no subextension of transcendence degree less than $m-1$ over which $q'$ is anisotropic.
  Since $q'$ is isotropic over an intermediate field $L'$ of $F/L(t)$ if and only if $t\in q((L')^m)$, 
  we have shown the desired property.

Now consider a cyclic extension $M/L$ of prime power degree $m$
and let $f\in L[Y_1, \ldots, Y_m]$ be the norm form of $M/L$ (with respect to an arbitrary choice of $L$-basis). % for which $\operatorname{N}_{M/L}$ is not surjective.
Let $\tau$ be a generator of $\Gal(M/L) \cong \Gal(M(t)/L(t))$ and let $A$ be the cyclic $L(t)$-algebra $(M(t), \tau, t)$,
defined as a degree $m$ central division algebra over $L(t)$, generated over $L(t)$ by $M(t)$ and an element $\alpha$ subject to the relations $\alpha^m = t$ and $y\alpha = \alpha \tau(y)$ for all $y \in M$. 
See for example \cite[Sections 5.8-5.10]{Albert} for background on cyclic algebras.
For an intermediate extension $L'$ of $F/L(t)$, we have that $A_{L'}$ is split if and only if $t$ is a norm of $ML'/L'$ \cite[Theorem 5.14]{Albert}, i.e.~if and only if $L' \models \sigma_f(t)$. On the other hand, if $L'$ has transcendence degree $m-2$ over $L(t)$, then by Example \ref{vb:edCyclicDivisionAlgebra} and Lemma \ref{lem:incompressibleFunctionField} $A_{L'}$ is not split.
This shows that the desired property holds.
\end{proof}

%\begin{stel}\label{stel:sums_of_squares}
%  Let $K$ be an infinite field.
%  There exists a field $K \prec_\exists L$
%  such that for every anisotropic quadratic form $q$ in $m \in \Natwithoutzero$ variables defined over $L$,
%  $\qr_L(\sigma_q) = m$.
%\end{stel}
%
%\begin{proof}
%  We construct a chain of fields $K_1 \prec_\exists K_2 \prec_\exists \ldots$ in the following way.
%  Set $K_1 = K$.
%  Given $K_{i-1}$, let $K_{i-1}\prec_\exists K_{i}$ be such that $\qr_{\Th_{\Lar(K_{i-1})}(K_i)}(\varphi) = \qr_{\subTec{K_{i-1}}}(\varphi)$ for every $\exists$-$\Lar(K_{i-1})$-formula $\varphi$. Such a field is afforded by Proposition \ref{prop:lowerBoundsCompositum}.
%  Let $L$ be the direct limit of $(K_i)_{i \in \Natwithoutzero}$.
%  
%  Let $q$ be an anisotropic quadratic form in $m \in \Natwithoutzero$ variables defined over $L$.
%  Let $\psi$ be an $\exists_{m-1}$-$\Lar(L)$-formula; we shall show that $\psi$ is not equivalent in $L$ to $\sigma_q$.
%  Both $q$ and $\psi$ are actually defined over $K_i$ for some $i \in \Natwithoutzero$. 
%  By Lemma \ref{lem:weakLowerBoundSumOfSquares} and by construction of $K_{i+1}$, $\qr_{\Th_{\Lar(K_{i})}(K_{i+1})}(\sigma_q) = m$.
%  In particular, $\psi$ is not equivalent to $\sigma_q$ in $K_{i+1}$.
%  Any witness for the non-equivalence of these formulas in $K_{i+1}$ also witnesses that they are not equivalent in $L$ since $K_{i+1} \prec_\exists L$.
%\end{proof}
\begin{opm}
In Theorem \ref{stel:erk_quad_cyclic_norm}\eqref{it:cyclic}, we cannot allow arbitrary cyclic extensions.
For example, if $M/L$ is a cyclic extension of degree $6$, then 
${\rm N}_{M/L}(M^\times)={\rm N}_{M_2/L}(M_2^\times)\cap {\rm N}_{M_3/L}(M_3^\times)$,
where $M_i$ ($i=2,3$) is the unique subextension of $M/L$ of degree $i$,
for if $x={\rm N}_{M_i/L}(x_i)$ with $x_i\in M_i^\times$ for $i=2,3$, then $x={\rm N}_{M/L}(\frac{x_2}{x_3})$.
%an element $x \in L$ lies in the norm group of $M$ if and only if it lies in the norm group of both its degree $2$ subextension and its degree $3$ subextension; this property can be expressed with only $4$ quantifiers.
Thus, since trivially $\erk_L({\rm N}_{F/L}(F^\times))\leq[F:L]$ for every finite extension $F/L$,
Remark \ref{qrtrivial} shows that $\erk_L({\rm N}_{M/L}(M^\times))\leq5$.
\end{opm}

%Applying this result to $K = \qq$, observing that for all $m \in \Natwithoutzero$ the form $q(Y_1, \ldots, Y_m) = \sum_{i=1}^m Y_i^2$ is anisotropic over $\qq$ and therefore over every $\qq \prec_\exists L$, we obtain the following.
\begin{gev}\label{gev:sums_of_squares}
There exists a field $L$ of characteristic $0$ such that $\erk_L(\sum_{i=1}^m L\pow{2}) = m$ for all $m \in \Natwithoutzero$.
In particular, $\erk_L((L\pow{2}+L\pow{2})^n)=n+1$ for every $n\in\Natwithoutzero$.
\end{gev}
\begin{proof}
Applying Theorem \ref{stel:erk_quad_cyclic_norm} to $K = \qq$, observing that for all $m \in \Natwithoutzero$ the quadratic form $q(Y_1, \ldots, Y_m) = \sum_{i=1}^m Y_i^2$ is anisotropic over $\qq$ and therefore over every $\qq \prec_\exists L$, the first part is immediate.

For the second part, the upper bound $\erk_L((L\pow{2}+L\pow{2})^n) \leq n+1$ follows from the observation that $\erk_L(L\pow{2}+L\pow{2}) \leq 2$ and by applying Corollary \ref{cor:equivalences}(6) inductively.
On the other hand, for each $x \in L$ we have that $x \in \sum_{i=1}^{2n} L\pow{2}$ if and only if there exist $y_1, \ldots, y_{n-1} \in L$ such that $(y_1, \ldots, y_{n-1}, x - \sum_{i=1}^{n-1}y_i) \in (L\pow{2}+L\pow{2})^n$.
We infer that $2n = \erk_L(\sum_{i=1}^{2n} L\pow{2}) \leq n - 1 + \erk_L((L\pow{2}+L\pow{2})^n)$, from which the other inequality follows.
\end{proof}
\begin{vb}
For $K = \qq$ we have $\sum_{i=1}^m \qq\pow{2} = \qq_{\geq 0}$ for $m\geq 4$ by a theorem due to Euler (sometimes incorrectly attributed to Lagrange).
On the other hand, $\sum_{i=1}^3 \qq\pow{2}$ is a proper subset of $\qq_{\geq 0}$: one easily verifies that it does contain not any integers of the form $8m-1$ for $m \in \zz$.

Nevertheless we have $\erk_{\qq}(\qq_{\geq 0}) \leq 3$.
Indeed, as a consequence of Legendre's Three-Square Theorem, we observe that $\qq_{\geq 0} = (\sum_{i=1}^3 \qq\pow{2}) \cup 2(\sum_{i=1}^3 \qq\pow{2})$, from which the $\exists_3$-$\Lar$-definability of $\qq_{\geq 0}$ follows.
This illustrates the subtleties which arise when computing the existential rank of subsets of fields.
\end{vb}

\section{Global fields}\label{sect:global}

\noindent
In this final section we make a few observations regarding existential rank
in Hilbertian fields
and then focus on $\mathbb{Q}$ and $\mathbb{F}_p(t)$.

Let $K$ be a field  of characteristic exponent $p$.

\begin{defi}
A dominant morphism $f\colon W\rightarrow V$ of integral $K$-varieties is
of {\em finite degree}, {\em separable} or {\em purely inseparable}
if the corresponding extension $K(W)/K(V)$ of function fields has this property.
If $V$ is an integral $K$-variety,
a subset $D\subseteq V(K)$ is {\em [separably] thin} (in $V$)
if $D$ is contained in a finite union of sets of the form
$W(K)$ for $W\subseteq V$ a proper closed subvariety, or
$f(W(K))$ for 
$W$ a geometrically integral $K$-variety and
$f\colon W\rightarrow V$ a [separable] dominant morphism of finite degree ${\rm deg}(f)>1$.
A [separably] thin subset of $K^n$ is a [separably] thin subset of $\mathbb{A}_K^n(K)$.
The field $K$ is {\em Hilbertian} if $\mathbb{A}_K^n(K)$ is not separably thin in $\mathbb{A}^n_K$ for any $n$.
\end{defi}

\begin{opm}
The equivalence of our definition of {\em Hilbertian} with the definition in \cite{FJ}
(which differs from the one in the first edition of the same book)
is explained in \cite[Chapter 13.5]{FJ}.
In the literature, the term `thin' is used mostly in characteristic zero, 
where our notions of {\em thin} and {\em separably thin} coincide,
while in positive characteristic the precise meaning varies:
While \cite{FJ} call `thin' what we call {\em separably thin},
the definition of `thin'  in \cite{CZ} coincides with what we call {\em thin}.
\end{opm}

\begin{lem}\label{lem:thin}
If $K$ is Hilbertian and imperfect, then $\mathbb{A}_K^n(K)$ is not thin in $\mathbb{A}^n_K$ for any $n$.
\end{lem}

\begin{proof}
This
follows from Uchida's result \cite[Proposition 12.4.3]{FJ} that in a Hilbertian field which is imperfect,
every (not necessarily separable) Hilbert set is nonempty.
\end{proof}

 We start with a refinement of Lemma \ref{lem:formulasEquivVarieties} 
 in the case of essential fibre dimension zero. 
 \begin{prop}\label{lem:E1}\label{prop:cofinite_or_thin}
 Let $D\subseteq K^n$ be a diophantine set with $\efd_K(D)=0$.
 Then 
 \begin{equation}\label{eq:efd0thin}
  D = D_0 \cup \bigcup_{i=1}^rf_i(W_i(K))
 \end{equation} 
 with $D_0\subseteq K^n$ separably thin,
 $r\in\Natwithzero$,
 and, for each $i$,
 $W_i$ a geometrically integral $K$-variety and $f_i\colon W_i\rightarrow\mathbb{A}^n_K$ a purely inseparable dominant morphism.
\end{prop}
 
\begin{proof}
By Corollary \ref{cor:efdCharacterisationDiag}, 
$D=f(W(K))$ for a $K$-variety $W$ and a $K$-morphism $f\colon W\rightarrow\mathbb{A}_K^n$ that has all fibres of dimension at most $\efd_K(D)=0$.
Since a union of finitely many separably thin sets is separably thin, we can assume that $W$ is integral.
Further we can assume that $W$ is geometrically integral, since otherwise the Zariski closure of $W(K)$ has smaller dimension than $W$ (see \cite[Remark 2.3.27]{Poonen_RationalPoints}), and therefore $f(W(K))$ is separably thin.
Similarly, we can assume that $f$ is dominant, since otherwise $f(W(K))$ is contained in a subvariety of smaller dimension, hence separably thin.
Finally, if $f$ is not purely inseparable, then up to replacing $W$ by an open subvariety,
$f$ factors through a separable morphism $f'\colon W'\rightarrow\mathbb{A}^n_K$
of degree ${\rm deg}(f')>1$, hence $D\subseteq f'(W'(K))$ is separably thin.
\end{proof} 

\begin{gev}\label{cor:thinorcofinite}
Let $D\subseteq K^n$ be a diophantine set with $\efd_K(D)=0$.
Then $D$ is either thin in $K^n$ or contains $U(K)$ for a nonempty Zariski-open set $U\subseteq\mathbb{A}^n_K$.
In particular, if $n=1$, then $D$ is either thin in $K$ or cofinite in $K$.
\end{gev}

\begin{proof}
Write $D$ as in (\ref{eq:efd0thin}).
If ${\rm deg}(f_i)=1$ for some $i$, then $f_i$ is a birational morphism, hence $D$ contains $U(K)$ for a nonempty Zariski-open $U\subseteq\mathbb{A}_K^n$.
Otherwise, $D$ is thin.
\end{proof}

\begin{gev}\label{cor:thinperfect}
Let $K$ be perfect and let $D\subseteq K^n$ be a diophantine set with $\efd_K(D)=0$.
  Then $D$ is either separably thin in $K^n$
  or contains $U(K)$ for a nonempty Zariski-open set $U\subseteq\mathbb{A}^n_K$.
  In particular, if $n=1$, then $D$ is either separably thin in $K$ or cofinite in $K$.
\end{gev}

\begin{proof}
Write $D$ as in (\ref{eq:efd0thin}).
If $r\geq1$, then since $K$ is perfect, 
the purely inseparable extension $K(W_1)$ of $K(\mathbb{A}_K^n)$ is contained in $K(\mathbb{A}_K^n)^{p^{-e}}\cong_K K(\mathbb{A}_K^n)$ for some $e$,
hence there exists a dominant rational map $f_1'\colon \mathbb{A}^n_K\dashrightarrow W_1$ such that the rational map
$f_1\circ f_1'$ is a power $\phi^e$ of the Frobenius morphism $\phi\colon \mathbb{A}^n_K\rightarrow\mathbb{A}^n_K$,
 i.e.~the $p^e$-th power map.
So if $U\subseteq\mathbb{A}^n_K$ is a nonempty open subvariety on which $f_1'$ is regular,
we get that $f_1(W_1(K))$ contains $\phi^e(U(K))=U^{p^e}(K)$,
where $U^{p^e}\subseteq\mathbb{A}^n_K$ is again nonempty and Zariski-open.
\end{proof}

\begin{gev}\label{cor:thinOrCofinite}
  Let $K_0$ be a perfect field and $K/K_0$ a finitely generated extension of transcendence degree one.
  Let $D\subseteq K$ be a diophantine set with $\efd_K(D)=0$.
  Then $D=D_0\cup\bigcup_{i=1}^rD_i$
  where $D_0$ is separably thin
  and $D_i$ is cofinite in $f_i(K\pow{p^{e_i}})$ for some 
  $e_i\geq0$ and a Möbius transformation
  $f_i \colon \mathbb{P}_K^1 \to \mathbb{P}_K^1$.
\end{gev}
\begin{proof} 
By Proposition \ref{prop:cofinite_or_thin} it suffices to show that 
if $C$ is a geometrically integral $K$-curve and $f\colon C\rightarrow\mathbb{A}^1_K$ a purely inseparable morphism, then
   $f(C(K))$ is either finite or 
   cofinite in the image of $K\pow{p^{e}}$ for some $e$ under a Möbius transformation.
  We may remove the finitely many non-regular points of $C$ without affecting the claim.
   Since the function field of the base change $C_{\overline K}$ is a purely inseparable extension of $\overline{K}(\mathbb{A}_K^1)$ and therefore isomorphic to $\overline{K}(\mathbb{A}_K^1)$, $C_{\overline K}$ has genus zero.
 We distinguish three cases, in which we prove the claim:

  If $C$ has non-zero genus, 
  then $C(K)$ is finite by Proposition \ref{thm:genusChangingFinite}.

  If $C$ has genus zero but is not rational, then $C(K)$ is empty (see for instance \cite[Proposition 7.4.1]{Liu}).

  If $C$ is rational, then $C$ is isomorphic to an open subscheme of $\mathbb{P}_K^1$, and the morphism $f \colon C \to \mathbb{A}_K^1$ extends to a purely inseparable morphism $\mathbb{P}_K^1 \to \mathbb{P}_K^1$.
  By \cite[Proposition 7.4.21]{Liu}, this must be a power of the Frobenius morphism, 
  composed with an automorphism of $\mathbb{P}_K^1$, i.e.~a Möbius transformation.
  Since $\mathbb{P}_K^1 \setminus C$ is finite, this proves that $f(C(K))$ contains a cofinite subset of the image of $K\pow{p^e}$ under a Möbius transformation.
\end{proof}

\begin{opm}
  Let $K$ be a number field.
  Then Corollary \ref{cor:thinorcofinite}, together with \cite[Theorem~13.3.5(c)]{FJ}, shows that the ring of integers $\mathcal{O}_K$ is not  diophantine
  with essential fibre dimension zero, i.e.~by Corollary \ref{cor:equivalences} not definable by an $\exists_1$-formula.
  More generally, any $\exists_1$-definable $D\subseteq K$ which is not cofinite does not contain a translate of an ideal of $\mathcal{O}_K$.
\end{opm}

The following lemma generalizes \cite[Exercise 2 on p.~20]{Serre}, which treats the separable case:

\begin{lem}\label{cor:efd_or_cofinite}
Let $f\colon V'\rightarrow V$ be a dominant morphism of geometrically integral $K$-varieties with geometrically integral generic fibre,
and let $D'\subseteq V'(K)$.
If $f(D')$ is [separably] thin in $V$, then $D'$ is [separably] thin in $V'$.
\end{lem}

\begin{proof}
Suppose that $f(D')\subseteq W(K)\cup\bigcup_{i=1}^r{f_i(W_i(K))}$ with $W\subseteq V$ a proper closed subvariety,
$W_i$ a geometrically integral $K$-variety and $f_i\colon W_i\rightarrow V$ a dominant [separable] morphism of finite degree ${\rm deg}(f_i)>1$.
By generic flatness, assume without loss of generality that each $f_i$ is flat (possibly enlarging $W$).
Then $W_i':=V'\times_{V}W_i$ 
 is geometrically integral, and $f_i'\colon W_i'\rightarrow V'$ 
is [separable] dominant of finite degree ${\rm deg}(f_i')={\rm deg}(f_i)$.
Indeed, $(W_i')_{\overline{K}}\rightarrow V'_{\overline{K}}$ is flat and its generic fibre is integral,
because $\overline{K}(V')$ and $\overline{K}(W_i)$ are linearly disjoint over $\overline{K}(V)$ by the assumption on the generic fibre of $f$,
hence $(W_i')_{\overline{K}}$ is integral, see \cite[Proposition 4.3.8]{Liu}.
Moreover, $W':=f^{-1}(W)$ is a proper closed subvariety of $V'$.
Thus $D'\subseteq W'(K)\cup\bigcup_{i=1}^r{f_i'(W_i'(K))}$ is [separably] thin.
\end{proof}

\begin{stel}\label{thm:norms}
  Let $K$ be Hilbertian.
  If there exists a finite separable extension $L/K$ with $\operatorname{N}_{L/K}(L^\times) \subsetneq K^\times$, 
  then $\erk_K(\operatorname{N}_{L/K}(L^\times))\geq 2$,
  and in particular, $\erk(K) \geq 2$.
\end{stel}

\begin{proof}
  Let $n = [L \colon K]$, and let $f \in K[Y_1, \dotsc, Y_n]$ be the norm form of $L/K$ with respect to some basis.
  This is a homogeneous polynomial of degree $n$,
  and $f-T$ is absolutely irreducible, cf.~\cite[Section 11 Exercise 6]{FJ},
  in other words, the morphism $f\colon \mathbb{A}^n_K\rightarrow\mathbb{A}^1_K$ has geometrically integral generic fibre.
  Since $K$ is Hilbertian, $D':=K^n\setminus\{0\}$ is not separably thin in $\mathbb{A}^n_K$, and not thin if $K$ is imperfect (Lemma \ref{lem:thin}),
  so Lemma \ref{cor:efd_or_cofinite} implies that $D := f(D')$ is not separably thin in $K$, and not thin if $K$ is imperfect.
  Since $D = \operatorname{N}_{L/K}(L^\times)$ is a proper subgroup of $K^\times$, it is also not cofinite in $K$.
  Therefore Corollaries \ref{cor:thinorcofinite} and \ref{cor:thinperfect} imply
  $\efd_K(D)\geq1$,
  hence
  $\erk_K(D)\geq2$ by Propositions \ref{edleqqr} and \ref{prop:qrvsqrp}.
\end{proof}

 \begin{gev}\label{cor:real}
 If $K$ is Hilbertian and real, then 
 $\erk_K(K\pow{2} + K\pow{2})=2$.
 \end{gev}
 
\begin{proof}
It is clear that $2$ is an upper bound, and the other inequality follows from Theorem~\ref{thm:norms} with $L=K(\sqrt{-1})$.
\end{proof}

\begin{gev}\label{cor:global}
If $K$ is a global field, then $\Erk(K)\geq 2$.
\end{gev}

\begin{proof}
  Take a separable quadratic extension $L/K$ and a discrete valuation $v_{\mathfrak p}$ on $K$ (i.e.\ the valuation corresponding to a non-archimedean place) which is inert in $L$, i.e.~there is a unique valuation $v_{\mathfrak P}$ on $L$ above $v_{\mathfrak p}$ and $v_{\mathfrak P}/v_{\mathfrak p}$ is unramified.
  Then for every $x \in L^\times$ with conjugate $\overline x$ over $K$, the valuation
  \[ v_{\mathfrak p}(\operatorname{N}_{L/K}(x)) = v_{\mathfrak p}(x \overline x) = v_{\mathfrak P}(x) + v_{\mathfrak P}(\overline x) = 2 v_{\mathfrak P}(x) \]
  is even, thus $\operatorname{N}_{L/K}(L^\times)$ is a proper subgroup of $K^\times$.
  Since global fields are Hilbertian \cite[Theorem 13.4.2]{FJ},
Theorem \ref{thm:norms} applies.
\end{proof}

\begin{opm}
  There exist perfect Hilbertian fields which are PAC \cite[Chapter 27]{FJ}, so Corollary \ref{cor:erkPAC} shows that some condition beyond Hilbertian is necessary in Theorem \ref{thm:norms} and Corollary \ref{cor:real}.
\end{opm}

\begin{opm}\label{opm:dittmann-leijnse}
  In connection with Corollary~\ref{cor:global},
  we point out another sense in which sets of existential rank at most $1$ in a global field $K$ are rather special
  (beyond Corollary~\ref{cor:thinorcofinite}):
  by \cite[Corollary~3.2]{DittmannLeijnse} (after \cite{Dittmann}),
  for every $D \subseteq K^n$ with $\erk_K(D) \leq 1$ the complement $K^n \setminus D$ is also diophantine.
  In \cite{DittmannLeijnse}, this property is studied as a condition on a general field $K$.
\end{opm}

We now turn more specifically to the global fields $\mathbb{Q}$ and $\mathbb{F}_p(t)$.

 \begin{opm}\label{rem:Cohen}
 In $\mathbb{Q}$, thin sets 
 in particular have density zero when counted by height.
 Moreover, by a result of Cohen, if $S\subseteq\mathbb{Q}^n$ is thin, then the number of $(x_1,\dots,x_n)\in S\cap\mathbb{Z}^n$ with
 $|x_i|<N$ for all $i$ is $O(N^{n-1/2}\log N)$, see \cite[Theorem 3.4.4]{Serre}.
 For the analogous result for $\mathbb{F}_p(t)$ see \cite{BarysorokerEntin}.
 \end{opm}

\begin{vb}
Let $S\subseteq\mathbb{N}$ be the set of prime numbers that are $1$ modulo $4$.
Comparing the result of Cohen
with Dirichlet's density theorem and the prime number theorem,
we get from Corollary \ref{cor:thinorcofinite} that
every definable set $S\subseteq D\subseteq\mathbb{Q}$ 
that
is not cofinite in $\mathbb{Q}$
has $\erk_{\mathbb{Q}}(D)\geq 2$ (possibly $\infty$).
This includes for example the set $D = \mathbb{Q} \setminus \mathbb{Q}\pow{2}$ of non-squares, which is known to be existentially definable \cite{PoonenNonsquares}.
\end{vb}

\begin{ques}
What is $\Erk(\mathbb{Q})$? What is $\Erk(\mathbb{F}_p(t))$?
\end{ques}

\begin{opm}
It is known that the $\exists_1$-theory of $\mathbb{Q}$ is decidable
(this follows from the existence of a splitting algorithm, cf.~\cite[Definition 19.1.2]{FJ}),
but from Corollary \ref{cor:global} we see that $\Erk(\mathbb{Q})\geq2$.
It can be expected that also the $\exists_2$-theory of $\mathbb{Q}$ is decidable (see for example \cite{Poonen_Brauer-Manin,Alpoege}).
However, even if $\Erk(\mathbb{Q})=2$ this would not immediately imply that ${\rm Th}_\exists(\mathbb{Q})$ is decidable,
as the universal-existential fragment ${\rm Th}_{\forall\exists}(\mathbb{Q})$ is not recursively enumerable,
as follows from Koenigsmann's universal definition \cite{Koenigsmann} of $\mathbb{Z}$ in $\mathbb{Q}$,
cf.~Remark \ref{rem:effective}.
\end{opm}

In the following we view a ring $R$ (always assumed commutative and with $1$) as an $\mathcal{L}_{\rm ring}$-structure
and apply the definitions from Section \ref{sect:erk}.

\begin{lem}\label{lem:pol_injection}
Let $R$ be a ring.
If there exists $f\in R[X,Y]$ such that $f\colon R\times R\rightarrow R$ is injective,
then $\Erk(R)=\Erkone(R)$. 
\end{lem}

\begin{proof}
Write $f_n$ for the injection $R^n\rightarrow R$ obtained by iteratively defining $f_n(a_1,\dots,a_n)=f(f_{n-1}(a_1,\dots,a_{n-1}),a_n)$.
Let $\varphi(X_1,\dots,X_n)$ be an existential formula.
The diophantine set
$\{f_n(a_1,\dots,a_n):(a_1,\dots,a_n)\in\varphi(R)\}$ is defined by an 
existential formula $\varphi'$ with at most $d=\Erkone(R)$ quantifiers,
hence $\varphi$ is equivalent in $R$ to the $\exists_d$-formula
$\varphi'(f_n(X_1,\dots,X_n))$.
Thus, $\erk(R)\leq\Erkone(R)$,
and $\Erkone(R)\leq\erk(R)$ by definition.
\end{proof}

\begin{lem}\label{lem:polynomial_injection_on_R}
Let $R=\mathbb{Z}$ or let $R$ be an integral domain such that $\Frac(R)$ is imperfect.
There exists $f\in R[X,Y]$ such that $f\colon R\times R\rightarrow R$ is injective.
Furthermore, if $R = \zz$, we may assume that $f(R\times R) \subseteq \nat$.
\end{lem}

\begin{proof}
Consider first the case $R = \zz$.
The map $\zz \to \nat : X \mapsto 2X^2 - X$ is easily seen to be injective.
Furthermore, also the map $\nat \times \nat \to \nat : (X, Y) \mapsto (X+Y)(X+Y+1) + 2Y$ is known to be injective ($\frac{f}{2}$ is the Cantor pairing function). 
We may thus take the polynomial
$$ f = (2X^2 - X + 2Y^2 - Y)(2X^2 - X + 2Y^2 - Y + 1) + 4Y^2 - 2Y.$$
If $R$ is an integral domain of characteristic $p > 0$ and $t \in R \setminus \Frac(R)\pow{p}$, one may take
$f=X^p+tY^p$, cf.~\cite[Remark 1.6]{Poonen}.
\end{proof}

%\begin{opm}\label{rem:pairing}
It is strongly expected that also for $R=\mathbb{Q}$ 
a polynomial $f$ as in Lemma \ref{lem:pol_injection} exists,
see \cite{Cornelissen,Poonen,Pasten2}.
%\end{opm}

\begin{lem}\label{lem:universal}
Let $R=\mathbb{Z}$ or $R=\mathbb{F}_p[t]$.
There exists a {\em universal diophantine set}, i.e.~an existential $\mathcal{L}_{\rm ring}(R)$-formula in 2 variables
$$
 \theta(X, Z) = \exists Y_1,\dots,Y_N\vartheta(X,\underline{Y},Z)
$$ 
for some $N \in \Natwithoutzero$ and $\vartheta$ quantifier-free
such that for every existential $\mathcal{L}_{\rm ring}(R)$-formula $\varphi(X)$ in one variable
there exists $c\in R$ with $\varphi(R)=\theta(R,c)$.
\end{lem}

\begin{proof}
We fix any recursive presentation of $R$,
i.e.~a bijection $\rho\colon R \rightarrow \Natwithzero$  such that the images of the graphs of addition and multiplication are recursive.
%so that we may speak of recursive and recursively enumerable subsets of $R$.
It is proven in \cite{Matiyasevich} respectively \cite{Demeyer}
that every recursively enumerable subset of $R$ is diophantine in $R$.
Since the recursively enumerable subsets of $R$ can be enumerated as $M_0,M_1,\dots$ such that the relation 
$x\in M_{\rho(y)}$ is again recursively enumerable, cf.~the introduction of \cite{Jones} in the case $R=\mathbb{Z}$, this gives the existence of a universal diophantine set in $R$.
In the case $R=\mathbb{Z}$,
an explicit universal diophantine set is constructed in \cite{Jones}.
\end{proof}

\begin{prop}\label{prop:ZinQ1}
Let $R=\mathbb{Z}$ or $R=\mathbb{F}_p[t]$, and $K={\rm Frac}(R)$.
If $R$ is diophantine in $K$, then $\Erkone(K)<\infty$.
\end{prop}

\begin{proof}
Let $\theta(X,Z) = \exists Y_1,\dots,Y_N\vartheta(X,\underline{Y},Z)$ be the universal diophantine set from Lemma \ref{lem:universal},
and suppose that $\rho(X)$ is an $\exists_M$-formula defining $R$ in $K$.
Let $\varphi(X)$ be an existential $\Lar(K)$-formula.
We claim that $\qr_K(\varphi)\leq N+1+(N+2)M$.
The set
$$
 D=\left\{f(a,b):a,b\in R,b\neq0,\frac{a}{b}\in\varphi(K)\right\},
$$
where $f\in R[X,Y]$ 
is a polynomial as in Lemma \ref{lem:polynomial_injection_on_R},
is existentially definable in $R$
(where we use the usual quantifier-free interpretation of $K=\Frac(R)$ in $R$ to express $\frac{a}{b} \in \varphi(K)$).
%is again diophantine in $K$.
%By the usual quantifier-free interpretation of $K=\Frac(R)$ in $R$,
%we obtain that $D$ is existentially definable in $R$,
Thus there exists $c\in R$ with $\theta(R,c)=D$. Let $\varphi'(X)$ be the formula
$$
 \exists Y_1,\dots,Y_{N+1}\left(\bigwedge_{i=1}^{N+1}\rho(Y_i)\wedge \rho(XY_{N+1}) \wedge Y_{N+1}\neq0\wedge\vartheta(f(XY_{N+1},Y_{N+1}),Y_1,\dots,Y_N,c)\right).
$$
Then $\varphi'$ is an existential formula with
$N+1+(N+2)M$ quantifiers,
and $\varphi'(K)=\{\frac{a}{b}:f(a,b)\in D\}=\varphi(K)$.
\end{proof}

\begin{opm}\label{rem:ErkZ}
Let $R = \zz$ or $R = \ff_p[t]$. In Lemma \ref{lem:universal}, the quantity $N$ can be chosen equal to $\Erk(R)$, which is the same as $\Erkone(R)$ by virtue of Lemmas \ref{lem:pol_injection} and \ref{lem:polynomial_injection_on_R}.
The proof of Proposition \ref{prop:ZinQ1} can then be refined by applying Corollary \ref{cor:finGenOverPerf} $N+1$ times to show that for $K = \Frac(R)$ one has $\Erkone(K) \leq (\Erkone(R)+2)\cdot\erk_K(R)$. 

It follows from \cite[Theorem 1.1(ii)]{Sun21} that $\Erkone(\zz) \leq 10$
(the quoted result talks about subsets of $\Natwithzero$, 
but applying an injective polynomial function $f:\mathbb{Z}\rightarrow\Natwithzero$,
like $f=2X^2-X$, one obtains the same result for subsets of $\mathbb{Z}$), yielding $\Erkone(\qq) \leq 12\cdot\erk_\qq(\zz)$.

In a similar vein, Matijasevič's 9 Unknowns Theorem as proven in \cite{Jones} implies that $\Erkone(\Natwithzero) \leq 9$ (using the general Definition \ref{def:erk_K} with the language $\lbrace +, \cdot, 0, 1 \rbrace$), and one can use the same arguments as above to show that $\Erkone(\qq) \leq 11\cdot \erk_\qq(\Natwithzero)$.
\end{opm}

\begin{gev}\label{cor:erk_Z_in_Q}
If $\Erk(\mathbb{Q})=\infty$
and there exists $f\in\mathbb{Z}[X,Y]$ such that $f\colon \mathbb{Q}\times\mathbb{Q}\rightarrow\mathbb{Q}$ is injective, 
then $\mathbb{Z}$ is not diophantine in $\mathbb{Q}$.
If $\Erk(\mathbb{F}_p(t))=\infty$, then $\mathbb{F}_p[t]$ is not diophantine in $\mathbb{F}_p(t)$.
\end{gev}

\begin{proof}
This follows by combining Proposition \ref{prop:ZinQ1} with Lemma \ref{lem:pol_injection},
where in the case of $\mathbb{F}_p(t)$ we also use Lemma \ref{lem:polynomial_injection_on_R}.
\end{proof}

This leads to the question whether one can prove  
that these existential ranks are in fact infinite.
One hope to prove that the existential rank of $\mathbb{F}_p(t)$ is infinite
and thus to resolve Question \ref{q:FqT} via Corollary \ref{cor:erk_Z_in_Q}
would be to use the same construction that works for its completion $\mathbb{F}_p((t))$:
The point there was that the set of $n$-tuples of $p$-th powers has 
$\Erk_{\mathbb{F}_p((t))}((\mathbb{F}_p((t))\pow p)^n)=n$.
However, Theorem \ref{thm:finGenOverPerf} implies that
$\Erk_{\mathbb{F}_p(t)}((\mathbb{F}_p(t)\pow p)^n)=1$.

As for the existential rank of $\mathbb{Q}$, Theorem \ref{stel:erk_quad_cyclic_norm} and Corollary \ref{gev:sums_of_squares} suggest diophantine sets which are candidates to have high existential rank:
%for example, $\Erk(\mathbb{Q}\pow 2+\mathbb{Q}\pow 2)=2$,
%and Corollary \ref{cor:finGenOverPerf} gives only the upper bound
%$\Erk_\mathbb{Q}((\mathbb{Q}\pow 2+\mathbb{Q}\pow 2)^n)\leq n+1$,
%so one might hope that this existential rank tends to infinity with $n$,
%as it does in other fields of characteristic zero (Corollary \ref{gev:sums_of_squares}):

\begin{ques}
Is $\Erk_\mathbb{Q}((\mathbb{Q}\pow 2+\mathbb{Q}\pow 2)^n)=n+1$ for every $n\in\Natwithoutzero$?
\end{ques}

%In the spirit of Theorem \ref{thm:norms} one might also ask:

\begin{ques}
Is $\Erk_\mathbb{Q}(\operatorname{N}_{L/\mathbb{Q}}(L^\times))=m$
for every cyclic Galois extension $L/\mathbb{Q}$ of prime power degree $m$?
\end{ques}

Theorem \ref{thm:norms} gives 
the lower bounds
$\Erk_\mathbb{Q}((\mathbb{Q}\pow 2+\mathbb{Q}\pow 2)^n)\geq 2$
and $\Erk_\mathbb{Q}(\operatorname{N}_{L/\mathbb{Q}}(L^\times))\geq 2$
and hence answers these questions positively for $n=1$ respectively $m=2$. 

Another point to notice is that 
we know that $\Erk(\mathbb{Q})\geq2$, but
for all fields $K$ whose existential rank we were able to determine, we found $\Erk(K)\in\{0,1,\infty\}$.

\begin{ques}
Is there any field $K$ with $1<\Erk(K)<\infty$?
\end{ques}
\begin{opm}
Recall from Remark \ref{rem:ErkZ} that $\Erk(\zz) = \Erkone(\zz) \leq 10$.
On the other hand, $\Erk(\zz) > 1$: any $\exists_1$-$\Lar$-definable subset of $\zz$ is recursive, but there exist recursively enumerable subsets of $\zz$ (hence $\exists$-$\Lar$-definable by Matijasevič' theorem) which are not recursive.

It was shown recently that for any number field $K$, $\zz$ is $\exists$-$\Lar$-definable in the ring of integers $\mathcal{O}_K$ \cite{ABHS25,KP25}.
By a bi-interpretability argument similar to that given in the proof of Proposition \ref{prop:ZinQ1}, one can show that $1 < \Erk(\mathcal{O}_K) < \infty$ for any number field $K$.
\end{opm}

\end{document}